\documentclass[a4paper,11pt]{article}

\usepackage{amsmath}
\usepackage{amssymb}
\usepackage{amsthm}
\usepackage{amsfonts}
\usepackage{mathtools}
\usepackage{graphicx}

\usepackage{titlesec}
\usepackage{enumitem}
\usepackage{mathtools}
\usepackage{listings}
\usepackage{geometry}
\usepackage{bbm}
\usepackage[svgnames]{xcolor}
\usepackage{hyperref}
\usepackage{caption}
\usepackage{mathrsfs}
\usepackage{booktabs}
\usepackage{floatrow}
\newfloatcommand{capbtabbox}{table}[][\FBwidth]
\usepackage{blindtext}
\usepackage{colortbl}
\usepackage{tikz}
\usepackage{pgfplots}
\usepackage{subcaption}
\pgfplotsset{compat=1.16}
\newtheorem{theorem}{Theorem}[section]

\newtheorem{lemma}[theorem]{Lemma}
\newtheorem{remark}[theorem]{Remark}
\newtheorem{definition}[theorem]{Definition}
\newtheorem{conclusion}[theorem]{Conclusion}

\newcommand{\re}{\mbox{\rm Re}}

\newcommand{\D}{\mathcal{D}}

\newcommand{\R}{\mathbb{R}}

\renewcommand{\S}{\mathbb{S}}

\newcommand{\eps}{\varepsilon}

\definecolor{mycolor1}{rgb}{0.00000,0.44700,0.74100}%

\newcommand{\quotes}[1]{``#1''}

\newcommand{\dx}{\hspace{2pt}\mbox{d}x}
\newcommand{\dt}{\hspace{2pt}\mbox{d}t}

\newcommand{\ci}{\mathrm{i}} 

\newcommand{\sR}{\mbox{\rm \tiny R}}

\newcommand{\nablaR}{\nabla_{\hspace{-2pt}\sR,\eps}}
\newcommand{\VR}{V_{\hspace{-1pt}\mbox{\tiny eff}}} 

\newcommand{\Veps}{V_{\eps}}
\newcommand{\Omegaeps}{\Omega_{\eps}}

\newcommand{\tangentspace}[1]{T_{#1}\mathbb{S}}

\newcommand{\calI}{\mathcal{I}}
\newcommand{\calJ}{\mathcal{J}}
\newcommand{\calF}{\mathcal{F}}

\renewcommand{\Re}{\mathrm{Re}}          

\newcommand{\mucrit}{\mu_{\mbox{\tiny crit}}}

\allowdisplaybreaks

\begin{document}

	\begin{center}
		{\LARGE 
				{\LARGE 
				Numerical vortex resolution for the Gross--Pitaevskii equation in the rapid rotation Thomas--Fermi scaling
				\renewcommand{\thefootnote}{\fnsymbol{footnote}}\setcounter{footnote}{0}
				\hspace{-3pt}\footnote{P. Henning acknowledges the support by the Deutsche Forschungsgemeinschaft (DFG, German Research Foundation) through the project grant 551527112. A. Persson acknowledges support by the Swedish research council through the project grant 2022-03543.}}\\
	}
\end{center}

\begin{center}
{\large Patrick Henning,\hspace{-2pt}\footnote[1]{\label{affiliation}Department of Mathematics, Ruhr-University Bochum, DE-44801 Bochum, Germany.\\ Email: \href{mailto:patrick.henning@rub.de}{patrick.henning@rub.de}} \,\,
Anna Persson\footnote[2]{\label{affiliation-2}Department of Information Technology; Division of Scientific Computing, Uppsala University, SE-751 05 Uppsala, Sweden.\,\, Email: \href{mailto:apersson@it.uu.se}{apersson@it.uu.se}, \href{mailto:christos.pilichos@it.uu.se}{christos.pilichos@it.uu.se}}
\,\,and\,\, Christos Pilichos\textsuperscript{\ref{affiliation-2}} 
}\\[2em]
\end{center}

\begin{abstract}
In this paper we analyze finite element approximations of ground states of the Gross--Pitaevskii equation in the rapid rotation Thomas--Fermi scaling. In this regime, the healing length and vortex core size are of order $\eps \ll 1$, while the effective confinement potential may degenerate as the angular velocity approaches a critical value. In this setting, we analyze the $\eps$-dependence of the ground states and show that the local flatness of the energy landscape plays a decisive role for numerical resolution. More precisely, we establish mesh size conditions that guarantee the
existence of discrete ground states in finite element spaces which are quasi-best approximations of an exact ground state. In particular, we prove that the absolute $H^1$-error behaves asymptotically like $h/\eps^2$. However, to enter this asymptotic regime, the mesh size must satisfy a significantly stronger resolution condition than the natural requirement $h \lesssim \eps$. The additional restriction is governed by the first spectral gap of the Riemannian Hessian of the energy functional at the ground state, which measures the local flatness of the energy surface. With this, our results provide an explanation of the mesh resolution required to capture vortex structures in rapidly rotating Bose--Einstein condensates and highlight the interplay between vortex core size, spectral stability, and discretization accuracy.
\end{abstract}

\section{Introduction}
At temperatures close to absolute zero, dilute bosonic gases may undergo a phase transition to a Bose--Einstein condensate (BEC), a state of matter in which a macroscopic fraction of the particles occupies the same quantum state; see, e.g., \cite{Bos24,Ein24,PiS03}. This collective behavior leads to striking quantum effects on a macroscopic scale, among which superfluidity (the ability of the fluid to flow without dissipation) is one of the most prominent features \cite{MAH99}. When a condensate is set into rotation, superfluidity manifests itself through the formation of quantized vortices. As the rotation frequency increases, these vortices arrange in regular lattice patterns and may eventually fill large portions of the condensate.

A widely used mean-field description of rotating Bose--Einstein
condensates is provided by the Gross--Pitaevskii energy functional
\cite{Gro61,Pit61,Sei02}, whose ground states are obtained by
minimizing the energy under a normalization constraint. These ground
states describe stationary configurations of the condensate and capture
vortex structures generated by rotation. In the rapid-rotation
Thomas--Fermi regime, corresponding to strong interactions and angular
velocities that compete with trapping effects at leading order
\cite{Aftalion2006,CRY07}, the healing length and vortex core size
become small and are characterized by a parameter $\eps\ll1$. At the
same time, the effective confinement may weaken as the angular velocity
approaches a critical value, leading to locally flat energy landscapes
near the ground state. Hence, significant challenges arise from the
presence of multiple vortices, small length scales, and locally flat
energy landscapes.

The numerical computation of such ground states must address these challenges and involves two complementary components. On the one hand, iterative methods are required to compute constrained minimizers of the Gross--Pitaevskii energy, where gradient-based approaches and their variants are commonly employed; see, e.g., \cite{AHYY24,AHP20,101093imanumdraf046,APS22,ALT17,BaC13b,BaoDu04,Bao-et-al-2005,CLLZ24,CLYZ25,DaK10,DaP17,FT26,FengTangWangIMA,HeP20,PHMY242,JKM14} and \cite{HJ25} for a recent survey. On the other hand, the accuracy and reliability of the computed states crucially depend on the underlying spatial discretization, which must be sufficiently fine to resolve the small vortex cores and to capture the reduced stability of discrete ground states caused by the local flatness of the energy landscape. In this work, we focus on this latter aspect by analyzing the approximation properties of discrete minimizers in finite element spaces and by deriving $\eps$--sensitive resolution conditions that ensure that they are quasi-best approximations of exact ground states.

Let us briefly review existing results on the error analysis of spatial discretizations for the Gross--Pitaevskii energy. While a substantial body of work is devoted to finite element approximations of ground states, most contributions focus on the non-rotating setting. In this case, the problem reduces to a real-valued minimization problem with a unique positive ground state, which significantly simplifies the analysis. Early $H^1$-error estimates were derived by Zhou \cite{Zho04,Zho07}, although with a suboptimal contribution in the $L^2$-norm. Asymptotically optimal convergence rates in both $L^2$- and $H^1$-norms, as well as for the ground state energy and the chemical potential, were first established by Canc\`es et al. \cite{CCM10}. These results were later extended to more general finite element spaces in \cite{HMP14b,HeP21}. More recently, Hassan et al. \cite{HMW25} proved superconvergence results for the difference between discrete ground states and corresponding best approximations. Related error analyses have also been obtained for variants of the Gross--Pitaevskii model, such as equations with Hartree-type interactions \cite{CHZ11}, as well as for closely related electronic structure models including the Kohn--Sham and Hartree--Fock equations \cite{CDGHZ14,CGHY13,MadayTurinici2000}. Finally, mixed and nonconforming FEM discretizations were analyzed in \cite{GHLP25} and \cite{ZhangZhuChen2026} respectively, with a focus on lower bounds for the ground state energy.

In contrast to the non-rotating case, many of these models -- including rotating Bose--Einstein condensates as well as Kohn--Sham and Hartree--Fock theories -- exhibit non-uniqueness of ground states due to intrinsic symmetries of the energy, such as invariance under unitary transformations or global phase shifts, and possibly additional symmetry-breaking effects \cite{Bao-et-al-2005}. This lack of uniqueness leads to a degeneracy of the Riemannian Hessian at the ground state and requires a more delicate analysis. For the rotating Gross--Pitaevskii problem, this is further complicated by the transition to a complex-valued setting. Corresponding techniques have only recently been developed. In particular, \cite{HY25MathComp} establishes optimal-order a priori error estimates in $L^2$- and $H^1$-norms, as well as for the energy and chemical potential, for finite element discretizations of the rotating Gross--Pitaevskii problem using Lagrange elements of arbitrary order.

While the results in \cite{HY25MathComp} provide optimal-order a priori error estimates for finite element approximations of ground states of the rotating Gross--Pitaevskii energy, they are asymptotic in the mesh size and do not make the dependence on physical parameters explicit. The same applies to the other aforementioned works. In particular, these estimates do not capture the additional effects arising in the rapid rotation Thomas--Fermi scaling. Since the analysis in \cite{HY25MathComp} is not $\eps$--explicit, it cannot reflect the interplay between small vortex core sizes and the degenerating stability properties of the energy landscape as $\eps \to 0$.

In the present work, we provide an error analysis in the $\eps$--dependent regime that reveals additional effects. In contrast to the classical setting, the coercivity of the constrained Hessian deteriorates with $\eps$ and depends on the spectral gap structure at the ground state. As a consequence, the mesh resolution is no longer solely determined by the vortex core size, but must additionally compensate for the reduced stability induced by locally flat energy landscapes. This leads to significantly stronger $\eps$--dependent resolution conditions than suggested by standard resolution considerations. To be more precise, we show that the mesh size must satisfy
\begin{align*}
h \,\,\lesssim \,\,\eps^{(d+2)/2} \,\Bigl(1 - \tfrac{\lambda_1}{\lambda_2}\Bigr)\, \mucrit,
\end{align*}
where $\eps$ denotes the vortex core size, $d$ is the spatial dimension, 
$\bigl(1 - \tfrac{\lambda_1}{\lambda_2}\bigr)$ represents the (first) relative spectral gap of the Riemannian Hessian of the energy at a ground state (and thus measures the local flatness of the energy landscape), and $\mucrit$ is a computable constant that quantifies the proximity to the critical rotation frequency.

To establish the corresponding error analysis, it is not sufficient to merely track the constants in the arguments of \cite{HY25MathComp}; instead, new techniques are required. A first key step is to analyze the $\eps$--dependence of the underlying analytical setting. In particular, we derive $\eps$--explicit bounds for ground states in various norms and quantify the local flatness of the energy landscape through spectral properties of the Riemannian Hessian.

A central ingredient of the analysis is a Ritz projection associated with the constrained Hessian. However, due to the lack of local uniqueness, this operator admits a bounded inverse only on a subspace that does not contain the ground state $u$ itself. As a consequence, a direct construction of such a projection cannot be applied to $u$. We overcome this difficulty by first introducing a suitably modified projection operator with the required stability properties, which can then be shown to reduce to a standard Ritz projection.

Another fundamental difference to previous works, in particular \cite{HY25MathComp}, is the avoidance of compactness arguments. In the existing literature, convergence of discrete ground states is typically established along subsequences, allowing higher-order terms to be absorbed on sufficiently fine meshes. However, such arguments do not provide quantitative information on how fine the mesh must be, and in particular do not yield explicit $\eps$--dependent conditions. To address this issue, we employ a Pousin--Rappaz technique \cite{PoRa94}, which enables a fixed-point argument yielding explicit bounds on the distance between discrete and exact ground states in terms of $\eps$ and stability constants. A related approach has recently been developed in the context of the Ginzburg--Landau equation \cite{CFH25}.\\[0.5em]
\emph{Outline:} The remainder of this paper is organized as follows. In Section~\ref{sec:setting}, we introduce the Gross--Pitaevskii energy in the rapid rotation Thomas--Fermi scaling and establish $\eps$--explicit stability properties of ground states and the associated Hessian. In Section~\ref{sec:main-results}, we present our main results on $\eps$--dependent mesh resolution conditions and approximation properties of discrete minimizers. The corresponding error analysis is carried out in Section~\ref{sec:error-analysis}, where we develop the required projection techniques and establish the main estimates. Finally, in Section~\ref{sec:num-experiments}, we illustrate our theoretical findings by numerical experiments.

\section{The Gross--Pitaevskii energy in the rapid rotation scaling}
\label{sec:setting}

Before presenting the setting, let us fix the basic analytical notation. In the following,  $\mathcal{D} \subset \mathbb{R}^d$ is a Lipschitz domain with 
$d \in \{2,3\}$. For $2 \le q \le \infty$, we denote by $L^q(\mathcal{D})$ 
the usual Lebesgue space of complex-valued functions equipped with the norm 
$\|\cdot\|_{L^q(\mathcal{D})}$. The space $L^2(\mathcal{D})$ is endowed with the real inner product
\begin{align*}
(v,w)_{L^2(\mathcal{D})}
:=
\Re \int_{\mathcal{D}} v \, \overline{w} \, \mathrm{d}x,
\qquad v,w \in L^2(\mathcal{D}),
\end{align*}
where $\overline{w}$ denotes the complex conjugate of $w$. 
With this choice, $L^2(\mathcal{D})$ is regarded as a real Hilbert space 
consisting of complex-valued functions.

Furthermore, $H^1_0(\mathcal{D})$ denotes the Sobolev space of all 
$v \in L^2(\mathcal{D})$ whose weak first-order partial derivatives 
belong to $L^2(\mathcal{D})$ and whose trace vanishes on 
$\partial \mathcal{D}$. 
We equip $H^1_0(\mathcal{D})$ with the real inner product
\begin{align*}
(v,w)_{H^1_0(\mathcal{D})}
:=
(v,w)_{L^2(\mathcal{D})}
+
(\nabla v,\nabla w)_{L^2(\mathcal{D})},
\qquad v,w \in H^1_0(\mathcal{D}).
\end{align*}
The dual space of the real Hilbert space $H^1_0(\mathcal{D})$ is denoted by 
$H^{-1}(\mathcal{D})$.

Finally, we define the canonical identification operator 
$\mathcal{I} : L^2(\mathcal{D}) \to H^{-1}(\mathcal{D})$ by
\begin{align*}
\langle \mathcal{I} v , w \rangle
:=
(v,w)_{L^2(\mathcal{D})}
=
\Re \int_{\mathcal{D}} v \, \overline{w} \, \mathrm{d}x,
\qquad v \in L^2(\mathcal{D}),\; w \in H^1_0(\mathcal{D}).
\end{align*}

\subsection{The setting of the rapid rotation Thomas--Fermi scaling}
We consider the Gross--Pitaevskii equation with angular momentum rotation. In non-dimensional form, the corresponding energy functional is given by 
\begin{align*}
E(w)= \frac{1}{2}\int_{\mathcal{D}} |\nabla w|^2 + V_{\star}\, |w|^2  - \Omega_{\star}\, \bar{w}\, \mathcal{L}_{3}w + \frac{\beta_{\star}}{2} |w|^4 \dx.
\end{align*}
Here $V_{\star} \in L^\infty(\D)$ is a real-valued non-negative trapping potential, $\beta_{\star}>0$ denotes the strength of (repulsive) particle interactions, $\Omega_{\star} \in \mathbb{R}$ the angular velocity and $\mathcal{L}_3 = - \ci \left( x_1 \partial_{x_2} - x_2 \partial_{x_1} \right)$ the $x_3$-component of the angular momentum. 

To introduce the Thomas--Fermi regime, we let
$0<\eps\ll1$ denote a small dimensionless parameter called the Thomas--Fermi parameter. Following \cite[Section 7.2]{BaC13b}, we define $\eps:=\sqrt{\tfrac{\beta}{\beta_{\star}}}$ 
for some fixed constant $\beta\simeq1$ (e.g. $\beta:=1$). The limit $\eps\to0$ corresponds to the Thomas--Fermi regime of strong interactions. Physically, $\eps$ measures the ratio between the healing length and a characteristic macroscopic length scale of the condensate.

In addition, we consider a rapid-rotation regime, that is, the angular velocity scales as $\Omega_{\star} \simeq \eps^{-1}$ as discussed by Correggi et al. \cite{CRY07} (see also \cite{CoYn08}). To retain confinement in the rapid-rotation regime, we consider
trapping potentials satisfying $V_\star \simeq \eps^{-2}$. Otherwise centrifugal effects dominate the trapping potential, cf.~the discussion in Section~\ref{subsec:energy-inner-product} and assumption~\ref{A3}.

In summary, given $\beta \simeq 1$ and with
\begin{align*}
\eps:=\sqrt{\tfrac{\beta}{\beta_{\star}}}, \qquad 
\Omegaeps := \eps \, \Omega_{\star}, \qquad
\Veps := \eps^2 \, V_{\star},
\end{align*}
we obtain that the energy can be written as
\begin{align}
\label{scaled-energy}
E(w)= \frac{1}{2}\int_{\mathcal{D}} |\nabla w|^2 + \frac{\Veps}{\eps^2}\, |w|^2  - \frac{\Omegaeps}{\eps}\, \bar{w}\, \mathcal{L}_{3}w + \frac{\beta}{2\eps^2} |w|^4 \dx.
\end{align}
In this formulation, the factors $\eps^{-2}$ and $\eps^{-1}$ encode the rapid-rotation Thomas--Fermi scaling in which trapping, interaction and rotational effects remain visible in the limit $\eps\to0$. In the strict rapid-rotation regime one has
$$
|\Omegaeps| \simeq 1, \qquad \Veps \simeq 1.
$$
However, our analysis only requires the weaker assumptions that $\Omegaeps$ and $\Veps$ remain bounded independently of $\eps$, i.e.  
$$
|\Omegaeps| \lesssim 1,
\qquad
\|\Veps\|_{L^\infty(\mathcal D)} \lesssim 1,
$$
and therefore also covers slower rotation regimes, where the rotation has no leading order effect in the Thomas-Fermi regime. For example, an admissible choice is $\Omegaeps\,=\,\eps \,\Omega_{\star}$ and $\Veps\,=\, \eps^2 \,V_{\star}$ for $|\Omega_{\star}| = O(1)$ and $\|V_{\star}\|_{L^\infty(\mathcal D)} = O(1)$, which recovers the  semiclassical scaling considered in \cite[Section 7.2]{BaC13b}.

We fix the following assumptions
\begin{enumerate}[label={(A\arabic*)}]
\item \label{A1} $\D \subset \mathbb{R}^d$ is a bounded, convex domain for $d=2,3$ with polygonal boundary.
\item\label{A2} The potential $\Veps \in L^{\infty}(\mathcal{D})$ is real-valued with $\Veps\ge 0$ a.e. in $\mathcal{D}$ and $\|\Veps\|_{L^\infty(\mathcal D)} \lesssim 1$;\\
$\Omegaeps \in \R_{\ge 0}$ with $|\Omegaeps| \lesssim 1$ and $\beta \in \R_{>0}$ with $\beta \simeq 1$.
\end{enumerate}
In this setting and $E$ given by \eqref{scaled-energy}, we seek a ground state such that
\begin{align}
\label{definition-groundstate}
E(u) =\underset{v \in \mathbb{S}}{\inf}\hspace{2pt} E(v)
\qquad \mbox{where } \mathbb{S}:= \{  v \in H^1_0(\D) \mbox{ and } \| v \|_{L^2(\D)} =1 \}.
\end{align}
Due to \ref{A1}-\ref{A2} (in particular since $\D$ is bounded) it is easy to prove that under the above assumptions there exists a (possibly negative) constant $M \in \R$ such that
$$
E(v) \ge M \qquad \mbox{for all } v\in \S.
$$
With this, existence of a ground state follows with standard compactness arguments and the lower semi-continuity of $E$.

Furthermore, by the first-order condition for minimizers, i.e. the Euler--Lagrange equations for the constrained minimization problem, there exists a Lagrange multiplier $\lambda\in\R$ for each ground state $u \in \S$ such that
\begin{align}
\label{GPE-abstract}
\langle E^{\prime}(u) , v \rangle = \lambda \,( u , v)_{L^2(\D)} \qquad \mbox{for all } v\in H^1_0(\D),
\end{align}
with
\begin{align*}
\langle E^{\prime}(u), v \rangle =
\Re \int_{\D}
\nabla u \cdot \nabla\overline{v}
+\tfrac{\Veps}{\eps^2}\,u \overline{v}
-\tfrac{\Omegaeps}{\eps}\,\overline{v}\,\mathcal{L}_3 u
+\tfrac{\beta}{\eps^2}|u|^2 u\,\overline{v}
\,\dx.
\end{align*}
Equation \eqref{GPE-abstract} is known as the Gross--Pitaevskii eigenvalue problem where $\lambda$ is called the ground state eigenvalue.

\subsection{Energy-inner product and weighted $H^1$-norm}
\label{subsec:energy-inner-product}

It is common to express the energy in terms of the effective potential given by 
$$
\frac{1}{\eps^2} \VR(x):= \frac{1}{\eps^2} \left( \Veps(x) - \tfrac{1}{4}\, \Omegaeps^2\, (x_1^2+x_2^2) \right)
$$ 
and the covariant gradient $\nablaR$ by
\begin{align}
\label{def-nablaR}
\nablaR w := \nabla w + \ci \tfrac{\Omegaeps}{2\eps } R^{\top} w
\end{align}
for the divergence-free vector field $R(x):=(x_2,-x_1,0)$ if~$d=3$ and $R(x):=(x_2,-x_1)$~if $d=2$. Physically, $-\nabla \VR$ can be interpreted as the net external force per unit mass acting on the condensate, arising from the trapping potential $\Veps$ (pulling inward) and the centrifugal effect of rotation (pushing outward). In this sense, it describes the combined trap--centrifugal force that determines where the condensate density tends to concentrate (moving according to $-\nabla \VR$).

With this, we introduce the bilinear form 
\begin{align}
\label{R-inner-product}
a_{\eps}(v,w) := \re \Big(  \int_{\mathcal{D}} \nablaR v \cdot \overline{\nablaR w} \dx + \int_{\mathcal{D}} \tfrac{1}{\eps^2} \VR v \, \overline{w} \dx  \Big) .
\end{align}
An easy calculation shows that we can express the energy through $a_{\eps}(\cdot,\cdot)$ as
\begin{align}
\label{energy-expressed-aeps}
E(w) \,\,=\,\, \tfrac{1}{2} a_{\eps}(w,w) + \tfrac{\beta}{4\eps^2} \| w \|_{L^4(\D)}^4 \qquad \mbox{for any } w\in H^1_0(\D).
\end{align}
On unbounded domains, the existence of ground states requires the effective potential $\VR$ to be positive such that the trapping potential is strong enough (towards infinity) to compensate the high centrifugal forces caused by fast rotation. Indeed, while $-\nabla \VR$ determines the direction of the net trap--centrifugal force, the absolute level of $\VR$ controls whether the energy can be lowered by moving mass to larger distances. If $\VR$ becomes negative at infinity, the condensate can continuously decrease its energy by escaping outward, and no minimizer exists.

On bounded domains, this mechanism is absent because the zero Dirichlet boundary condition acts as an infinitely strong trapping potential so that the condensate cannot escape the domain. Even if $\VR$ becomes negative in parts of $\D$, the force $-\nabla \VR$ may push the density toward the boundary, but the condensate remains confined and ground states still exist. 

Nevertheless, in the present work we impose a positivity condition on $\VR$. While not required for existence on bounded domains, it selects a balanced trap--centrifugal regime and yields $\eps$-explicit coercivity bounds that are convenient for the subsequent error analysis. The following fixes the assumption.
\begin{enumerate}[resume,label={(A\arabic*)}]
\item\label{A3}
It holds
\begin{align*}
 \Veps(x) - \frac{1+\eps^2}{4}\Omegaeps^2 (x_1^2+x_2^2)
   = \VR(x) - \frac{\eps^2}{4}\Omegaeps^2 (x_1^2+x_2^2)
   \ge 0
   \qquad \text{for almost all } x\in \D .
\end{align*}
\end{enumerate}
The $\eps$-dependence in \ref{A3} allows the critical rotation frequency to be approached as $\eps\to 0$. Note that adding an arbitrary real constant to $\Veps$ does not change the set of ground states on $\S$, since this only shifts the energy by a constant on the constraint $\|u\|_{L^2(\D)}=1$.

Since $a_{\eps}(\cdot,\cdot)$ in \eqref{R-inner-product}  plays an important role in our error analysis, we will introduce a weighted $H^1$-norm such that $a_{\eps}(\cdot,\cdot)$ is continuous and weakly coercive with respect to this norm. For the same parameter $\eps>0$ from the problem formulation, we define 
\begin{align}
\label{def-H1eps-norm}
\| v \|_{H^1_{\eps}(\D)} := \sqrt{ \tfrac{1}{\eps^{2}} \| v \|_{L^2(\D)}^2 + \| \nabla v \|_{L^2(\D)}^2 }
\qquad \mbox{for } v\in H^1_0(\D).
\end{align}
The next lemma establishes continuity and weak-coercivity of $a_{\eps}(\cdot,\cdot)$ with respect to $\| \cdot \|_{H^1_{\eps}(\D)}$ if assumptions \ref{A1}-\ref{A2} hold. If additionally assumption \ref{A3} is valid, then $a_{\eps}(\cdot,\cdot)$ is also strongly coercive, however, with a coercivity constant that degenerates with $\eps$. 

\begin{lemma}[Continuity and coercivity of $a_{\eps}(\cdot,\cdot)$]
\label{lem:aeps-cont-coercive}
Let Assumptions~\ref{A1}--\ref{A2} hold and let $a_{\eps}(\cdot,\cdot)$ be defined by
\eqref{R-inner-product}. Then there exists an $\eps$-independent constant $C_0>0$, such that for all $v,w\in H^1_0(\D)$,
\begin{align}
\label{eq:aeps-cont}
|a_{\eps}(v,w)|
\;\le\;
C_0\, \| v \|_{H^1_{\eps}(\D)} \, \| w \|_{H^1_{\eps}(\D)}.
\end{align}
Moreover, there are $\eps$-independent constants $c_1>0$, such that for all $v\in H^1_0(\D)$,
\begin{align}
\label{eq:aeps-weak-coercive}
a_{\eps}(v,v)
\;\ge\;
\tfrac{1}{2}\, \| v \|_{H^1_{\eps}(\D)}^2
\;-\;
c_1 \,\tfrac{1}{\eps^2}\,\|v\|_{L^2(\D)}^2.
\end{align}
If, in addition, \ref{A3} holds, then 
\begin{align}
\label{eq:aeps-strong-coercive}
a_{\eps}(v,v)
\;\ge\;
\tfrac{\eps^2}{1+\eps^2} \,\| \nabla v \|_{L^2(\D)}^2
\qquad\text{for all }v\in H^1_0(\D).
\end{align}
\end{lemma}

\begin{proof}
\emph{Continuity:} Since $\D$ is bounded, the vector field $R$ satisfies $\|R\|_{L^\infty(\D)} =: C_R < \infty$. Hence, we can estimate $\nablaR v = \nabla v + \ci \tfrac{\Omegaeps}{2\eps} R^\top v$ as
\begin{eqnarray}
\label{eq:nablaR-H1eps}
\nonumber \|\nablaR v\|_{L^2(\D)}
&\le&
\|\nabla v\|_{L^2(\D)} + |\tfrac{\Omegaeps}{2\eps}|\,\|R\|_{L^\infty(\D)}\,\|v\|_{L^2(\D)}
\;\le\;
\|\nabla v\|_{L^2(\D)} + \frac{|\Omegaeps|C_R}{2}\,\frac{1}{\eps}\,\|v\|_{L^2(\D)} \\
&\overset{\eqref{def-H1eps-norm}}{\le}& C_{\Omegaeps,R}\, \|v\|_{H^1_{\eps}(\D)},
\qquad
C_{\Omegaeps,R} := 1+\tfrac{|\Omegaeps|C_R}{2}.
\end{eqnarray}
Therefore, $(\nablaR v , \nablaR w)_{L^2(\D)} \,\le\, C_{\Omegaeps,R}^2\, \|v\|_{H^1_{\eps}(\D)}\,\|w\|_{H^1_{\eps}(\D)}$. For the potential term we use that $\Veps\in L^{\infty}(\D)$ to see that $\|  \VR \|_{L^{\infty}(\D)} \le \| \Veps \|_{L^{\infty}(\D)} + \tfrac{\Omegaeps^2}{4} C_R^2 $ is bounded independent of $\eps$ (since $\| \Veps \|_{L^{\infty}(\D)}$ is bounded independent of $\eps$ by assumption). Hence, 
\begin{eqnarray*}
(\tfrac{1}{\eps^2}\VR v , w )_{L^2(\D)}
\,\,\,\le\,\,\, \frac{1}{\eps^2}\,\|\VR\|_{L^{\infty}(\D)}\,\|v\|_{L^2(\D)}\,\|w\|_{L^2(\D)} 
&\overset{\eqref{def-H1eps-norm}}{\le}& \|\VR\|_{L^{\infty}(\D)}\,
\|v\|_{H^1_{\eps}(\D)}\,\|w\|_{H^1_{\eps}(\D)},
\end{eqnarray*}
together with \eqref{eq:nablaR-H1eps} proves \eqref{eq:aeps-cont}.\\[0.4em]
\emph{Weak coercivity:} The Young inequality together with $\nablaR v = \nabla v + \ci \tfrac{\Omegaeps}{2\eps} R^\top v$ imply that it holds $|\nablaR v|^2 \ge \tfrac{1}{2}|\nabla v|^2 - \tfrac{1}{4\eps^2} |\Omegaeps|^2 |R|^2 |v|^2$.
Using $|R|\le C_R$ we get
\begin{align*}
\| \nablaR v\|^2 
&\ge 
\tfrac12 \|\nabla v\|_{L^2(\D)}^2
- \tfrac{\Omegaeps^2 C_R^2}{4\eps^2}\,\|v\|_{L^2(\D)}^2.
\end{align*}
We obtain 
\begin{eqnarray*}
a_{\eps}(v,v)
&=& \| \nablaR v\|^2 
  +  \tfrac{1}{\eps^2} \int_{\D} \VR |v|^2\dx \\
&\ge&
\tfrac{1}{2} \|\nabla v\|_{L^2(\D)}^2
- \tfrac{1}{\eps^2} \tfrac{\Omegaeps^2 C_R^2}{4}\,\|v\|_{L^2(\D)}^2
- \tfrac{1}{\eps^2} \| \VR \|_{L^{\infty}(\D)}  \,\|v\|_{L^2(\D)}^2,
\end{eqnarray*}
which proves \eqref{eq:aeps-weak-coercive}.\\[0.4em]
\emph{Strong coercivity:} For any $\eta>0$, Young's inequality yields
\begin{align}
\label{eq:grad-upper}
\|\nabla v\|_{L^2(\D)}^2
\,\,\,\le\,\,\, (1+\eta)\,\| \nablaR v\|_{L^2(\D)}^2
  + \left(1+\tfrac{1}{\eta}\right) \tfrac{1}{\eps^2} \tfrac{|\Omegaeps|^2}{4} \int_{\D} |R|^2 |v|^2 \dx.
\end{align}
On the other hand, assumption~\ref{A3} ensures 
$\VR(x) \ge 
\frac{\eps^2}{4}\,\Omegaeps^2 |R|^2$ (for a.e. $x\in\D$) 
hence
\begin{align}
\label{eq:VR-lower}
\int_{\D} \tfrac{1}{\eps^2}\VR |v|^2\,dx
\;\ge\;
 \tfrac{|\Omegaeps|^2}{4}\int_{\D} |R|^2 |v|^2\dx.
\end{align}
Combining \eqref{eq:VR-lower} with the definition of $a_{\eps}$,
\begin{eqnarray*}
a_{\eps}(v,v)
&=& \int_{\D} |\nablaR v|^2\dx
  + \int_{\D} \tfrac{1}{\eps^2}\VR |v|^2\dx
\;\ge\;
\int_{\D} |\nablaR v|^2\,dx
+ \tfrac{|\Omegaeps|^2}{4} \int_{\D} |R|^2 |v|^2\dx \\
&\overset{\eqref{eq:grad-upper}}{\ge}&
\tfrac{1}{1+\eta} \|\nabla v\|_{L^2(\D)}^2 + \left( 1- \tfrac{1}{\eta} \tfrac{1}{\eps^2} \right) \tfrac{|\Omegaeps|^2}{4} \int_{\D} |R|^2 |v|^2 \dx .
\end{eqnarray*}
The choice $\eta=\eps^{-2}$ proves \eqref{eq:aeps-strong-coercive}.
\end{proof}

\begin{remark}[Critical velocity regime]
\label{remark-critical-velocity-regime}
In Lemma~\ref{lem:aeps-cont-coercive} we obtained
\begin{align*}
a_{\eps}(v,v)
\;\ge\;
\tfrac{\eps^2}{1+\eps^2}\,\|\nabla v\|_{L^2(\D)}^2
\qquad\text{for all } v\in H^1_0(\D).
\end{align*}
which implies for some generic $C>0$ (depending on the Poincar\'e-Friedrichs constant) that
\begin{align*}
a_{\eps}(v,v)
\;\ge\;
C\, \eps^2\,\| v\|_{H^1_{\eps}(\D)}^2
\qquad\text{for all } v\in H^1_0(\D).
\end{align*}
Far from the critical angular velocity this bound is typically pessimistic. 
In a near-critical regime, however, the effective trapping potential $\VR$ may be comparable to the centrifugal potential, i.e.
\begin{align*}
\tfrac{\eps^2}{4}\,\Omegaeps^2 (x_1^2+x_2^2)
\;\le\; \VR(x) \;\le\;
C\,\tfrac{\eps^2}{4}\,\Omegaeps^2 (x_1^2+x_2^2)
\end{align*}
for some constant $C>0$. In this case one expects the coercivity constant of $a_\eps(\cdot,\cdot)$ with respect to the norm $\|\cdot\|_{H^1_{\eps}(\D)}$ to be of order $\eps^2$. Considering the generalized eigenvalue problem
\begin{align*}
a_{\eps}(z_i,v)
= \mu_i\,( z_i, v)_{H^1_{\eps}(\D)}
\qquad\text{for all } v\in H^1_0(\D),
\end{align*}
the optimal coercivity constant is given by
\begin{align}
\label{hat-lambda-min}
\mucrit
:= \inf_{v\in H^1_0(\D)\setminus\{0\}}
\frac{a_\eps(v,v)}{\| v\|_{H^1_{\eps}(\D)}^2},
\end{align}
and in the near-critical regime one expects $\mucrit \sim\eps^2$, whereas away from the critical velocity, i.e. if $ \Veps(x) \ge \frac{1+M^2}{4}\Omegaeps^2 (x_1^2+x_2^2)$ for some positive constant $M=O(1)$, we have $\mucrit\sim 1$.
\end{remark}

\subsection{Stability bounds for the ground state and the ground state energy}

In this section, we will establish $\eps$-explicit bounds for the ground state energy and ground state eigenvalues, as well as for ground states itself in the $L^{\infty}$-norm, the $H^1$-norm and the $H^2$-norm.

For simplicity, we exploit the notation $A \lesssim B$ to abbreviate that $A \le C \, B$ for some constant $C>0$ that is independent of $\eps$, but which can potentially depend on $\D$, $\Veps$, $\Omegaeps$ and $\beta$. 

\begin{lemma}[Bounds for ground state energy and eigenvalue]
Assume~\ref{A1}--\ref{A3}, $0<\eps\ll1$ and let $u \in \S$ denote a ground state in the sense of \eqref{definition-groundstate}, with associated eigenvalue $\lambda \in \R$ such that $E^{\prime}(u) = \lambda \mathcal{I} u$ in $H^{-1}(\D)$. Then it holds
\begin{align}
\label{energy-lambda-gs-bounds}
E(u) \,\,\lesssim \,\, \eps^{-2}
\qquad  \mbox{and} \qquad
\lambda \,\, \lesssim \,\, \eps^{-2}.
\end{align}
\end{lemma}

\begin{proof}
Let $u_0 \in \S$ be a fixed purely real--valued test function, hence $\int_{\mathcal{D}} \overline{u_0}\, \mathcal{L}_{3}u_0 = 0$ and rotation term in the energy disappears. For instance, one may take $u_0$ to be the (positive) ground state for $\Omegaeps = 0$. By minimality of $u$ we have
\begin{eqnarray*}
E(u) &=& \inf_{v\in\S} E(v) \,\le\, E(u_0) \,=\, \tfrac12 \int_{\D} \Big( |\nabla u_0|^2 + \tfrac{\Veps}{\eps^2} |u_0|^2 
+ \tfrac{\beta}{2\eps^2} |u_0|^4 \Big) \dx \\
&\le& \tfrac12 \Big( \|\nabla u_0\|_{L^2(\D)}^2 
+ \tfrac{1}{\eps^2} \|\Veps\|_{L^\infty(\D)} \|u_0\|_{L^2(\D)}^2
+  \tfrac{1}{\eps^2}  \tfrac{\beta}{2} \|u_0\|_{L^4(\D)}^4 \Big) \\
&\overset{\ref{A2}}{\lesssim}&
 \|\nabla u_0\|_{L^2(\D)}^2 + \tfrac{1}{\eps^2} \left( \|u_0\|_{L^2(\D)}^2
 +  \|u_0\|_{L^4(\D)}^4 \right)
\end{eqnarray*}
Since $u_0$ is fixed, all norms of $u_0$ are finite and independent of $\eps$, and $\|u_0\|_{L^2(\D)}=1$. Thus there exists a constant $C_E>0$, independent of $\eps$, such that $E(u) \,\le\, C_E \Big( 1 + \tfrac{1}{\eps^2} \Big)$, which already yields $E(u)=O(\eps^{-2})$ as $\eps\to0$.

For the bound on $\lambda$, we use the constrained Euler--Lagrange relation $E'(u) \,=\, \lambda \,\mathcal{I} u$. Using $\|u\|_{L^2(\D)}=1$, we obtain
\begin{align*}
\lambda 
= \int_{\D} \Big( |\nabla u|^2 + \tfrac{\Veps}{\eps^2} |u|^2 
- \tfrac{\Omegaeps}{\eps}\, \bar{u}\, \mathcal{L}_{3} u 
+ \tfrac{\beta}{\eps^2} |u|^4 \Big)\dx
= 2 E(u) + \tfrac{1}{2} \int_{\D} \tfrac{\beta}{\eps^2} |u|^4 \dx 
\le 4 E(u) \lesssim \eps^{-2}.
\end{align*}
\end{proof}

\begin{lemma}[Stability bounds for the ground state]
\label{lem:stability-bounds-groundstates}
Assume~\ref{A1}--\ref{A3}, $0<\eps\ll1$ and let $u \in \S$ denote a ground state in the sense of \eqref{definition-groundstate}. Then $u \in H^1_0(\D) \cap H^2(\D)$ and it holds
\begin{align*}
\| u \|_{L^{\infty}(\D)}\,\,\lesssim \,\, 1, \quad
\| \nabla u \|_{L^2(\D)}\,\,\lesssim \,\, \eps^{-1}
\quad
\| \nabla u \|_{L^4(\D)}\,\,\lesssim \,\, \eps^{-1}
\quad  \mbox{and} \quad
\| D^2 u \|_{L^2(\D)}\ \,\, \lesssim \,\, \eps^{-2},
\end{align*}
where $D^2 u$ denotes the Hessian of $u$.
\end{lemma}

\begin{proof}
The ground state $u$ satisfies the Euler--Lagrange equation
\begin{align*}
-\Delta u + \tfrac{1}{\eps^2} \Veps u - \tfrac{\Omegaeps}{\eps} \mathcal{L}_3 u 
+ \tfrac{\beta}{\eps^2} |u|^2 u = \lambda \,u
\qquad\text{in }\D,\qquad u=0\text{ on }\partial\D,
\end{align*}
for some $\lambda\in\R$. By standard elliptic regularity on convex domains with bounded coefficients, this implies $u\in H^2(\D)\cap H_0^1(\D)$ and, by embedding, $u\in L^{\infty}(\D)$.

For the $L^\infty$-bound, we use that, by a maximum--principle argument (cf. \cite[Proof of Lemma 3.5]{PHMY242}) applied to the scaled coefficients $\Veps/\eps^2$, $\Omegaeps/\eps$ and $\beta/\eps^2$, one can show
\begin{align*}
\|u\|_{L^\infty(\D)} \,\le\, \sqrt{\tfrac{\lambda}{\beta} \, \eps^2} \, \overset{\eqref{energy-lambda-gs-bounds}}{\lesssim} \, 1,
\end{align*}
which yields the first estimate in the statement.

The estimate in $H^1$ is obtained from the weak coercivity estimate in Lemma \ref{lem:aeps-cont-coercive}. In fact, using $u \in \S$, we have
\begin{align*}
\tfrac{1}{2}\, \| \nabla u \|_{L^2(\D)}^2 
\overset{\eqref{def-H1eps-norm}}{\le}
\tfrac{1}{2}\, \| u \|_{H^1_{\eps}(\D)}^2 
\overset{\eqref{eq:aeps-weak-coercive}}{\le} a_{\eps}(u,u) + c_1 \,\tfrac{1}{\eps^2} 
\le 2 \, E(u)   + c_1 \,\tfrac{1}{\eps^2}  \overset{\eqref{energy-lambda-gs-bounds}}{\lesssim} \eps^{-2}.
\end{align*}
Taking the square root on both sides proves $\| \nabla u \|_{L^2(\D)}\lesssim \eps^{-1}$.

For the $H^2$--bound, we rewrite the Euler-Lagrange equation as
\begin{align*}
-\Delta u = f_\eps
\quad\text{with}\quad
f_\eps := \lambda u - \tfrac{1}{\eps^2} \Veps u + \tfrac{\Omegaeps}{\eps} \mathcal{L}_3 u - \tfrac{\beta}{\eps^2}|u|^2u.
\end{align*}
By elliptic regularity, we have
\begin{align*}
\|u\|_{H^2(\D)} \,\lesssim\, \|f_\eps\|_{L^2(\D)}.
\end{align*}
Using the bounds on the coefficients and the previously obtained estimates, we get
\begin{align*}
\|f_\eps\|_{L^2(\D)}
&\le |\lambda| \|u\|_{L^2(\D)}
+ \tfrac{1}{\eps^2} \|\Veps\|_{L^\infty(\D)} \|u\|_{L^2(\D)}
+ \tfrac{|\Omegaeps|}{\eps} \|\mathcal{L}_3 u\|_{L^2(\D)}
+ \tfrac{\beta}{\eps^2} \|u\|_{L^\infty(\D)}^2 \|u\|_{L^2(\D)} \\
&\lesssim \eps^{-2} 
+\eps^{-1} \|\nabla u\|_{L^2(\D)}
+\eps^{-2}  \,\,\lesssim\,\, \eps^{-2},
\end{align*}
where we used $\|u\|_{L^2(\D)}=1$, $\|\mathcal{L}_3 u\|_{L^2(\D)}\lesssim \|\nabla u\|_{L^2(\D)}$ on bounded domains, the bound $\|\nabla u\|_{L^2(\D)}\lesssim\eps^{-1}$, and $\|u\|_{L^\infty(\D)}\lesssim 1$. Hence $\|u\|_{H^2(\D)}\lesssim \eps^{-2}$, and in particular $\|D^2 u\|_{L^2(\D)} \,\lesssim\, \eps^{-2}$, which proves the $H^2$-estimate.

It remains to prove the $W^{1,4}$-bound. Here we use the Gagliardo–Nirenberg inequality $\| \nabla u \|_{L^4(\D)}^2 \lesssim \| u \|_{L^{\infty}(\D)} \| u\|_{H^2(\D)}$ ($d\le 3$) together with the previous estimates to conclude $\| \nabla u \|_{L^4(\D)} \lesssim \eps^{-1}$.
\end{proof}

\subsection{Second-order conditions for minimizers}
In this subsection we analyze the second-order conditions for constrained minimizers in the rapid rotation Thomas--Fermi regime. 
Adapting the corresponding findings of \cite{HY25MathComp} to the present $\eps$-scaled setting, we examine the structure of the Hessian $E''(u)$ at a ground state and discuss how its coercivity and stability properties depend on the underlying $\eps$-scaling. These properties will be essential for the subsequent error analysis.

First of all, it is easy to check that $E$ is five-times Fr\'echet differentiable with   vanishing derivatives of order $\ge 5$. Using the representation of the energy in \eqref{energy-expressed-aeps}, the first derivative can be computed as
\begin{align*}
\langle E^{\prime}(v), w \rangle = a_{\eps}(v,w) + \tfrac{\beta}{\eps^2} ( |v|^2 v , w )_{L^2(\D)} \qquad \mbox{for } v,w \in H^1_0(\D)
\end{align*}
and the second derivative as 
\begin{align} 
\label{sd-new}
\langle E''(u)v,w \rangle =  a_{\eps}(v,w) + \tfrac{\beta}{\eps^2} ( |u|^2 v + 2 \, \Re (u \overline{v}) \, u , w )_{L^2(\D)}
\end{align}
for $u,v,w \in H^1_0(\D)$. In particular, we have for $u=v$ the useful identity
\begin{align}
\label{relation-secE-firstE}
\langle E^{\prime\prime}(u) u , w \rangle \,\,\,=\,\,\, \langle E^{\prime}(u) , w \rangle  + \tfrac{2 \beta}{\eps^2} ( |u|^2 u , w )_{L^2(\D)} 
\,\,\,=\,\,\, a_{\eps}(u,w)  + \tfrac{3 \beta}{\eps^2} ( |u|^2 u , w )_{L^2(\D)} .
\end{align}
Since the minimization problem \eqref{definition-groundstate} is posed on the unit sphere $\mathbb{S} \subset H^1_0(\D)$, the relevant first and second-order optimality conditions for minimizers are obtained by restricting the test functions to the tangent space at $u \in \mathbb{S}$ given by
\begin{align*}
T_{u}\mathbb{S} \,:=\, 
\{ v\in H^1_0(\D) \,| \, (u,v)_{L^2(\D)} = 0\}.
\end{align*}
This space represents the admissible first-order variations that preserve the constraint. The usual first-order optimality condition for constrained minimizers $u$ yields the existence of a Lagrange multiplier (or eigenvalue) $\lambda\in\R$ such that
\begin{align*}
\langle E^{\prime}(u) , v \rangle = \lambda ( u , v)_{L^2(\D)} \qquad \mbox{for all } v\in H^1_0(\D).
\end{align*}
This condition is equivalent to
\begin{align*}
\langle E^{\prime}(u) - \lambda \mathcal{I} u , v \rangle = 0 \qquad \mbox{for all } v\in T_{u}\mathbb{S}.
\end{align*}
Note that this implies positivity of $\lambda$ as
$$
\lambda = \lambda \, \| u \|_{L^2(\D)}^2 = \langle E^{\prime}(u) , u \rangle = a_{\eps}(u,u)  + \tfrac{\beta}{\eps^2} \|u\|_{L^4(\D)}^4 >0.
$$
The necessary second-order condition for minimizers requires that the constrained Hessian $E''(u)-\lambda \calI$ has no negative eigenvalues on the tangent space, i.e.
\begin{align*}
\langle (E^{\prime\prime}(u) - \lambda \mathcal{I}) v , v \rangle \ge 0 \qquad \mbox{for all } v\in T_{u}\mathbb{S}.
\end{align*}
Ideally, one would hope for strict positivity of the spectrum such that the constrained Hessian has a bounded inverse. However, this is not possible since minimizers are at most locally unique up to constant phase shifts, i.e., it holds $E(u)=E(e^{\ci \omega}u)$ for any phase angle $\omega \in [-\pi,\pi)$. In other words, if $u \in \S$ is a minimizer then $e^{\ci \omega}u \in \S$ is another minimizer. This causes the constrained Hessian to degenerate in the tangential direction $\ci u$, i.e. $(E^{\prime\prime}(u) - \lambda \mathcal{I})(\ci u)=0$, where we refer to \cite{HY25MathComp} for more detailed explanations. The usual sufficient second-order condition for the GPE is therefore obtained by restricting the constrained Hessian to the horizontal space
\begin{align*}
H_u \mathbb{S} := T_{u}\mathbb{S} \cap T_{\ci u}\mathbb{S},
\end{align*}
which removes the phase direction $\ci u$ while retaining all other admissible directions in the tangent space. On this subspace we can expect a positive spectrum of $E^{\prime\prime}(u) - \lambda \mathcal{I}$. This is reflected in the following definition, which says that, except for phase shifts, a minimizer $u \in \S$ is non-degenerate.
\begin{definition}[Quasi-isolation]
A minimizer $u \in \S$ of $E$ with $E^{\prime}(u)= \lambda \calI u$ is called {\it quasi-isolated} if it fulfils the sufficient second-order condition, i.e.
\begin{align*}
\langle E^{\prime\prime}(u) v , v \rangle - \lambda ( v , v)_{L^2(\D)} > 0  \qquad \mbox{for all } v\in H_u \mathbb{S} \setminus \{ 0 \}. 
\end{align*}
\end{definition}
For any ground state $u$, $E^{\prime\prime}(u)$ represents a continuous and weakly coercive bilinear form on $H^1_0(\D)$ where the following lemma specifies the dependence on $\eps$. Furthermore, if $u$ is quasi-isolated then $E^{\prime\prime}(u)$  is strongly coercive on the horizontal space.
\begin{lemma}[Continuity and coercivity of the constrained Hessian]\label{lem:continuity-coercivity-secE}
Assume \ref{A1}-\ref{A3} and let $u\in \S$ denote a ground state with ground state eigenvalue $\lambda\in\R$. Then $E^{\prime\prime}(u)$ is an $H^1_{\eps}$-continuous bilinear form, i.e., for all $v,w\in H^1_0(\D)$ it holds
\begin{align*}
\langle E^{\prime\prime}(u) v ,w \rangle 
\,\,\, \lesssim \,\,\, \| v\|_{H^1_{\eps}(\D)} \, \| w \|_{H^1_{\eps}(\D)}
\end{align*}
and consequently
\begin{align}
\label{continuity-constrained-Hessian}
\langle (E^{\prime\prime}(u)  -\lambda \mathcal{I}) v ,w \rangle 
\,\,\, \lesssim \,\,\, \| v\|_{H^1_{\eps}(\D)} \, \| w \|_{H^1_{\eps}(\D)}.
\end{align}
On $H^1_0(\D)$, the constrained Hessian $(E^{\prime\prime}(u)  -\lambda \mathcal{I})$ is also weakly coercive, i.e., there exists a constant $\tilde{c}_1>0$ such that the following G{\aa}rding inequality holds:
\begin{align}
\label{eq:secE-weak-coercive}
\langle (E^{\prime\prime}(u)  -\lambda \mathcal{I}) v ,v \rangle 
\;\ge\;
\tfrac{1}{2}\, \| v \|_{H^1_{\eps}(\D)}^2
\;-\;
\tilde{c}_1 \,\tfrac{1}{\eps^2}\,\|v\|_{L^2(\D)}^2  \hspace{30pt} \mbox{for all } v\in  H^1_0(\D).
\end{align}
Finally, if $u$ is quasi-isolated, then the constrained Hessian is strongly coercive on the horizontal space, i.e., the exists an $\eps$-dependent constant $\eta(\eps)>0$ such that
 \begin{align}
\label{eq:secE-strong-coercive}
\langle (E^{\prime\prime}(u)  -\lambda \mathcal{I}) v ,v \rangle 
\;\ge\;
\eta(\eps)\, \| v \|_{H^1_{\eps}(\D)}^2 \qquad \mbox{for all } v\in H_u \S
\end{align}
and we can bound 
\begin{align}
\label{bounds-etaeps}
\mucrit \, (1- \tfrac{\lambda_{1}}{\lambda_{2}}) \,\, \le \,\, \eta(\eps) \,\, \le \,\, O(1).
\end{align}
where $\mucrit>0$ is defined in \eqref{hat-lambda-min} and measures closeness to the critical velocity and $\lambda_{1}=\lambda>0$ and $\lambda_2>\lambda_1$ are the smallest and second smallest eigenvalue of $E^{\prime\prime}(u) \vert_{\tangentspace{u}}$.
\end{lemma}

\begin{proof}
The continuity of $E^{\prime\prime}(u)$ follows from the $L^{\infty}$-bound for ground states in Lemma \ref{lem:stability-bounds-groundstates} together with the representation \eqref{sd-new}. For the continuity of $E^{\prime\prime}(u) - \lambda \mathcal{I}$ we additionally use the estimate $\lambda \lesssim \eps^{-2}$ from \eqref{energy-lambda-gs-bounds}.
The weak coercivity of $E^{\prime\prime}(u)  -\lambda \mathcal{I}$ follows from the weak coercivity of $a_{\eps}(\cdot,\cdot)$ in \eqref{eq:aeps-weak-coercive} together with again $\lambda \lesssim \eps^{-2}$. For the strong coercivity in \eqref{eq:secE-strong-coercive}, we let $\lambda_2$ denote the smallest eigenvalue of $E^{\prime\prime}(u)$ on $H_u \S$ (equivalently second smallest eigenvalue of $E^{\prime\prime}(u)$ on $\tangentspace{u}$), which fulfils  $\lambda_2>\lambda=\lambda_1$ by quasi-isolation. Hence, 
 \begin{align}
 \label{bounds-etaeps-step1}
\langle (E^{\prime\prime}(u)  -\lambda \mathcal{I}) v ,v \rangle 
\;\ge\;
(\lambda_2 - \lambda)\, \| v \|_{L^2(\D)}^2 \qquad \mbox{for all } v\in H_u \S.
\end{align}
Using \eqref{sd-new} in combination with the definition of $\mucrit$ in \eqref{hat-lambda-min} yields 
 \begin{align}
  \label{bounds-etaeps-step2}
\langle (E^{\prime\prime}(u)  -\lambda \mathcal{I}) v ,v \rangle 
\;\ge\;
a_{\eps}(v,v)  - \lambda\, \| v \|_{L^2(\D)}^2 
\;\ge\; \mucrit \, \| v\|_{H^1_{\eps}(\D)}^2  - \lambda\, \| v \|_{L^2(\D)}^2.
\end{align}
Multiplying \eqref{bounds-etaeps-step1} with $\tfrac{\lambda}{\lambda_2-\lambda}$ and adding it to \eqref{bounds-etaeps-step2} gives
 \begin{align*}
\langle (E^{\prime\prime}(u)  -\lambda \mathcal{I}) v ,v \rangle 
\;\ge\;
\mucrit \, (1- \tfrac{\lambda_{1}}{\lambda_{2}})\, \| v \|_{H^1_{\eps}(\D)}^2 \qquad \mbox{for all } v\in H_u \S,
\end{align*}
hence \eqref{eq:secE-strong-coercive} with $\eta(\eps) \ge \mucrit \, (1- \tfrac{\lambda_{1}}{\lambda_{2}})$, where the $\eps$ dependency enters through $\lambda_1$ and $\lambda_2$ (and potentially through $\mucrit$). The upper bound in \eqref{bounds-etaeps} is a direct consequence of \eqref{continuity-constrained-Hessian}.
\end{proof}
Since our error analysis relies crucially on the inverse operator $(E''(u)-\lambda \mathcal{I})_{\vert H_u\S}^{-1}$, the following lemma establishes corresponding stability and regularity bounds. 
\begin{lemma}[Regularity of solutions to the constrained Hessian problem]
\label{H2-lemma}
Assume \ref{A1}-\ref{A3} and let $u \in \S$ be a quasi-isolated ground state with eigenvalue $\lambda$. Then for any $f \in L^2(\D)$, there exists a unique $z \in H^2(\D) \cap H_u \S$ with
\begin{equation}
\label{adjoint1-1}
 \langle(E''(u)-\lambda \mathcal{I}) z ,v \rangle = \langle \mathcal{I} f ,v\rangle \qquad \mbox{ for all } v \in H_u\S \, ,
\end{equation}
and such that
\begin{align}
\label{h2regularity-general-1}
\|  z \|_{H^{1}_{\eps}(\D)} \lesssim \tfrac{\eps}{\eta(\eps) } \| f \|_{L^2(\D)}
\qquad
\mbox{and}
\qquad
\|  D^2 z \|_{L^{2}(\D)} \lesssim  \tfrac{1}{\eta(\eps) }  \| f \|_{L^2(\D)}.
\end{align}
\end{lemma}

\begin{proof}
Existence of $z \in  H_u \S$ follows from the coercivity of $E''(u)-\lambda \mathcal{I}$ on the horizontal space $H_u\S$. This also implies 
\begin{align*}
\eta(\eps) \| z \|_{H^1_{\eps}(\D)}^2 \,\,\le\,\,   \langle(E''(u)-\lambda \mathcal{I}) z ,z \rangle \,\,=\,\, \langle \mathcal{I} f ,z \rangle 
\,\,\le\,\, \eps \, \|f \|_{L^2(\D)}  \| z \|_{H^1_{\eps}(\D)}.
\end{align*}
The $H^2$-regularity of $z$ is proved in \cite[Lemma 5.11]{HY25MathComp}. By following the arguments in \cite{HY25MathComp} carefully and by exploiting the stability bounds for $u$ in \ref{lem:stability-bounds-groundstates}, the precise $\eps$-dependency can be extracted to show that $\|  D^2 z \|_{L^{2}(\D)} \lesssim  \tfrac{1}{\eta(\eps) }  \| f \|_{L^2(\D)}$.
\end{proof}
Finally, we conclude this section with a useful continuity estimate for $(E^{\prime\prime}(u) - \lambda \calI) u$.
\begin{lemma}
\label{lem:est-secE-firstE}
Assume \ref{A1}-\ref{A3} and let $u \in \S$ be a ground state with eigenvalue $\lambda$. Then it holds 
\begin{align*}
\left| \langle E''(u)u -\lambda \calI u ,v \rangle \right| \,=\,   2\, \tfrac{\beta}{\eps^2} \, \left| ( |u|^2 \, u,v)_{L^2(\D)} \right|
\,\lesssim\, \tfrac{1}{\eps}  \| u \|_{L^2(\D)}   \,  \tfrac{1}{\eps}  \| v \|_{L^2(\D)}   \, 
\,\lesssim\,  \| u \|_{H^1_{\eps}(\D)} \| v \|_{H^1_{\eps}(\D)}.
\end{align*}
for all $v\in H^1_0(\D)$.
\end{lemma}

\begin{proof}
Using $E^{\prime}(u) = \lambda \calI u$, the result follows with \eqref{relation-secE-firstE}.
\end{proof}

\section{Mesh size conditions for vortex resolution}
\label{sec:main-results}
In this section we present our main result on $\eps$-dependent mesh constraints that ensure that discrete minimizers in a $\mathbb{P}^1$ finite element space are reasonable approximations of an exact ground state.

To introduce the precise setting, we consider a shape regular family $\{ \mathcal{T}_h\}$  of conforming triangulations of $\D$, where $h>0$ denotes the mesh size, i.e. the largest diameter of an element of $ \mathcal{T}_h$. On each mesh, the corresponding $\mathbb{P}^1$-Lagrange finite element space is given by
$$  
 V_{h} = \{ v \in H^{1}_{0}(\D) \cap C^{0}(\overline{\mathcal{D}}) \vert \,\, v|_{K} \in \mathbb{P}^{1}(K) \hspace{1mm}\mbox{ for all } K \in \mathcal{T}_{h}\}.
$$
A discrete ground state is defined as a global minimizer $u_h \in V_h \cap \mathbb{S}$ with 
\begin{align}
\label{weak_min_problem}
 E(u_{h}) = \min_{v \in V_h \cap \mathbb{S} } E(v). 
\end{align}
In general, we call $u_h \in V_h \cap \mathbb{S}$ a discrete \emph{local} minimizer if there exists a neighborhood $\mathcal{U}_{\mathbb{S}} \subset \S$ such that 
\begin{align}
\label{weak_min_problem_loc}
 E(u_{h}) = \min_{v \in V_h \cap \mathcal{U}_{\mathbb{S}} } E(v). 
\end{align}
Any local minimizer fulfills the discrete first-order condition, i.e., there exists an eigenvalue (Lagrange multiplier) $\lambda_h \in \R$ such that
\begin{align*}
\langle E^{\prime}(u_h) , v_h \rangle = \lambda_h \, (u_h , v_h)_{L^2(\D)}
\qquad \mbox{for all } v_h \in V_h.
\end{align*}
The corresponding sufficient second-order condition becomes
\begin{align*}
\langle (E^{\prime\prime}(u_h) -\lambda_h \mathcal{I}) v_h , v_h \rangle > 0
\qquad \mbox{for all } v_h \in V_h \cap \tangentspace{u}.
\end{align*}
With this, we are interested in the question: \emph{How fine do we need to select the mesh size $h$ relative to $\eps$ such that $u$ is approximated by some discrete minimizer $u_h$ with optimal order in $h$ and $\eps$?} In the considered Thomas-Fermi scaling, the characteristic vortex-core (healing) length is expected to be of order $\eps$ in the bulk region, cf. \cite{Aftalion2006,AftDu2001,IgnatMillot2006}. Consequently, a natural mesh-resolution requirement is $h \lesssim \eps$, so that the numerical grid is fine enough to resolve the vortex cores. In fact, this condition is also plausible in the light of the stability estimates in Lemma \ref{lem:stability-bounds-groundstates}. Using $\| u \|_{H^2(\D)} \lesssim \eps^{-2}$ in combination with the standard interpolation estimates in $V_h$, we obtain for any ground state $u$ that
\begin{align*}
\inf_{v_h \in V_h} \| u - v_h \|_{L^2(\D)} \lesssim h^2 \| u \|_{H^2(\D)} \lesssim \left(\tfrac{h}{\eps}\right)^2 
\qquad
\mbox{and}
\qquad 
\inf_{v_h \in V_h} \eps \, \| \nabla u - \nabla v_h \|_{L^2(\D)} \lesssim \tfrac{h}{\eps},
\end{align*}
where the $H^1$-error is scaled with $\eps$ since $\| \nabla u \|_{L^2(\D)} \lesssim \eps^{-1}$ according to Lemma \ref{lem:stability-bounds-groundstates}. Hence, we interpret $\eps \, \| \nabla u - \nabla v_h \|_{L^2(\D)}$ as a relative error.

As we will see from our main result (and later confirmed in the numerical experiments), the condition $h \lesssim \eps$ is not sufficient to guarantee the existence of a meaningful discrete minimizer in the neighborhood of each ground state. In fact, the local flatness of $E$ in the neighborhood of a ground state, measured by the degenerate coercivity constant $\eta(\eps)$ in \eqref{eq:secE-strong-coercive}, plays a crucial role and enforces a much stronger condition on the mesh size. Our main theorem reads a follows.
\begin{theorem}[Approximation properties of discrete minimizers]
\label{theorem-main-result}\quad\\
Let \ref{A1}-\ref{A3} hold and and let all constants be defined as in Lemma \ref{lem:continuity-coercivity-secE}. If $h\lesssim \eps$, then
\begin{align*}
\min_{u_h \in V_h \cap \S } E(u_h) \, - \, \min_{u \in \S} E(u)
\,\,\,\,\lesssim\,\,\,\, \left( \tfrac{h}{\eps^2} \right)^2.
\end{align*}
Furthermore, if $u \in \S$ is a quasi-isolated ground state with ground state eigenvalue $\lambda$, then there exists a constant $c^{\ast}>0$ (independent of $h$ and $\eps$) such if the mesh size fulfills 
$$h \,\,\le\,\,  c^{\ast}\, \mucrit \, (1- \tfrac{\lambda_{1}}{\lambda_{2}})\,\eps^{(d+2)/2}$$
there is a local discrete minimizer pair $(\lambda_h ,u_h) \in \R \times (V_h \cap T_{\ci u} \S)$ 
which satisfies the discrete first- and second-order conditions
\begin{align*}
\langle E^{\prime}(u_h) , v_h \rangle \,\, = \,\, \lambda_h \, (u_h , v_h)_{L^2(\D)} \qquad \mbox{for all } v_h \in V_h
\end{align*}
and
\begin{align*}
\langle (E^{\prime\prime}(u_h) - \lambda_h \mathcal{I}) v_h , v_h \rangle \,\,\, \gtrsim \,\,\, \eta(\eps) \, \| v_h \|_{H^1_{\eps}(\D)}^2 \qquad \mbox{for all } v_h \in V_h \cap H_{u} \S,
\end{align*}
and that is a quasi-best approximation to $u$ in the sense that
\begin{align*}
\| u - u_h \|_{H^1_{\eps}(\D)} \,\,\,\lesssim \,\,\,  \inf_{v_h \in V_h} \| u - v_h \|_{H^1_{\eps}(\D)}.
\end{align*}
In particular, we have the asymptotic optimal estimate
\begin{align*}
\| \nabla u - \nabla u_h \|_{L^2(\D)} 
\,\,\,\lesssim \,\,\, \tfrac{h}{\eps^2}.
\end{align*}
\end{theorem}
Theorem \ref{theorem-main-result} is a direct consequence of Theorem \ref{thm:main-H1-est} and Conclusion \ref{conclusion:energy-error}, which we prove later in Section \ref{sec:error-analysis}, together with the inequality $\mucrit \, (1- \tfrac{\lambda_{1}}{\lambda_{2}}) \,\, \le \,\, \eta(\eps)$ from Lemma \ref{lem:continuity-coercivity-secE}.

Our main result predicts that the absolute $H^1$-error behaves asymptotically like $\tfrac{h}{\eps^2}$ and accordingly the scaled (\quotes{relative}) error like $\tfrac{h}{\eps}$. Both rates are asymptotically sharp and coincide with the rates for the best-approximation. However, to achieve these rates, a pre-asymptotic regime needs to be resolved subject to the resolution condition $h \lesssim  \mucrit \, (1- \tfrac{\lambda_{1}}{\lambda_{2}})\,\eps^{(d+2)/2}$, which is significantly stronger than the natural condition $h\lesssim \eps$.

For a better understanding of the resolution condition, recall that $\mucrit$ measures how close we are to the critical frequency and that $\lambda_{1}$ and $\lambda_2$ denote the two smallest eigenvalues of the tangent space hessian $E^{\prime\prime}(u)\vert_{\tangentspace{u}}$. In the fast rotation regime (small $\eps$) with many vortices of core size of order $\eps$, but away from the critical velocity (i.e. $\mucrit=O(1)$, cf. Remark \ref{remark-critical-velocity-regime}), the resolution condition effectively becomes
\begin{align*}
h \,\, \lesssim \,\, \eps^{(d+2)/2} \, \left(1- \tfrac{\lambda_{1}}{\lambda_{2}}\right).
\end{align*}
Consequently, in addition to the core size $\eps$, the mesh size has to compensate for a small first spectral gap in $E^{\prime\prime}(u)\vert_{\tangentspace{u}}$. This corresponds to a small minimal eigenvalue of the Riemannian Hessian $(E^{\prime\prime}(u) - \lambda \mathcal{I})\vert_{H_u \mathbb{S}}$ on the horizontal space. Hence, geometrically, a small spectral gap implies that the energy landscape around the ground state $u$ exhibits \quotes{flat} directions in the energy surface, i.e., directions in which the curvature of the energy is very small. Moving along such directions produces comparatively large variations of the state while only mildly affecting the energy level. 
As a consequence, discretization errors may shift the discrete minimizer along these nearly flat directions, which explains
why a substantially finer mesh is required to keep the numerical solution close to the continuous ground state. In the fast rotation regime close to the critical velocity, the effect is further amplified since $\mucrit$ (defined in \eqref{hat-lambda-min}) may now scale like a small power of $\eps$ and thus further reduces the admissible mesh size.

\section{Error analysis}
\label{sec:error-analysis}

The proof of Theorem \ref{theorem-main-result} takes place in several steps. First, we introduce a suitable Ritz-projection $P_h$ based on the Riemannian Hessian $(E''(u)-\lambda \calI)_{\vert T_{\ci u} \S}$ in Section \ref{subsec:ritz-estimates} and analyze the projection error. In Section \ref{subsec:loc-existance} we establish the existence of a discrete minimizer $u_h$ in the neighborhood of each ground state $u$. The defect $u_h - P_h u$ is estimated in Section \ref{subsec:defect-estimates}. Finally, all results are combined in Section \ref{subsec:combination-estimates} to establish  Theorem \ref{theorem-main-result}.

\subsection{Ritz-projection estimates}
\label{subsec:ritz-estimates}
We start with introducing a Ritz-projection $P_h : T_{\ci u} \S \rightarrow V_h \cap T_{\ci u} \S$ based on the constrained Hessian $E''(u)-\lambda \calI$. Note that we only have coercivity of $E''(u)-\lambda \calI$ on the horizontal space $H_{u}\S$ and that the existence of a Ritz-projection on $T_{\ci u} \S$ is therefore not trivial.
\begin{lemma}[Existence of Ritz-projection]
\label{lem:def-ritz-proj} 
\quad\\
Assume \ref{A1}-\ref{A3} and let $u \in \S$ denote a quasi-isolated ground state with eigenvalue $\lambda$. Then, there exists a projection operator  $P_h : T_{\ci u} \S \rightarrow V_h \cap T_{\ci u} \S$ such that
\begin{align*}
v - P_h v \in T_{u} \S \qquad \mbox{for all } v\in T_{\ci u} \S
\end{align*}
and
 \begin{align}
 \label{galerkin-orth}
\langle  (E''(u)-\lambda \calI) (v-P_{h}v) , w_h \rangle = 0 \quad \mbox{for all } w_h \in V_h \cap \tangentspace{\ci u}.
\end{align}
\end{lemma}

\begin{proof}
Consider the Lagrange functional $\calJ : \mathbb{R} \times \tangentspace{\ci u} \rightarrow (\mathbb{R} \times \tangentspace{\ci u})^*$ given by
 \begin{align*}
\langle \calJ (\sigma , v) ,(\tau , w) \rangle := \langle  E'(v)-\sigma \calI v , w \rangle + \frac{\tau}{2}(1- \int_{\D} |v|^2 ) ,  
\end{align*} 
for $ (\sigma , v)$, $(\tau ,w ) \in \R  \times T_{\ci u}\S$, equipped with the norm $|||(\sigma , v)|||:= |\sigma| + \|v\|_{H^1(\D)}$. For $(\lambda,u) \in  \R  \times T_{\ci u}\S$ we have 
\begin{align} 
\label{new-eigenvalue-problem}	
 \langle \calJ (\lambda, u) , (\tau , w) \rangle =0 \qquad \mbox{ for all } (\tau , w) \in \mathbb{R} \times T_{\ci u}\S
 \end{align}
 and the derivative  $\calJ'(\lambda ,u) :\R  \times T_{\ci u}\S \rightarrow \R  \times T_{\ci u}\S$ at $(\lambda , u)$ is given by
\begin{align}
\label{j-derivative}
\langle \calJ' (\lambda , u) (\sigma , v) , (\tau , w) \rangle \,\,=\,\, \langle  (E''(u)-\lambda \calI)v , w \rangle - \sigma \langle \calI u,w   \rangle - \tau \langle \calI u,v \rangle \, 
\end{align}
for $(\sigma , v),(\tau ,w) \in \R \times T_{\ci u}\S$. Now consider $(\mu,z) \in \R \times T_{\ci u}\S$ in the kernel of $\calJ' (\lambda , u)$, then it holds $\langle \calJ' (\lambda , u) (\mu,z) , (1,0) \rangle = 0$, which implies, by \eqref{j-derivative}, that $\langle \calI u,z \rangle =0$, hence $z \in   T_{u}\S \cap T_{\ci u}\S= H_{u}\S$. We obtain 
\begin{align*}
0 = \langle \calJ' (\lambda , u) (\mu , z) , (\mu , z) \rangle \overset{z\in T_{u}\S}{=}  \langle  (E''(u)-\lambda \calI)z , z \rangle.
\end{align*}
By the quasi-isolation of $u$ we know that $ \langle  (E''(u)-\lambda \calI)v , v \rangle>0$ for all $v\in H_u \mathbb{S} \setminus \{ 0 \}$. We conclude that $z=0$ and consequently also $\mu=0$. Hence,  $\calJ' (\lambda , u) $ has a trivial kernel on $\R \times T_{\ci u}\S$. Since $\calJ'(\lambda,u)$ is self-adjoint and bounded, its range is closed in our setting and
$\mathrm{Ran}(\calJ(\lambda,u)') = (\ker \calJ'(\lambda,u))^\perp = \R \times T_{\ci u}\S$. We conclude that $\calJ' (\lambda , u) $ has a bounded inverse on $\R \times T_{\ci u}\S$.

As a direct consequence, there exists a unique projection $\mathbf{P}_{\hspace{-1pt}h} : \R \times \tangentspace{\ci u} \rightarrow \R \times (V_h \cap \tangentspace{\ci u} )$ with
 \begin{align*}
 \langle \calJ' (\lambda , u) \left( (\sigma , v) - \mathbf{P}_{\hspace{-1pt}h}(\sigma , v) \right), (\tau_h , w_h ) \rangle = 0
 \qquad
 \mbox{for all }
 (\tau_h , w_h )  \in  \R \times (V_h \cap \tangentspace{\ci u} ).
 \end{align*}
Denote the components by $\mathbf{P}_{\hspace{-1pt}h} = (P_{h,1},P_{h,2})$, then
\begin{align*}
0 = \langle  (E''(u)-\lambda \calI) (v-P_{h,2}v) , w_h \rangle - (\sigma-P_{h,1}\sigma) \langle \calI u,w_h   \rangle - \tau_h \langle \calI u, v-P_{h,2}v \rangle \, 
\end{align*}
for all  $ (\tau_h , w_h )  \in  \R \times (V_h \cap \tangentspace{\ci u} )$.  Selecting $w_h=0$ and $\tau_h=1$, we observe that it must hold
 \begin{align*}
\langle \calI u, v-P_{h,2}v \rangle = 0  \quad \Rightarrow \quad v-P_{h,2}v \in \tangentspace{u}.
\end{align*}
 This in turn implies
 \begin{align*}
\langle  (E''(u)-\lambda \calI) (v-P_{h,2}v) , w_h \rangle = (\sigma-P_{h,1}\sigma) \,\langle \calI u,w_h   \rangle
\end{align*}
for all $w_h \in V_h \cap \tangentspace{\ci u}$. 
Linearity $\mathbf{P}_{\hspace{-1pt}h}(\sigma,v) = (P_{h,1}(\sigma),0) + (0,P_{h,2}(v))$ shows that the problem decouples and 
$$
 \sigma-P_{h,1}\sigma = 0.
$$
Hence $P_{h,2} : \tangentspace{\ci u}  \rightarrow V_h \cap \tangentspace{\ci u} $ is uniquely determined by the two conditions
 \begin{align}
 \label{L2-orthogonality}
v-P_{h,2}v \in \tangentspace{u}  \qquad (\Rightarrow \, v-P_{h,2}v \in \tangentspace{u} \cap \tangentspace{\ci u})
\end{align}
and
 \begin{align}
 \label{galerkin-orth}
\langle  (E''(u)-\lambda \calI) (v-P_{h,2}v) , w_h \rangle = 0 \quad \mbox{for all } w_h \in V_h \cap \tangentspace{\ci u}.
\end{align}
These are just the properties claimed in the lemma with $P_h=P_{h,2}$.
\end{proof}
Before we can present the error estimates for $P_h$, we need a short preparation.
\begin{lemma}
Assume \ref{A1}-\ref{A3} and let $u\in \S$ be a ground state. 
If $h\lesssim \eps$, then for every $v\in T_{\ci u} \S$
\begin{align}
\label{best-approx-Tiu-vs-full}
\inf_{v_h \in V_h \cap T_{\ci u} \S}  \| v - v_h \|_{H^1_{\eps}(\D)} \,\,\lesssim\,\, \inf_{v_h \in V_h}  \| v - v_h \|_{H^1_{\eps}(\D)}.
\end{align} 
\end{lemma}

\begin{proof}
The argument is standard and elaborated in \cite[Lemma 5.5]{HY25MathComp}. We briefly sketch it in our setting to show where the constraint for $h$ enters. Let $P_{L^2} : H^1_0(\D) \rightarrow V_h$ denote the $L^2$-projection and consider $P_{L^2}^{\perp} : T_{\ci u}\S \rightarrow V_{h} \cap T_{\ci u}\S$ with
\begin{align*}
P_{L^2}^{\perp} (v) := P_{L^2} (v) - \frac{(P_{L^2}(v) - v, \ci u )_{L^2(\D)} }{ (P_{L^2}(\ci u) , \ci u )_{L^2(\D)} } \, P_{L^2}(\ci u)
\end{align*}
for $v \in T_{\ci u}\S$. It is easy to see that $P_{L^2}^{\perp}$ is a projection on $V_{h} \cap T_{\ci u}\S$. Since $(P_{L^2}(\ci u) , \ci u )_{L^2(\D)} = 1 - (\ci u - P_{L^2}(\ci u) , \ci u )_{L^2(\D)} \ge 1 - \| \ci u - P_{L^2}(\ci u) \|_{L^2(\D)} \ge 1 - h \| \nabla u \|_{L^2(\D)} \ge 1 - c \tfrac{h}{\eps}$, we need $h \le \tfrac{1}{2c} \eps$ to bound $(P_{L^2}(\ci u) , \ci u )_{L^2(\D)}^{-1} \le 2$. In this case, we obtain with $\| P_{L^2}(\ci u) \|_{L^2(\D)} \le \| \ci u \|_{L^2(\D)}  = 1$ that $\| v - P_{L^2}^{\perp} (v)\|_{L^2(\D)}  \, \le  \, 3 \, \| v - P_{L^2} (v)\|_{L^2(\D)}$. Similarly, using the $H^1$-stability of the $L^2$-projection on quasi-uniform meshes, we have
\begin{eqnarray*}
\| \nabla (v - P_{L^2}^{\perp} (v)) \|_{L^2(\D)}  &\lesssim&  \| \nabla (v - P_{L^2} (v)) \|_{L^2(\D)}  + \| v - P_{L^2} (v)\|_{L^2(\D)} \| \nabla u \|_{L^2(\D)} \\
&\lesssim& (1 + \tfrac{h}{\eps} )  \| \nabla (v - P_{L^2} (v)) \|_{L^2(\D)} 
\,\,\, \overset{ h \lesssim \eps}{\lesssim}   \| \nabla (v - P_{L^2} (v)) \|_{L^2(\D)} .
\end{eqnarray*}
Combining the estimates for $\| v - P_{L^2}^{\perp} (v) \|_{L^2(\D)}$ and $\| \nabla (v - P_{L^2}^{\perp} (v)) \|_{L^2(\D)}$ and using that $\| v - P_{L^2} (v) \|_{H^1_{\eps}(\D)} \lesssim \inf_{v_h \in V_h} \| \| v - v_h \|_{H^1_{\eps}(\D)}$ finishes the proof. 
\end{proof}

Next, we prove that $P_h$ yields a quasi-best approximation on $\tangentspace{\ci u}$ provided $h$ is sufficiently small. 
\begin{lemma}\label{lem:H1-ritz-proj-est}
Assume \ref{A1}-\ref{A3}, let $u \in \S$ denote a quasi-isolated ground state with eigenvalue $\lambda$ and let $P_h : T_{\ci u} \S \rightarrow V_h \cap T_{\ci u} \S$ denote the Ritz-projection from Lemma \ref{lem:def-ritz-proj}, then if $h \lesssim \eps \, \eta(\eps)$, it holds
\begin{align*}
\| v - P_h v \|_{H^1_{\eps}(\D)} \,\,\lesssim\,\, \inf_{v_h \in V_h}  \| v - v_h \|_{H^1_{\eps}(\D)} \qquad \mbox{for all } v\in T_{\ci u} \S.
\end{align*} 
In particular, $P_h$ is $H^1_{\eps}$-stable on $T_{\ci u} \S$ for $h \lesssim \eps \, \eta(\eps)$, i.e.,
\begin{align}
\label{Ph-Honeeps-stability}
\| P_h v \|_{H^1_{\eps}(\D)} \,\,\,\lesssim \,\,\, \| v\|_{H^1_{\eps}(\D)} 
\qquad \mbox{for all } v\in T_{\ci u} \S.
\end{align} 
\end{lemma}

\begin{proof}
Recall the G{\aa}rding inequality in \eqref{eq:secE-weak-coercive}, i.e. 
$\langle (E^{\prime\prime}(u)  -\lambda \mathcal{I}) w ,w \rangle \ge
\tfrac{1}{2}\, \| w \|_{H^1_{\eps}(\D)}^2 - \tfrac{\tilde{c}_1}{\eps^2}\,\| w \|_{L^2(\D)}^2$ for all $w\in H^1_0(\D)$, and let $\xi \in H_{u}\S$ denote the unique solution to
 \begin{align}
 \label{def-xi}
\langle  (E''(u)-\lambda \calI) \xi , w \rangle =  \tilde{c}_1 \,\tfrac{1}{\eps^2} \, (v - P_h v , w )_{L^2(\D)} \quad \mbox{for all } w \in  H_{u}\S,
\end{align}
which exists by the coercivity of the constrained Hessian on the horizontal space $H_u\S$ (cf. Lemma \ref{lem:continuity-coercivity-secE}). Here, $v-P_h v$ is an admissible test function in \eqref{def-xi} because Lemma  \ref{lem:def-ritz-proj} ensures $v-P_h v \in H_u\S$ for any $v\in  T_{\ci u} \S$. 
Hence, we can apply a Schatz argument \cite{Sch74} and use $w=v-P_h v \in  H_u\S$ in the G{\aa}rding inequality to obtain for arbitrary $\xi_h, v_h \in V_h \cap \tangentspace{\ci u}$
 \begin{eqnarray*}
 \| v-P_h v \|_{H^1_{\eps}(\D)}^2  &\le&  \langle  (E''(u)-\lambda \calI) (v-P_h v) , v- P_h v \rangle +   \tilde{c}_1 \,\tfrac{1}{\eps^2} \, \| v-P_h v  \|_{L^2(\D)}^2 \\
 &\overset{\eqref{def-xi}}{=}& \langle (E''(u)-\lambda \calI) (v-P_h v + \xi) , v- P_h v \rangle \\
 &\overset{\eqref{galerkin-orth}}{=}& \langle (E''(u)-\lambda \calI) (v-P_h v + \xi - \xi_h) , v- P_h v \rangle \\
  &\overset{\eqref{galerkin-orth}}{=}& \langle (E''(u)-\lambda \calI) (v-v_h + \xi - \xi_h) , v- P_h v \rangle.
\end{eqnarray*}
Note that the last two steps exploited that $(E''(u)-\lambda \calI)$ is symmetric. With the $H^1_{\eps}$-continuity of $(E''(u)-\lambda \calI)$ in Lemma \ref{lem:continuity-coercivity-secE} we obtain
   \begin{eqnarray*}
 \| v-P_h v \|_{H^1_{\eps}(\D)}^2  &\lesssim&  \left( \| v-v_h \|_{H^1_{\eps}(\D)}  + \|  \xi - \xi_h \|_{H^1_{\eps}(\D)}  \right) \| v- P_h v  \|_{H^1_{\eps}(\D)}
\end{eqnarray*}
 and hence
  \begin{eqnarray}
\nonumber  \| v-P_h v \|_{H^1_{\eps}(\D)}  &\lesssim&  \inf_{v_h \in V_h \cap \tangentspace{\ci u}} \| v-v_h \|_{H^1_{\eps}(\D)}  +\inf_{\xi_h \in V_h \cap \tangentspace{\ci u}}  \|  \xi - \xi_h \|_{H^1_{\eps}(\D)}  \\
\label{est-v-Phv-proof} &\overset{\eqref{best-approx-Tiu-vs-full}}{\lesssim}&  \inf_{v_h \in V_h} \| v-v_h \|_{H^1_{\eps}(\D)}  +\inf_{\xi_h \in V_h}  \|  \xi - \xi_h \|_{H^1_{\eps}(\D)}.
\end{eqnarray} 
Using Lemma \ref{H2-lemma}, the second term can be further estimated as
  \begin{eqnarray*}
 \inf_{\xi_h \in V_h}  \|  \xi - \xi_h \|_{H^1_{\eps}(\D)} &\lesssim& (h^2 \eps^{-1} + h ) \, \| D^2\xi \|_{L^2(\D)} \,\,\, \overset{ h\lesssim \eps}{\lesssim} \,\,\,  h  \, \| D^2\xi \|_{L^2(\D)} \\
  &\lesssim& h  \, \eta(\eps)^{-1} \,\tfrac{1}{\eps^2} \| v-P_h v \|_{L^2(\D)} 
  \,\,\, \lesssim \,\,\, \tfrac{h}{\eps \eta(\eps)} \| v-P_h v \|_{H^1_{\eps}(\D)}. 
\end{eqnarray*} 
Consequently, if $\tfrac{h}{\eps \eta(\eps)} \lesssim 1$ is sufficiently small, we can absorb the term $\inf_{\xi_h \in V_h}  \|  \xi - \xi_h \|_{H^1_{\eps}(\D)}$ in the left hand side of \eqref{est-v-Phv-proof}, which proves the desired estimate.
\end{proof}
As a direct conclusion, we can apply Lemma \ref{lem:H1-ritz-proj-est} to $u \in T_{\ci u} \S$ by using that $\| u \|_{H^2(\D)} \lesssim \eps^{-2}$. We obtain the following.
\begin{conclusion}\label{conc:H1-ritz-proj-est-gs}
Assume \ref{A1}-\ref{A3}, let $u \in \S$ denote a quasi-isolated ground state with eigenvalue $\lambda$ and let $P_h : T_{\ci u} \S \rightarrow V_h \cap T_{\ci u} \S$ denote the Ritz-projection from Lemma \ref{lem:def-ritz-proj}. If $h \lesssim \eps \, \eta(\eps)$, it holds
\begin{align}
\label{est:H1-ritz-proj-est-gs}
\| u - P_h u \|_{H^1_{\eps}(\D)} \,\,\lesssim\,\, \tfrac{h}{\eps^2}. 
\end{align} 
\end{conclusion}

\subsection{Local existence of discrete minimizers}
\label{subsec:loc-existance}

To prove the existence of a discrete minimizer in a small neighborhood of an exact ground state we shall use a Pousin--Rappaz technique \cite{PoRa94}. For this, we need to construct a suitable auxiliary representation of the discrete Gross--Pitaveskii equation, which is done as follows.

Using the Ritz-projection $P_h : T_{\ci u} \S \rightarrow V_h \cap T_{\ci u} \S$ from Lemma \ref{lem:def-ritz-proj}, we define the operator $\calJ_h : \R \times T_{\ci u} \mathbb{S} \rightarrow (\R \times T_{\ci u} \mathbb{S})^{\ast}$ by
\begin{eqnarray}
\label{def-calJ-h}
\lefteqn { \langle \calJ_h( \sigma , v ) , (\tau , w) \rangle  } \\
\nonumber &:=& 
\langle  E'(v)-\sigma \calI v , P_{h} w \rangle + \frac{\tau}{2}(1- \int_{\D} |v|^2 )  
 + 
\langle  (E''(u)-\lambda \calI)v ,  w - P_{h} w \rangle
\end{eqnarray}
for $(\sigma,v),(\tau,w) \in  \R \times T_{\ci u} \mathbb{S}$. This operator allows us characterize discrete ground state pairs $(\lambda_h,u_h) \in \R \times (\S \cap V_h)$ as zeros of $\calJ_h$. To be precise, the following lemma holds.
\begin{lemma}\label{lem:equiv-char-disc-gs}
Assume \ref{A1}-\ref{A3} and let $u \in \S$ be a quasi-isolated ground state with eigenvalue $\lambda\in\R$. Then it holds:
\begin{align*}
(u_h,\lambda_h) \in T_{\ci u} \S \times \R \quad
\mbox{ solves }  \calJ_h( \lambda_h , u_h ) = 0
\end{align*}
\emph{if and only if} \,\,$(u_h,\lambda_h) \in (V_h \cap T_{\ci u} \S \cap \S) \times \R$\,\,with  
\begin{align*}
\langle E^{\prime}(u_h) , v_h \rangle = \lambda_h ( u_h , v_h )_{L^2(\D)} \quad \mbox{ for all } v_h \in V_h \cap T_{\ci u} \S .  
\end{align*}
\end{lemma}
\begin{proof}
$\Rightarrow$: If $(u_h,\lambda_h) \in T_{\ci u} \S \times \R$ solves $\calJ_h( \lambda_h , u_h ) = 0$, we immediately conclude $u_h \in \S$ (by testing with $w=0$ and $\tau=1$) and $\langle E^{\prime}(u_h) , v_h \rangle = \lambda_h ( u_h , v_h )_{L^2(\D)}$ for all $v_h \in V_h \cap T_{\ci u} \S$ (by testing with $w \in V_h \cap T_{\ci u} \S$ and using the projection property of $P_h$). It remains to verify that $u_h \in V_h$. For this, note that  $\calJ_h( \lambda_h , u_h ) = 0$ also implies $\langle  (E''(u)-\lambda \calI)u_h ,  w - P_{h} w \rangle = 0$ for all $w\in  T_{\ci u} \S$. Since $P_h u_h \in V_h \cap T_{\ci u} \S$, we obtain for any $w \in T_{\ci u} \S$
 \begin{eqnarray*}
\lefteqn{  \langle  (E''(u)-\lambda \calI) ( u_h -P_{h}u_h) , w - P_{h} w  \rangle }\\
&=&  \langle  (E''(u)-\lambda \calI)  u_h , w - P_{h} w  \rangle -  \langle  (E''(u)-\lambda \calI) ( P_{h}u_h , w - P_{h} w  \rangle \\
&=&  -  \langle  (E''(u)-\lambda \calI) ( P_{h}u_h , w - P_{h} w  \rangle \overset{\eqref{galerkin-orth}}{=} 0.
 \end{eqnarray*}
Selecting $w=u_h$ yields
 \begin{eqnarray*}
\langle  (E''(u)-\lambda \calI) ( u_h -P_{h}u_h) ,u_h -P_{h}u_h)  \rangle  &=& 0.
 \end{eqnarray*}
 However, since $u_h - P_h u_h \in H_{u}\S$ by Lemma \ref{lem:def-ritz-proj}, we can use the strong coercivity of $E''(u)-\lambda \calI$ on the horizontal space $H_{u}\S$ (cf. Lemma \ref{lem:continuity-coercivity-secE}) to conclude $u_h - P_h u_h=0$, and hence $u_h \in V_h$.\\[0.4em]
$\Leftarrow$: The converse direction follows directly by using again property \eqref{galerkin-orth} in Lemma \ref{lem:def-ritz-proj}.
\end{proof}

The next lemma ensures that $\calJ_h^{\prime}(\lambda,u)$ exists and that it has a bounded inverse. For simplicity of the presentation we define on $\R  \times T_{\ci u}\S$ the norm
\begin{align*}
\| (\sigma , v) \|_{\eps} := \sqrt{ \eps^{2} |\sigma|^2 + \|v\|_{H^1_{\eps}(\D)}^2 }.
 \end{align*}
 Note that the additional scaling $\eps^{2}$ is natural remembering that $\lambda \lesssim \eps^{-2}$ and $\|\nabla u \|_{L^2(\D)} \lesssim \eps^{-1}$. In this case we have
 $\| (\lambda , u) \|_{\eps}^2 = \eps^{2} |\lambda|^2 + \tfrac{1}{\eps^2} \|u\|_{L^2(\D)}^2 + \| \nabla u\|_{L^2(\D)}^2$, where each term is consistently of order $O(\eps^{-2})$.
 
\begin{lemma}\label{lem:bound:Jprimeinv}
Assume \ref{A1}-\ref{A3} and let $u \in \S$ be a quasi-isolated ground state with eigenvalue $\lambda\in\R$, then $\calJ_h$ defined in \eqref{def-calJ-h} is Fr\'echet differentiable with derivative
\begin{eqnarray*}
\lefteqn{ \big\langle \calJ_h'(\mu,z)(\sigma,v),(\tau,w)\big\rangle }\\
&=&
\big\langle (E''(z)- \mu \calI)v - \sigma \calI z, P_h w\big\rangle 
\,-\, \tau \langle \calI z,v \rangle \,+\, \big\langle (E''(u)-\lambda \calI)v, w-P_h w\big\rangle,
\end{eqnarray*}
where $(\mu,z),(\sigma,v),(\tau,w)\in \R \times T_{\ci u}\S$. 
Furthermore, $\calJ_h'(\lambda,u)$ has a bounded inverse, where for any $\calF  \in (\R \times T_{\ci u}\S)^{\ast}$,
\begin{align*}
 \| \calJ_h'(\lambda,u)^{-1} \calF  \|_{\eps} \,\, \lesssim \,\,
\eta(\eps)^{-1}  \, \sup\limits_{0\not= (\sigma,v) \in \R \times T_{\ci u}\S } \frac{ |\calF(\sigma,v)| }{ \| (\sigma , v) \|_{\eps} },
\end{align*}
where $\eta(\eps)$ denotes the coercivity constant from \eqref{eq:secE-strong-coercive}.

Finally, it also holds
\begin{align}
\label{projection-error-rep}
\calJ_h'(\lambda,u)^{-1}  \calJ_h( \lambda , u ) = (0, u - P_h u ).
\end{align}
\end{lemma}
\begin{proof}
The formula for the derivative $\calJ_h'(\mu,z)$ follows by direct calculation. For $(\mu,z)=(\lambda,u)$, it simplifies to 
\begin{eqnarray*}
 \big\langle \calJ_h'(\lambda,u)(\sigma,v),(\tau,w)\big\rangle &=& \big\langle (E''(u)-\lambda \calI)v, w\big\rangle \,-\, \sigma \big\langle \calI u , P_h w\big\rangle 
\,-\, \tau \langle \calI u ,v \rangle  \\
&\overset{w - P_hw \in T_{u}\S}{=}& \big\langle (E''(u)-\lambda \calI)v, w\big\rangle \,-\, \sigma \big\langle \calI u , w\big\rangle 
\,-\, \tau \langle \calI u ,v \rangle.
\end{eqnarray*}
This is exactly the same operator $\calJ_h'(\lambda,u) = \calJ'(\lambda,u)$ that appeared earlier in \eqref{lem:def-ritz-proj} in the proof of Lemma \ref{lem:def-ritz-proj} , where we already verified that it has a bounded inverse on $ \R \times T_{\ci u}\S$. Consequently $\calJ_h'(\lambda,u)^{-1}$ exists and it remains to verify the continuity constant.
For that, let $\calF\in (\R \times T_{\ci u}\S)^{\ast}$ be fixed and consider the solution $(\mu,z) := \calJ_h'(\lambda,u)^{-1}\calF \in \R \times T_{\ci u}\S$ to
\begin{align*}
\langle  \calJ_h'(\lambda,u)  (\mu,z) , (\sigma, v) \rangle =  \calF(\sigma, v) 
\qquad\mbox{for all } (\sigma, v) \in \R \times T_{\ci u}\S.
\end{align*}
Using the expression that we obtained for $\calJ_h'(\lambda,u)$, we have
\begin{eqnarray}
\label{def-Jhiprimnv-prob}
\langle (E''(u)-\lambda \mathcal{I} )z,v \rangle - \mu \langle \mathcal{I} u, v\rangle - \sigma \langle \mathcal{I} u, z\rangle  
 &=&  \calF(\sigma, v) .
\end{eqnarray}
Now decompose $z\in T_{\ci u}\S$ uniquely into
\begin{align*}
z = \alpha \,u  + z_{\perp}, \qquad \mbox{where } \alpha \in \R, \quad z_{\perp} \in H_u \S = T_{\ci u}\S \cap T_{u} \S.
\end{align*} 
Select $v=0$ in \eqref{def-Jhiprimnv-prob}, we obtain 
\begin{eqnarray*}
- \sigma \, \alpha \,  = - \sigma \, \alpha \, \langle \mathcal{I} u, u \rangle  = - \sigma \, \langle \mathcal{I} u, z\rangle  
 =  \calF (\sigma, 0) 
\qquad\mbox{for all } \sigma \in \R .
\end{eqnarray*}
Hence $\alpha = \tfrac{-\calF (\sigma, 0)}{\sigma}$ for any $\sigma \in \R\setminus \{ 0\}$ and therefore
\begin{eqnarray}
\label{est-for-alpha}
|\alpha|  &\le& %
\eps\, \sup\limits_{0\not= (\sigma,0) \in \R \times T_{\ci u}\S } \frac{ |\calF(\sigma,0)| }{ \| (\sigma , 0) \|_{\eps} }
\end{eqnarray}
Next, we test in \eqref{def-Jhiprimnv-prob} with $v=z_{\perp} \in H_{u}\S$ and $\sigma=0$, this yields
\begin{eqnarray*}
\big\langle (E''(u)-\lambda \mathcal{I} )z,\; z_{\perp} \big\rangle 
 &=& \langle \calF , (0, z_{\perp} ) \rangle
\end{eqnarray*}
from which we get, together with the coercivity of $E''(u)-\lambda \mathcal{I}$ on $H_{u}\S$, that
\begin{eqnarray*}
 \eta(\eps) \, \| z_{\perp} \|_{H^1_{\eps}(\D)}^2 \,\,\,\le\,\,\, 
\big\langle (E''(u)-\lambda \mathcal{I} )z_{\perp},\; z_{\perp} \big\rangle 
 \,\,\,=\,\,\, \langle \calF , (0, z_{\perp} ) \rangle - \alpha \big\langle (E''(u)-\lambda \mathcal{I} ) u ,\; z_{\perp} \big\rangle.
\end{eqnarray*}
We can use Lemma \ref{lem:est-secE-firstE} and the previous estimate for $|\alpha|$ to bound the last term as
\begin{eqnarray*}
| \alpha| \, | \big\langle (E''(u)-\lambda \mathcal{I} ) u , \; z_{\perp} \big\rangle|
 &\le& |\alpha|\,  \tfrac{1}{\eps}  \| u \|_{L^2(\D)}   \,  \tfrac{1}{\eps}  \| z_{\perp} \|_{L^2(\D)} \\
&\overset{\eqref{est-for-alpha}}{\le}& \| z_{\perp} \|_{H^1_{\eps}(\D)}  \sup\limits_{0\not= (\sigma,0) \in \R \times T_{\ci u}\S } \frac{ |\calF(\sigma,0)| }{ \| (\sigma , 0) \|_{\eps} }
\end{eqnarray*}
Combining the last two estimates yields
\begin{eqnarray*}
 \eta(\eps) \, \| z_{\perp} \|_{H^1_{\eps}(\D)}^2 &\lesssim&
 \left( \sup\limits_{0\not= (0,v) \in \R \times T_{\ci u}\S } \frac{ |\calF(0,v)| }{ \| (0 , v) \|_{\eps} } + \sup\limits_{0\not= (\sigma,0) \in \R \times T_{\ci u}\S } \frac{ |\calF(\sigma,0)| }{ \| (\sigma , 0) \|_{\eps} } \right)  \| z_{\perp} \|_{H^1_{\eps}(\D)}
 \end{eqnarray*}
Consequently 
\begin{eqnarray*}
\eta(\eps) \| z_{\perp} \|_{H^1_{\eps}(\D)} &\lesssim&  \sup\limits_{0\not= (\sigma,v) \in \R \times T_{\ci u}\S } \frac{ |\calF(\sigma,v)| }{ \| (\sigma , v) \|_{\eps}  } \,\,\, =:\,\,\, ||| \calF |||_{\eps}^*
 \end{eqnarray*}
On the other hand, we also have with \eqref{est-for-alpha} that
\begin{eqnarray*}
\| \alpha u \|_{H^1_{\eps}(\D)} &\lesssim& \eps \, \| u \|_{H^1_{\eps}(\D)} \, ||| \calF |||_{\eps}^* \,\,\, \lesssim \,\,\,||| \calF |||_{\eps}^*.
 \end{eqnarray*}
 Hence, the triangle inequality finally yields
\begin{eqnarray*}
\| z \|_{H^1_{\eps}(\D)} 
&\le& \| \alpha u \|_{H^1_{\eps}(\D)} + \| z_{\perp} \|_{H^1_{\eps}(\D)}  \,\,\,\lesssim\,\,\, ( \eta(\eps)^{-1}  + 1 ) \,||| \calF |||_{\eps}^* \,\,\,\lesssim\,\,\, \eta(\eps)^{-1} ||| \calF |||_{\eps}^*. 
 \end{eqnarray*} 
For the missing bound for $\mu$, we test in \eqref{def-Jhiprimnv-prob} with $v=u$ and $\sigma=0$ to get
\begin{eqnarray*} 
|\mu|  &=&
\left| \langle (E''(u)-\lambda \mathcal{I} )z,u \rangle - \calF(0, u) \right| \\
&\overset{\eqref{continuity-constrained-Hessian}}{\lesssim}& 
\| z \|_{H^1_{\eps}(\D)} \, \| u \|_{H^1_{\eps}(\D)} 
+  ||| \calF |||_{\eps}^* \, \| u \|_{H^1_{\eps}(\D)} 
\,\,\, \lesssim \,\,\, \eps^{-1} (\| z \|_{H^1_{\eps}(\D)}  +  ||| \calF |||_{\eps}^*),
\end{eqnarray*}
With the previous estimate for $\| z \|_{H^1_{\eps}(\D)}$ we conclude
\begin{eqnarray*} 
\eps \, |\mu|  &\lesssim& \eta(\eps)^{-1} ||| \calF |||_{\eps}^*.
\end{eqnarray*}
Altogether, this establishes the desired estimate as
$$
\| (\mu,z)\|_{\eps} \,\,\,\lesssim \,\,\,\eps \, |\mu| + \| z \|_{H^1_{\eps}(\D)} \,\,\,\lesssim \,\,\,\eta(\eps)^{-1} ||| \calF |||_{\eps}^*.
$$
To prove the final identity \eqref{projection-error-rep}, note that $\langle \calJ_h( \lambda , u ) , (\tau,w) \rangle = \langle (E^{\prime\prime}(u) - \lambda \calI) u , w - P_h w \rangle$. Consequently, $(\mu ,z):=\calJ_h'(\lambda,u)^{-1}  \calJ_h( \lambda , u )  \in \R \times T_{\ci u} \S$ solves
\begin{align*}
\big\langle (E''(u)-\lambda \calI)z, w\big\rangle \,-\, \mu \big\langle \calI u , w\big\rangle 
\,-\, \tau \langle \calI u ,z \rangle \,\,=\,\,   \langle (E^{\prime\prime}(u) - \lambda \calI) u , w - P_h w \rangle
\end{align*}
for all $(\tau,w) \in \R \times T_{\ci u} \S$. Since the solution is unique, it is sufficient to verify that $(\mu,z)=(0,u - P_h u )$ fulfills the equation. Indeed, with $u-P_hu \in H_{u}\S$ we have
\begin{eqnarray*}
\lefteqn{ \big\langle (E''(u)-\lambda \calI)(u-P_hu), w\big\rangle \,-\, 0 \, \big\langle \calI u , w\big\rangle 
\,-\, \tau \langle \calI u , u-P_hu \rangle } \\
&=&  \big\langle (E''(u)-\lambda \calI)(u-P_hu), w\big\rangle
\,\,\, \overset{P_hw \in T_{\ci u} \S}{=} \,\,\,  \big\langle (E''(u)-\lambda \calI)(u-P_hu), w- P_h w\big\rangle \\
&\overset{P_hu \in T_{\ci u} \S}{=}&  \big\langle \langle (E''(u)-\lambda \calI) u, w- P_h w\big\rangle, 
\end{eqnarray*}
 where we also used that $E''(u)-\lambda \calI$ is symmetric. Consequently, $(0,u - P_h u )$ solves the defining equation for $\calJ_h'(\lambda,u)^{-1}  \calJ_h( \lambda , u )$.
\end{proof}
Next, we give a Taylor expansion of $\calJ_h'(\mu,z)$ around $(\lambda,u)$. For that note that $\calJ_h''(\lambda,u)$ can be computed as 
\begin{eqnarray}
\label{calJsec}
\lefteqn{ \big\langle \calJ_h''(\lambda,u)\big[(\sigma_1,v_1),(\sigma_2,v_2)\big],(\tau,w)\big\rangle }\\
\nonumber&=&
\big\langle E'''(u)(v_1,v_2),\, P_h w \big\rangle 
- \sigma_1 \big\langle \calI v_2,\, P_h w \big\rangle
- \sigma_2 \big\langle \calI v_1,\, P_h w \big\rangle 
- \tau \big\langle \calI v_1,\, v_2 \big\rangle
\end{eqnarray}
for $v_1,v_2,w\in T_{\ci u} \S$ and $\sigma_1,\sigma_2,\tau \in \R$, where 
\begin{align*}
\big\langle E'''(u)(v_1,v_2),\, P_h w \big\rangle
&= \frac{\beta}{\eps^2} \int_{\D}
\Big(
  2 \Re\big(u \overline{v_2}\big)\, v_1
+ 2 \Re\big(u \overline{v_1}\big)\, v_2
+ 2 \Re\big(v_1 \overline{v_2}\big)\, u
\Big)\,
\overline{P_h w}\,\dx.
\end{align*}
With this, we also obtain that $ \calJ_h^{\prime\prime\prime}(\lambda,u)$ is constant (in $(\lambda,u)$) with
\begin{eqnarray}
\label{calJthird}
\lefteqn{ \big\langle \calJ_h^{\prime\prime\prime}(\lambda,u)
\big[(\sigma_1,v_1),(\sigma_2,v_2),(\sigma_3,v_3)\big],(\tau,w)\big\rangle }\\
\nonumber&=&
\frac{\beta}{\eps^2}
(
  2 \Re(v_3 \overline{v_2})\, v_1
+ 2 \Re(v_3 \overline{v_1})\, v_2
+ 2 \Re(v_1 \overline{v_2})\, v_3
\,,\, P_h w
)_{L^2(\D)}.
\end{eqnarray}
Consequently, the fourth derivative of $ \calJ_h$ needs to vanish, i.e. $ \calJ_h^{\prime\prime\prime\prime}(\lambda,u)=0$. Hence, by Taylor expansion of Fr\'echet differentiable operators and using that $ \calJ_h^{\prime\prime\prime\prime}\equiv 0$, we obtain for any $(\mu,z)\in\R \times T_{\ci u}\S$
\begin{eqnarray}
\label{Taylor:Jprime}
\nonumber
\lefteqn{ \big\langle \left(\calJ_h^{\prime}(\mu,z) -  \calJ_h^{\prime}(\lambda,u) \right)(\sigma,v),(\tau,w) \big\rangle \,\,\, = \,\,\, \big\langle \calJ_h^{\prime\prime}(\lambda,u) (\mu-\lambda,z-u) (\sigma,v),(\tau,w) \big\rangle }\\
&\enspace& + \,\,
\tfrac{1}{2} \, \big\langle \calJ_h^{\prime\prime\prime}(\lambda,u) [(\mu-\lambda,z-u),(\mu-\lambda,z-u)](\sigma,v),(\tau,w) \big\rangle,  \hspace{80pt}
\end{eqnarray}
for $(\sigma,v),(\tau,w) \in \R \times T_{\ci u}\S$. With this, the following continuity estimate can be proved.
\begin{lemma}\label{lem:Taylor:Jprime-est}
Assume \ref{A1}-\ref{A3} and $h \lesssim \eps \, \eta(\eps)$, let $u \in \S$ be a quasi-isolated ground state with eigenvalue $\lambda\in\R$ and let $\calJ_h$ be as in \eqref{def-calJ-h}. For arbitrary $(\mu,z) \in \R \times T_{\ci u}\S$, it holds 
\begin{eqnarray*}
\lefteqn{ | \big\langle \left(\calJ_h^{\prime}(\mu,z) -  \calJ_h^{\prime}(\lambda,u) \right)(\sigma,v),(\tau,w) \big\rangle  | }\\
&\lesssim&
\eps \, |\mu-\lambda| \, \|v \|_{L^2(\D)} \| w\|_{H^1_{\eps}(\D)} + \eps \,|\sigma | \, \|z-u\|_{L^2(\D)}  \| w \|_{H^1_{\eps}(\D)}  + \eps \, |\tau|  \| z-u \|_{L^2(\D)}  \tfrac{1}{\eps} \| v \|_{L^2(\D)} \\
&\enspace& \quad + \,\tfrac{1}{\eps} \, \| v \|_{L^4(\D)} \left( \| z-u \|_{L^4(\D)} + \tfrac{1}{\eps^{d/4}} \| z-u \|_{L^4(\D)}^2 \right)  \| w \|_{H^1_{\eps}(\D)}.
\end{eqnarray*}
\end{lemma}

\begin{proof}
We use the Taylor expansion \eqref{Taylor:Jprime} and estimate the first term with \eqref{calJsec} to obtain
\begin{eqnarray*}
\lefteqn{ |\big\langle \calJ_h^{\prime\prime}(\lambda,u) (\mu-\lambda,z-u) (\sigma,v),(\tau,w) \big\rangle | }\\
&\lesssim& |\mu-\lambda| \, \|v \|_{L^2(\D)} \| P_h w\|_{L^2(\D)} + |\sigma | \, \|z-u\|_{L^2(\D)}  \| P_h w \|_{L^2(\D)}  + |\tau|  \| z-u \|_{L^2(\D)}  \| v \|_{L^2(\D)} \\
&\enspace& \quad + \,\tfrac{2\beta}{\eps^2} \, | 
(\Re\big(u \overline{v}\big)\, (z-u) + \Re\big(u \overline{(z-u)}\big)\, v + \Re\big((z-u) \overline{v}\big)\, u
, P_h w )_{L^2(\D)} | \\
&\overset{\eqref{Ph-Honeeps-stability}}{\lesssim}&
\eps \, |\mu-\lambda| \, \|v \|_{L^2(\D)} \| w\|_{H^1_{\eps}(\D)} + \eps \,|\sigma | \, \|z-u\|_{L^2(\D)}  \| w \|_{H^1_{\eps}(\D)}  + \eps \, |\tau|  \| z-u \|_{L^2(\D)}  \tfrac{1}{\eps} \| v \|_{L^2(\D)} \\
&\enspace& \quad + \,\tfrac{1}{\eps^2} \, | 
(\Re\big(u \overline{v}\big)\, (z-u) + \Re\big(u \overline{(z-u)}\big)\, v + \Re\big((z-u) \overline{v}\big)\, u
, P_h w )_{L^2(\D)} | \\
&\overset{\| u \|_{L^{\infty}}\lesssim 1 }{\lesssim}&
\eps \, |\mu-\lambda| \, \|v \|_{L^2(\D)} \| w\|_{H^1_{\eps}(\D)} + \eps \,|\sigma | \, \|z-u\|_{L^2(\D)}  \| w \|_{H^1_{\eps}(\D)}  + \eps \, |\tau|  \| z-u \|_{L^2(\D)}  \tfrac{1}{\eps} \| v \|_{L^2(\D)} \\
&\enspace& \quad + \,\tfrac{1}{\eps^2} \, \| v \|_{L^4(\D)} \| z-u \|_{L^4(\D)}  \| P_h w \|_{L^2(\D)} \\
&\overset{\eqref{Ph-Honeeps-stability}}{\lesssim}&
\eps \, |\mu-\lambda| \, \|v \|_{L^2(\D)} \| w\|_{H^1_{\eps}(\D)} + \eps \,|\sigma | \, \|z-u\|_{L^2(\D)}  \| w \|_{H^1_{\eps}(\D)}  + \eps \, |\tau|  \| z-u \|_{L^2(\D)}  \tfrac{1}{\eps} \| v \|_{L^2(\D)} \\
&\enspace& \quad + \,\tfrac{1}{\eps} \, \| v \|_{L^4(\D)} \| z-u \|_{L^4(\D)}  \| w \|_{H^1_{\eps}(\D)}.
\end{eqnarray*}
For estimating the second term in \eqref{Taylor:Jprime}, we recall the Gagliardo--Nirenberg estimate $\| v\|_{L^4(\D)} \lesssim \| v\|_{L^2(\D)}^{1-d/4} \| \nabla v \|_{L^2(\D)}^{d/4}$ (for $d=2,3$) which implies 
\begin{align}
\label{gag-ni-L4}
\| v\|_{L^4(\D)} \lesssim \eps^{1-d/4} \| v \|_{H^1_{\eps}(\D)} \qquad \mbox{for any } v\in H^1_0(\D).
\end{align}
Similarly
\begin{align}
\label{gag-ni-L3}
\| v\|_{L^3(\D)} \lesssim \eps^{1-d/6} \| v \|_{H^1_{\eps}(\D)} \qquad \mbox{for any } v\in H^1_0(\D).
\end{align}
We are now ready to estimate the latter term in \eqref{Taylor:Jprime} with identity \eqref{calJthird} to obtain
\begin{eqnarray*}
\lefteqn{ | \big\langle \calJ_h^{\prime\prime\prime}(\lambda,u) [(\mu-\lambda,z-u),(\mu-\lambda,z-u)](\sigma,v),(\tau,w) \big\rangle  | 
\,\,\,\lesssim\,\,\, \tfrac{1}{\eps^2} \int_{\D} |v|\, |z-u|^2 |P_hw| \dx }\\
&\lesssim&  \tfrac{1}{\eps^2} \| v \|_{L^4(\D)} \| z-u \|_{L^4(\D)}^2 \| P_h w \|_{L^4(\D)} 
\,\,\,\overset{\eqref{gag-ni-L4}}{\lesssim}\,\,\, \eps^{-1-d/4} \| v \|_{L^4(\D)} \| z-u \|_{L^4(\D)}^2 \| w \|_{H^1_{\eps}(\D)}.
\end{eqnarray*}
Combining the estimates for the two terms in \eqref{Taylor:Jprime} finishes proof.
\end{proof}

\begin{theorem}\label{thrm:loc-existence:u_h}
Assume \ref{A1}-\ref{A3} and let $u \in \S$ be a quasi-isolated ground state with eigenvalue $\lambda$. Furthermore, assume that $h \lesssim \eps \,\eta(\eps)$. Then there exist constants $M^\ast>0$ and $c^\ast>0$, both independent of
$h$ and $\eps$, such that for every $M\in(0,M^\ast]$ the following holds.

If the mesh size satisfies $h \le c^\ast\,\eps\,\eta(\eps)\,\eps^{d/2} M$, 
then there exists a discrete critical point
$(\lambda_h,u_h)\in \R\times (V_h\cap \S\cap T_{\ci u}\S)$
satisfying
\begin{align*}
\langle E'(u_h),v_h\rangle
=
\lambda_h\,(u_h,v_h)_{L^2(\D)}
\qquad
\mbox{for all } v_h\in V_h,
\end{align*}
and
\begin{align}
\label{abstract-bounds}
\nonumber\| (\lambda,u) - (\lambda_h,u_h) \|_{\eps} &\le \tfrac{\eta(\eps)}{\eps} \eps^{d/4} M, \qquad  \| u- u_h \|_{L^2(\D)} \le \eta(\eps) \eps^{d/2} M,\\
 \| u- u_h \|_{L^4(\D)} &\le \eta(\eps) \eps^{d/4} M.
\end{align}
Moreover, $(\lambda_h,u_h)$ is the unique discrete critical point satisfying
\eqref{abstract-bounds}.

Furthermore, $u_h$ satisfies the discrete sufficient second-order condition
\begin{align*}
\langle (E''(u_h)-\lambda_h\mathcal I)v_h,v_h\rangle
\gtrsim
\eta(\eps)\,\|v_h\|_{H^1_\eps(\D)}^2
\qquad
\mbox{for all } v_h\in V_h\cap H_u\S.
\end{align*}
In particular, $u_h$ is a discrete local minimizer.
\end{theorem}

\begin{proof}
On $\R \times T_{\ci u} \mathbb{S}$ we consider the fixed-point map
\begin{align*}
F(\mu,z) \,\,:=\,\, (\mu,z) - \calJ_h'(\lambda,u)^{-1}  \calJ_h( \mu , z )  \,\,=\,\, \calJ_h'(\lambda,u)^{-1}  \left( \calJ_h'(\lambda,u)(\mu,z) -   \calJ_h( \mu , z ) \right).
\end{align*}
Apparently, any fixed point $(\mu,z) \in \R \times T_{\ci u} \mathbb{S}$ of $F$ fulfills  $\calJ_h'(\lambda,u)^{-1}  \calJ_h( \mu , z ) =0$ and since $\calJ_h'(\lambda,u)^{-1}$ has a trivial kernel, it must necessarily hold $\calJ_h( \mu , z ) =0$. By Lemma \ref{lem:equiv-char-disc-gs}, $( \mu , z )$ is must be a discrete critical point in $V_h \cap T_{\ci u} \S$. Hence, we need to prove the existence of a fixed point, for which we use the Banach fixed-point theorem. 

Given $0<M\le1$ (to be fixed later), we consider $F$ on the following closed subset of $ \R \times T_{\ci u} \mathbb{S}$:
\begin{align*}
K_{\eps}^M \,&:=\, \{ (\sigma,v) \in \R \times T_{\ci u} \mathbb{S} \,\,| \,\, \| (\lambda,u) - (\sigma,v) \|_{\eps} \le \tfrac{\eta(\eps)}{\eps} \eps^{d/4} M, \,\,\,\,   \| u- v \|_{L^2(\D)} \le \eta(\eps) \eps^{d/2} M \\
&\hspace{140pt} \mbox{and} \,\, \| u-v \|_{L^4(\D)} \le \eta(\eps) \eps^{d/4} M \, \}
\end{align*}
and show that $F : K_{\eps}^M \rightarrow K_{\eps}^M$ is a contraction. 

Note that for any $(\mu_1,z_1),(\mu_2,z_2)   \in \R \times T_{\ci u} \mathbb{S}$ we have 
\begin{eqnarray}
\label{F-diff-representation}
\lefteqn{ F(\mu_1,z_1) - F(\mu_2,z_2) } \\
\nonumber &=&  \int_0^1 \calJ_h'(\lambda,u)^{-1}  \left(  \left( \calJ_h'(\lambda,u) -   \calJ_h^{\prime}( \, s\,(\mu_1,z_1) + (1-s)\,(\mu_2,z_2) \, ) \right) [(\mu_1,z_1)-(\mu_2,z_2)]  \right) \,\mbox{d}s.
\end{eqnarray}
We will exploit this identity in the proof.\\[0.4em]
{\it Step 1: Contraction property.} Consider $(\mu_1,z_1),(\mu_2,z_2) \in \R \times T_{\ci u} \mathbb{S}$, then \eqref{F-diff-representation} yields in combination with Lemma \ref{lem:bound:Jprimeinv} that
\begin{eqnarray*}
\lefteqn{ \eta(\eps) \, \| F(\mu_1,z_1) - F(\mu_2,z_2) \|_{\eps} } \\
&\lesssim& \sup_{s\in [0,1]}  \underset{\|(\sigma , v) \|_{\eps}=1}{\sup\limits_{(\sigma,v) \in \R \times T_{\ci u}\S }}  \langle \left( \calJ_h'(\lambda,u) -   \calJ_h^{\prime}(s(\mu_1,z_1) + (1-s)(\mu_2,z_2) ) \right) [(\mu_1,z_1)-(\mu_2,z_2)] , (\sigma , v)\rangle .
\end{eqnarray*}
Applying Lemma \ref{lem:Taylor:Jprime-est} and recalling $\| (\sigma , v) \|_{\eps} = \sqrt{ \eps^{2} |\sigma|^2 + \|v\|_{H^1_{\eps}(\D)}^2 }$, we obtain
\begin{eqnarray*}
\lefteqn{  \| F(\mu_1,z_1) - F(\mu_2,z_2) \|_{\eps} } \\
&\lesssim& \tfrac{\eps}{\eta(\eps)} \, \| z_1-z_2 \|_{L^2(\D)}  \, \sum_{j=1}^2 |\mu_j- \lambda| \, + \tfrac{\eps}{\eta(\eps)} \,|\mu_1-\mu_2 | \, \sum_{j=1}^2  \| z_j-u \|_{L^2(\D)}  \\
&\enspace& \quad \,+  \tfrac{1}{\eps \eta(\eps)} \| z_1-z_2 \|_{L^2(\D)} \sum_{j=1}^2  \| z_j-u \|_{L^2(\D)} \\
&\enspace& \quad + \,\tfrac{1}{\eps\eta(\eps)} \, \| z_1-z_2 \|_{L^4(\D)}  \sum_{j=1}^2  \left( \| z_j - u \|_{L^4(\D)}  + \tfrac{1}{\eps^{d/4}} \|  z_j -u \|_{L^4(\D)}^2 \right) \\
&\overset{(\mu_j,z_j)\in K_{\eps}^M }{\lesssim}& M\, \eps^{d/4} \left( \tfrac{1}{\eps} \| z_1-z_2 \|_{L^2(\D)} +\eps \,|\mu_1-\mu_2 |\,\right) + \,\tfrac{1}{\eps} \, \| z_1-z_2 \|_{L^4(\D)} \left( \eps^{d/4} M + \eps^{d/4} \eta(\eps) M^2 \right) \\
&\overset{\eqref{gag-ni-L4}}{\lesssim}& 
 M\, \eps^{d/4} \left(  \tfrac{1}{\eps} \| z_1-z_2 \|_{L^2(\D)} +\eps \,|\mu_1-\mu_2 |\,\right) + \eps^{-d/4} \, \| z_1-z_2 \|_{H^1_{\eps}(\D)} \left( \eps^{d/4} M + \eps^{d/4}  \eta(\eps) M^2 \right) \\
&\lesssim&
  M\,  \| (z_1,\mu_1)-(z_2,\mu_2) \|_{\eps} + \, M^2 \, \eta(\eps) \, \| z_1-z_2 \|_{H^1_{\eps}(\D)}. 
\end{eqnarray*}
Since $M \,\eta(\eps) \le 1$, we have $M^2 \, \eta(\eps) \| z_1-z_2 \|_{H^1_{\eps}(\D)} \le M\, \| (z_1,\mu_1)-(z_2,\mu_2) \|_{\eps} $ and therefore 
\begin{eqnarray*}
 \| F(\mu_1,z_1) - F(\mu_2,z_2) \|_{\eps} 
&\lesssim& 
M\, \| (z_1,\mu_1)-(z_2,\mu_2) \|_{\eps}.
\end{eqnarray*}
We conclude that if $M \,\le \,c$ for a sufficiently small (generic) constant $c>0$, then $F$ is indeed a contraction, i.e. $ \| F(\mu_1,z_1) - F(\mu_2,z_2) \|_{\eps}  \le L \, \| (z_1,\mu_1)-(z_2,\mu_2) \|_{\eps}$ for some $L<1$.\\[0.4em]
{\it Step 2: Invariance of $K_{\eps}^M$.} Next, we show that $F(\mu,z) \in K_{\eps}^M$ for any $(\mu,z) \in K_{\eps}^M$. For that we need to bound $\| F(\mu,z) - (\lambda,u)\|_{\eps}$ and $\| F(\mu,z)_2 - u \|_{L^4(\D)}$, where $F(\mu,z)_2$ is the $T_{\ci u}\S$-component of the tuple $F(\mu,z)$. 

Using \eqref{F-diff-representation} for $(\mu_1,z_1)=(\lambda,u)$ and $(\mu,z) \in K_{\eps}^M$ we obtain
\begin{eqnarray}
\label{eq:contraction-equality}
\nonumber\lefteqn{ (\lambda,u) - F(\mu,z) \,\,\, = \,\,\, (\lambda,u) - F(\lambda,u) \,\, + \,\, F(\lambda,u)  - F(\mu,z)  } \\
 &=&  (\lambda,u) - F(\lambda,u) \\
\nonumber&\enspace& \quad+ \int_0^1 \calJ_h'(\lambda,u)^{-1}  \left(  \left( \calJ_h'(\lambda,u) -   \calJ_h^{\prime}( \, s\,(\lambda,u) + (1-s)\,(\mu,z) \, ) \right) [(\lambda,u)-(\mu,z)]  \right) \,\mbox{d}s.
\end{eqnarray}
For the first term, we directly have
\begin{align}
\label{eq:contraction-equality-2}
(\lambda,u)- F(\lambda,u) \,\,=\,\, \calJ_h'(\lambda,u)^{-1}  \calJ_h( \lambda , u ) \,\,\overset{\eqref{projection-error-rep}}{=} \,\,(0, u - P_h u )
\end{align}
and the second term is again controlled by  Lemma \ref{lem:bound:Jprimeinv} (estimate for $\calJ_h'(\lambda,u)^{-1}$) and Lemma \ref{lem:Taylor:Jprime-est} (Lipschitz-estimate for $ \calJ_h'(\lambda,u)$). We obtain from \eqref{eq:contraction-equality} and \eqref{eq:contraction-equality-2} that
\begin{eqnarray}
\label{selbstabbildung-est-1}
\nonumber\lefteqn{ \| (\lambda,u) - F(\mu,z) \|_{\eps}  \,\,\, \lesssim \,\,\, \| u - P_h u \|_{H^1_{\eps}(\D)} } \\
\nonumber&+& \eta(\eps)^{-1} \sup_{s\in [0,1]}  \underset{\|(\sigma , v) \|_{\eps}=1}{\sup\limits_{(\sigma,v) \in \R \times T_{\ci u}\S }}  \langle \left( \calJ_h'(\lambda,u) -   \calJ_h^{\prime}(s(\lambda,u) \!+\! (1-s)(\mu,z) ) \right) (\lambda-\mu,u-z) , (\sigma , v)\rangle \\
\nonumber&\overset{\mbox{\tiny Lem.~\ref{lem:Taylor:Jprime-est}}}{\lesssim}&
\| u - P_h u \|_{H^1_{\eps}(\D)} + \tfrac{\eps}{\eta(\eps)} \, \| u-z \|_{L^2(\D)}  \, | \lambda - \mu | \,+ \, \tfrac{1}{\eps\eta(\eps)} \| u-z \|_{L^2(\D)}^2 \\
\nonumber&\enspace& \quad + \,\tfrac{1}{\eps \eta(\eps)} \, \left( \| u - z \|_{L^4(\D)}^2 +  \tfrac{1}{\eps^{d/4}} \|  u - z \|_{L^4(\D)}^3 \right) \\
\nonumber&\overset{(\mu,z)\in K_{\eps}^M }{\lesssim}&
\| u - P_h u \|_{H^1_{\eps}(\D)} + \tfrac{\eta(\eps)}{\eps} \, \eps^{d/2}\, M^2  + \,\tfrac{1}{\eps \eta(\eps)} \, \left( \eta(\eps)^2 \eps^{d/2} M^2 +  \eta(\eps)^3 \eps^{d/2} M^3 \right) \\
&\lesssim& \| u - P_h u \|_{H^1_{\eps}(\D)} +  \tfrac{\eta(\eps)}{\eps} \,\eps^{d/2}  M^2 
\,\,\, \overset{\eqref{est:H1-ritz-proj-est-gs}}{\lesssim} \,\,\, \tfrac{h}{\eps^2}  +  \tfrac{\eta(\eps)}{\eps} \,\eps^{d/2}  M^2,
\end{eqnarray}
where we used the smallness of $M$, $\eps$ and $\eta(\eps)$. So if we write the estimate as
\begin{eqnarray*}
\| (\lambda,u) - F(\mu,z) \|_{\eps}  &\le& C\, \left( \tfrac{h}{\eps^2}  + \tfrac{\eta(\eps)}{\eps} \, \eps^{d/2} M^2 \right)
\end{eqnarray*}
for some generic constant $C$, then we require $C\, \left( \tfrac{h}{\eps^2}  + \tfrac{\eta(\eps)}{\eps} \, \eps^{d/2} M^2 \right) \le \tfrac{\eta(\eps)}{\eps} \eps^{d/4} M$ such that $F(\mu,z)$ can be an element of $K_{\eps}^M$. This requires $M \le C_{\mbox{\tiny res},1}$ for, e.g., $C_{\mbox{\tiny res},1} \le \tfrac{1}{2} C^{-1}$. On the other hand, we also need the mesh size to be small enough that $C\, \tfrac{h}{\eps^2} \le \tfrac{1}{2} \tfrac{\eta(\eps)}{\eps} \eps^{d/4} M$. Hence, we also require the resolution condition $h \le \tfrac{1}{2C} \eta(\eps)\,\eps^{1+d/4}  \,M$ for any given $M$.

Next, we verify the $L^2$-property, i.e. $\| u-F(\mu,z)_2 \|_{L^2(\D)} \le \eta(\eps) \eps^{d/2} M$ where $F(\mu,z) = (F(\mu,z)_1, F(\mu,z)_2) \in \R \times T_{\ci u}\S$ denotes the second component of $F(z,u)$. Here we use $\| \| u-F(\mu,z)_2 \|_{L^2(\D)} \le   \eps \| u-F(\mu,z)_2 \|_{H^1_{\eps}(\D)}$ together with the previous estimate to obtain
\begin{eqnarray*}
\| u-F(\mu,z)_2 \|_{L^2(\D)} &\le&  \eps \| u-F(\mu,z)_2 \|_{H^1_{\eps}(\D)} \,\,\, \le \,\,\, \eps \| u-F(\mu,z) \|_{\eps} \\
&\overset{\eqref{selbstabbildung-est-1}}{\lesssim}& 
\eps 
( \tfrac{h}{\eps^2}  +  \tfrac{\eta(\eps)}{\eps} \,\eps^{d/2}  M^2) 
\,\,\, \lesssim \,\,\, \tfrac{h}{\eps}+ \eta(\eps) \,\eps^{d/2}  M^2.
\end{eqnarray*}
Again, if $C>0$ is the hidden constant such that
\begin{eqnarray*}
\| u-F(\mu,z)_2 \|_{L^2(\D)} &\le& C (\tfrac{h}{\eps}+ \eta(\eps) \,\eps^{d/2}  M^2),
\end{eqnarray*}
we require $M \le \tfrac{1}{2}C^{-1}$ and $h$ such that $C \tfrac{h}{\eps} \le \tfrac{1}{2} \eta(\eps) \eps^{d/2} M$ to ensure that $\| u-F(\mu,z)_2 \|_{L^2(\D)} \le \eta(\eps) \eps^{d/2} M$. Hence, we again recover the condition $h \lesssim \eps\, \eta(\eps) \eps^{d/2} M$.

It remains to verify the $L^4$-property, i.e. $\| u-F(\mu,z)_2 \|_{L^4(\D)} \le \eta(\eps) \eps^{d/4} M$. For this, we can again use the Gagliardo--Nirenberg estimate \eqref{gag-ni-L4} to obtain
\begin{eqnarray*}
\| u-F(\mu,z)_2 \|_{L^4(\D)} &\lesssim&  \eps^{1-d/4} \| u-F(\mu,z)_2 \|_{H^1_{\eps}(\D)} \,\,\, \lesssim \,\,\, \eps^{1-d/4} \| u-F(\mu,z) \|_{\eps} \\
&\overset{\eqref{selbstabbildung-est-1}}{\lesssim}& 
\eps^{1-d/4} 
( \tfrac{h}{\eps^2}  +  \tfrac{\eta(\eps)}{\eps} \,\eps^{d/2}  M^2) 
\,\,\, \lesssim \,\,\,  \eps^{-1-d/4} h + \eta(\eps) \,\eps^{d/4}  M^2.
\end{eqnarray*}
Let once again $C>0$ denote the constant with
\begin{eqnarray*}
\| u-F(\mu,z)_2 \|_{L^4(\D)} &\le& C (\eps^{-1-d/4} h + \eta(\eps) \,\eps^{d/4}  M^2),
\end{eqnarray*}
then $M \le \tfrac{1}{2}C^{-1}$ and $h$ with $C \eps^{-1-d/4} h \le \tfrac{1}{2} \eta(\eps) \eps^{d/4} M$ yield $\| u-F(\mu,z)_2 \|_{L^4(\D)} ]\le \eta(\eps) \eps^{d/4} M$ and therefore $F(\mu,z) \in K_{\eps}^M$ as desired. Note that the condition on $h$ effectively requires $ h \,\lesssim \, \eps \, \eta(\eps) \, \eps^{d/2} M$ again.\\[0.4em]
{\it Step 3: Existence of discrete minimizer.} Since $F : K_{\eps}^M \rightarrow K_{\eps}^M$ fulfills all requirements of the Banach fixed point theorem (provided that $ h \,\lesssim \, \eps \, \eta(\eps) \, \eps^{d/2} M$), we conclude the existence of a unique fixed point $(\lambda_h,u_h) \in K_{\eps}^M$ with $F(\lambda_h,u_h) =(\lambda_h,u_h) $. By construction of $F$, this implies $\calJ_h(\lambda_h,u_h)=0$. Hence, we can apply  Lemma \ref{lem:equiv-char-disc-gs} to conclude that $u_h \in V_h \cap T_{\ci u}\S$ and that it holds
\begin{align*}
\langle E^{\prime}(u_h) , v_h \rangle = \lambda_h ( u_h , v_h )_{L^2(\D)} \qquad \mbox{for all } v_h \in V_h \cap T_{\ci u}\S.
\end{align*}
It remains to verify that the equation holds for all $v_h \in V_h$ (i.e. without the restriction to $T_{\ci u}\S$). Here we can use the invariance of $E$ under global phase shifts, which implies $\langle E^{\prime}(u_h) , \ci u_h \rangle = 0$, as well as $\lambda_h ( u_h , \ci u_h )_{L^2(\D)}=0$. This implies
 \begin{align}
 \label{disc-first-ord-cond-proof}
\langle E^{\prime}(u_h) , v_h \rangle = \lambda_h ( u_h , v_h )_{L^2(\D)}
\qquad \mbox{for all } v_h \in \left( V_h \cap  T_{\ci u} \S\right) \oplus \mbox{span}_{\R}\{ \ci u_h \} .  
 \end{align}
Since $V_h \cap T_{\ci u}\S$ is a real codim-1 subspace of $V_h$, it suffices to show that $\ci u_h \notin V_h \cap T_{\ci u}\S$. Indeed, in this case
$\ci u_h$ provides the missing linearly independent direction and therefore
$(V_h \cap T_{\ci u}\S)\oplus \operatorname{span}_{\mathbb R}\{\ci u_h\} = V_h$. In fact, using $u\in\S$ and $u_h\in K_{\eps}^M$ we have
\begin{align*}
\left| \Re \int_{\D} (\ci u) \,\overline{(\ci u_h)} \dx \right| &= \left| 1 - \Re \int_{\D} u \overline{(u_h -u)} \right| \\
&\ge 1 - \| u \|_{L^2(\D)} \| u_h - u \|_{L^2(\D)}
\ge 1 - \eta(\eps) \eps^{d/2} M > 0,
\end{align*}
by smallness of $\eps$ and $M$. Hence $\ci u_h \in V_h$, but $\ci u_h \not \in T_{\ci u} \S$. We conclude $\left( V_h \cap  T_{\ci u} \S\right) \oplus \mbox{span}_{\R}\{ \ci u_h \} = V_h $ and therefore with \eqref{disc-first-ord-cond-proof} $\langle E^{\prime}(u_h) , v_h \rangle = \lambda_h ( u_h , v_h )_{L^2(\D)}$  for all $v_h \in V_h$.\\[0.4em]
{\it Step 4: Sufficient second-order condition.}  To check the sufficient second-order condition, let $v_h \in V_h \cap H_{u}\S$ be arbitrary. We obtain from the strong coercivity \eqref{eq:secE-strong-coercive} that
\begin{eqnarray}
\label{est-disc-sec-order-cond}
\nonumber\lefteqn{ \langle (E^{\prime\prime}(u_h)  -\lambda \mathcal{I}) v_h ,v_h \rangle 
\,\,\,=\,\,\,  \langle (E^{\prime\prime}(u_h)  -E^{\prime\prime}(u))  v_h ,v_h \rangle  + \langle (E^{\prime\prime}(u)  -\lambda \mathcal{I}) v_h ,v_h \rangle } \\
&\ge& \eta(\eps) \| v_h \|_{H^1_{\eps}(\D)}^2  \,\,-\,\,  \left|  \langle (E^{\prime\prime}(u_h)  -E^{\prime\prime}(u))  v_h ,v_h \rangle \right|\\
\nonumber&=&  \eta(\eps) \| v_h \|_{H^1_{\eps}(\D)}^2  \,\,-\,\,  \left|  \langle E^{\prime\prime\prime}(u)(u_h-u) v_h ,v_h \rangle + \tfrac{1}{2} \langle E^{\prime\prime\prime\prime}(u)(u_h-u,u_h-u) v_h ,v_h \rangle  \right|.
\end{eqnarray}
Since
\begin{align*}
\langle E^{\prime\prime\prime}(u)(u_h-u)\, v_h , v_h \rangle
= \tfrac{\beta}{\eps^2} 
\int_{\D} 
\Big(
2\,\Re\big( u\,(u_h-u)\,\overline{v_h}^{\,2} \big)
+
4\,\Re\big( u\,\overline{(u_h-u)}\big)\,|v_h|^2
\Big)\dx 
\end{align*}
and
\begin{align*}
\tfrac{1}{2} \langle E^{\prime\prime\prime\prime}(u)(u_h-u,u_h-u)\, v_h , v_h \rangle
= \tfrac{\beta}{\eps^2}
\int_{\D} 
\Big(
\Re\big( (u_h-u)^{2}\,\overline{v_h}^{\,2}\big)
+
2\,|u_h-u|^2\,|v_h|^2
\Big)\dx 
\end{align*}
we have, with $\| u \|_{L^{\infty}(\D)} \lesssim 1$, that
\begin{eqnarray*}
|\langle E^{\prime\prime\prime}(u)(u_h-u)\, v_h , v_h \rangle|
&\lesssim& \tfrac{1}{\eps^2} \| u_h - u\|_{L^2(\D)} \| v_h \|_{L^4(\D)}^2 \\
&\overset{\eqref{gag-ni-L4}}{\lesssim}&   \eps^{-d/2} \| u_h - u\|_{L^2(\D)} \| v_h \|_{H^1_{\eps}(\D)}^2 
\,\,\, \overset{u_h \in K_{\eps}^M}{\lesssim} \,\,\,\eta(\eps) \, M\,  \| v_h \|_{H^1_{\eps}(\D)}^2 
\end{eqnarray*}
and similarly
\begin{eqnarray*}
\lefteqn{ |\langle E^{\prime\prime\prime\prime}(u)(u_h-u,u_h-u)\, v_h , v_h \rangle |
\,\,\,\lesssim\,\,\, \tfrac{1}{\eps^2} \| u_h-u\|_{L^4(\D)}^{2}\, \| v_h\|_{L^4(\D)}^{2} } \\
&\overset{\eqref{gag-ni-L4}}{\lesssim}&   \eps^{-d/2}\| u_h-u\|_{L^4(\D)}^{2}\, \| v_h \|_{H^1_{\eps}(\D)}^2 
\,\,\, \overset{u_h \in K_{\eps}^M}{\lesssim} \,\,\,  \eta(\eps)^2 M^2 \, \| v_h \|_{H^1_{\eps}(\D)}^2 
\end{eqnarray*}
Combining the last two estimates with \eqref{est-disc-sec-order-cond} yields
\begin{eqnarray}
\label{preli-est-uh-loc-min}
 \langle (E^{\prime\prime}(u_h)  -\lambda \mathcal{I}) v_h ,v_h \rangle 
&\ge& \eta(\eps) \, (1 - c_1  M - c_2 \eta(\eps) M^2 )\| v_h \|_{H^1_{\eps}(\D)}^2 
\end{eqnarray}
for some generic constants $c_1,c_2>0$. Hence, if $M\le c$ for a sufficiently small constant $c$ then $(1 - c_1 M - c_2 \eta(\eps) M^2 )\gtrsim 1$. It remains to verify that we can replace $\lambda$ by $\lambda_h$. Here we can use 
\begin{align*}
\eps\, |\lambda-\lambda_h| \,\,\le \,\, \| (\lambda,u) - (\lambda_h,u_h) \|_{\eps} \,\, \le \,\,  \tfrac{\eta(\eps)}{\eps} \eps^{d/4} M
\end{align*}
to conclude
\begin{align*}
|\lambda-\lambda_h| \, \| v_h \|_{L^2(\D)}^2  
\,\, \le\,\,   \eps^2 |\lambda-\lambda_h| \, \| v_h \|_{H^1_{\eps}(\D)}^2 
\,\, \le \,\, \eta(\eps) \eps^{d/4} M \, \| v_h \|_{H^1_{\eps}(\D)}^2 .
\end{align*}
Combining this with \eqref{preli-est-uh-loc-min} proves the final estimate.
\end{proof}

\subsection{Defect estimates}
\label{subsec:defect-estimates}

In this section we want to estimate the defect $u_h-P_h u$ for some (isolated) ground state $u$ and the corresponding unique local minimizer $u_h \in V_h \cap T_{\ci u}\S$ according to Theorem \ref{thrm:loc-existence:u_h}. The proof will take place in several steps. We start with a lemma that relates the eigenvalue error $\lambda - \lambda_h$ to the eigenfunction error $u-u_h$.
\begin{lemma}
\label{lem:identity-lam-lamh}
Assume \ref{A1}--\ref{A3}, let $(u,\lambda) \in \S \times \R$ denote a critical point of $E$, i.e.,
\begin{align*}
\langle E'(u), v \rangle = \lambda\, (u,v)_{L^2(\D)} 
\qquad \text{for all } v \in H^1_0(\D),
\end{align*}
and let $(u_h,\lambda_h) \in (\S \cap T_{\ci u}\S)\times \R$ be a discrete critical point, i.e.,
\begin{align*}
\langle E'(u_h), v_h \rangle = \lambda_h\, (u_h , v_h )_{L^2(\D)}
\qquad \text{for all } v_h \in V_h \cap T_{\ci u}\S .
\end{align*}
Then for all $v_h \in V_h \cap T_{\ci u}\S$ it holds that
\begin{eqnarray}
\label{eq:lemma-Epp-difference}
\langle (E''(u)-\lambda \mathcal{I})(u_h-u),\, v_h \rangle 
&=& (\lambda_h - \lambda)\,(u_h , v_h)_{L^2(\D)} - \langle R(u) (u-u_h) , v_h \rangle 
\end{eqnarray}
where
\begin{align*}
\langle R(u) (u-u_h) , v_h \rangle :=  \tfrac{1}{2} \langle E^{\prime\prime\prime}(u)(u_h-u,u_h-u)  , v_h \rangle
     + \tfrac{1}{6} \langle E^{\prime\prime\prime}(u_h-u)(u_h-u,u_h-u)  , v_h \rangle
\end{align*}
can be bounded by
\begin{align}
\label{remainder-bound}
| \langle R(u) (u-u_h) , v_h \rangle  | \,\,\,\lesssim\,\,\, \tfrac{1}{\eps^2} \int_{\D} \left( |u-u_h|^2 + |u-u_h|^3 \right) |v_h| \dx.
 \end{align}
\end{lemma}

\begin{proof}
Let $\xi := u_h - u$. By the critical point relations for $u$ and $u_h$ we have
\begin{align}
\label{eq:crit-diff}
\langle E'(u_h) - E'(u), v_h \rangle
= \lambda_h (u_h , v_h) - \lambda(u, v_h)
= (\lambda_h - \lambda)(u_h , v_h) + \lambda(\xi, v_h),
\end{align}
for all $v_h \in V_h \cap T_{\ci u}\S$. Using that $E$ is quartic, the Taylor formula for $E'$ at $u$ in direction $\xi$ yields:
\begin{align}
\label{eq:taylor-Eprime}
\nonumber 
E'(u_h) - E'(u)
&= E''(u)\,\xi 
  +\tfrac{1}{2} E^{\prime\prime\prime}(u)(\xi,\xi) 
     + \tfrac{1}{6}E^{\prime\prime\prime\prime}(u)(\xi,\xi,\xi) \\
&= E''(u)\,\xi 
  +\tfrac{1}{2} E^{\prime\prime\prime}(u)(\xi,\xi) 
     + \tfrac{1}{6}E^{\prime\prime\prime}(\xi)(\xi,\xi), 
\end{align}
which is valid by \ref{A1}--\ref{A3}.  
Testing \eqref{eq:taylor-Eprime} with $v_h$ gives
\begin{align}
\label{eq:taylor-tested}
\langle E'(u_h) - E'(u), v_h \rangle
= \langle E''(u)\,\xi , v_h \rangle
  + \tfrac{1}{2} \langle E^{\prime\prime\prime}(u)(\xi,\xi)  , v_h \rangle
     + \tfrac{1}{6} \langle E^{\prime\prime\prime}(\xi)(\xi,\xi)  , v_h \rangle.
\end{align}
Combining \eqref{eq:crit-diff} and \eqref{eq:taylor-tested}, and subtracting
$\lambda(\xi,v_h)$ from both sides, we obtain
\begin{align*}
\langle (E''(u)-\lambda \mathcal{I})\xi , v_h \rangle
= (\lambda_h - \lambda)(u_h , v_h)
  - \tfrac{1}{2} \langle E^{\prime\prime\prime}(u)(\xi,\xi)  , v_h \rangle
     - \tfrac{1}{6} \langle E^{\prime\prime\prime}(\xi)(\xi,\xi)  , v_h \rangle ,
\end{align*}
which is exactly \eqref{eq:lemma-Epp-difference}. Finally, a simple calculation shows
\begin{align*}
 \langle E^{\prime\prime\prime}(u)(\xi,\xi)  , v_h \rangle \lesssim \tfrac{\beta}{\eps^2} \int_{\D} |u|\,|\xi|^2 |v_h| \dx 
 \qquad\mbox{and}\qquad 
 \langle E^{\prime\prime\prime}(\xi)(\xi,\xi)  , v_h \rangle \lesssim \tfrac{\beta}{\eps^2} \int_{\D} |\xi|^3 |v_h| \dx,
\end{align*}
which yields, together with $\| u\|_{L^{\infty}(\D)}\lesssim 1$, the desired bound.
\end{proof}

We can conclude the following identity for the eigenvalue error and a corresponding estimate against $u-u_h$. The same type of estimate can be found in \cite[Theorem 3.2]{HY25MathComp}, however, with stronger requirements and only as an asymptotic result.

\begin{conclusion}\label{conc:ev-estimate}
In the setting of Theorem \ref{thrm:loc-existence:u_h} it holds
\begin{eqnarray}
\label{ev-estimate}
 |\lambda_h - \lambda | 
&\lesssim&
 \| u_h-u\|_{H^1_{\eps}(\D)}^2 + \tfrac{1}{\eps^2} \left| ( |u|^2 u , u_h - u )_{L^2(\D)} \right|.
\end{eqnarray}
\end{conclusion}

\begin{proof}
Testing with $v_h=u_h$ in \eqref{eq:lemma-Epp-difference} and adding and subtracting $\langle (E''(u)-\lambda \mathcal{I})(u_h-u),\, u \rangle$ to the equation yields
\begin{eqnarray*}
 \lefteqn{\lambda_h - \lambda }\\
&=& \langle (E''(u)-\lambda \mathcal{I})(u_h-u),\, (u_h-u) \rangle 
+ \langle (E''(u)-\lambda \mathcal{I})(u_h-u),\, u \rangle 
 + \langle R(u) (u-u_h) , u_h \rangle \\
&\overset{\eqref{relation-secE-firstE}}{=}&
 \langle (E''(u)-\lambda \mathcal{I})(u_h-u),\, (u_h-u) \rangle 
+ \tfrac{2 \beta}{\eps^2} ( |u|^2 u , u_h - u )_{L^2(\D)}  + \langle R(u) (u-u_h) , u_h \rangle.
\end{eqnarray*}
With the continuity estimate \eqref{continuity-constrained-Hessian} for $E^{\prime\prime}(u)  -\lambda \mathcal{I}$ and the remainder estimate \eqref{remainder-bound} for $R(u)$, we obtain
\begin{eqnarray*}
 \lefteqn{|\lambda_h - \lambda |}\\
&\lesssim& \| u_h-u\|_{H^1_{\eps}(\D)}^2 + \tfrac{1}{\eps^2} \left| ( |u|^2 u , u_h - u )_{L^2(\D)} \right| + \tfrac{1}{\eps^2} \int_{\D} \left( |u-u_h|^2 + |u-u_h|^3 \right) |u_h| \dx \\
&\lesssim& \| u_h-u\|_{H^1_{\eps}(\D)}^2 + \tfrac{1}{\eps^2} \left| ( |u|^2 u , u_h - u )_{L^2(\D)} \right| 
+ \tfrac{1}{\eps^2} \left( \| u-u_h \|_{L^3(\D)}^3 + \| u-u_h \|_{L^4(\D)}^4 \right) \\
&\enspace& \quad+ \tfrac{1}{\eps^2} \int_{\D} \left( |u-u_h|^2 + |u-u_h|^3 \right) |u| \dx \\ 
&\overset{\| u \|_{L^{\infty}(\D)} \lesssim 1}{\lesssim}& \| u_h-u\|_{H^1_{\eps}(\D)}^2 + \tfrac{1}{\eps^2} \left| ( |u|^2 u , u_h - u )_{L^2(\D)} \right| 
+ \tfrac{1}{\eps^2} \left(  \| u-u_h \|_{L^3(\D)}^3 + \| u-u_h \|_{L^4(\D)}^4 \right). 
\end{eqnarray*}
To estimate the $\| u-u_h \|_{L^3(\D)}^3$-contribution, we use the $L^p$-interpolation inequality $\| v \|_{L^3(\D)} \le \| v\|_{L^2(\D)}^{1/3} \| v \|_{L^4(\D)}^{2/3}$. Together with the bounds for $u-u_h $ from  Theorem \ref{thrm:loc-existence:u_h}, we have
\begin{align}
\label{abstract-bounds-L3}
\| u-u_h \|_{L^3(\D)} \,\le\, \| u - u_h\|_{L^2(\D)}^{1/3} \| u- u_h \|_{L^4(\D)}^{2/3} 
\,\lesssim\, \eta(\eps) \eps^{(d/2)\cdot(1/3)}  \eps^{(2/3)\cdot(d/4)}  M \,=\, \eta(\eps) \eps^{d/3} M. 
\end{align}
On the other hand, we have the Gagliardo--Nirenberg inequality 
$$
\| u-u_h \|_{L^3(\D)}^2 \lesssim \| u-u_h \|_{L^2(\D)}^{2-d/3} \| \nabla (u-u_h) \|_{L^2(\D)}^{d/3}
\lesssim \eps^{2-d/3} \| u-u_h \|_{H^{1}_{\eps}(\D)}^2.
$$
Combining the two estimates (for $\| u-u_h \|_{L^3(\D)}$ and $\| u-u_h \|_{L^3(\D)}^2$) yields
\begin{eqnarray*}
\tfrac{1}{\eps^2}  \| u-u_h \|_{L^3(\D)}^3 &\lesssim& \tfrac{\eta(\eps)}{\eps^2}  \eps^{d/3} M  \eps^{2-d/3} \| u-u_h \|_{H^{1}_{\eps}(\D)}^2 \\
&\lesssim& \eta(\eps) M \| u-u_h \|_{H^{1}_{\eps}(\D)}^2 \,\,\, \lesssim  \,\,\,\| u-u_h \|_{H^{1}_{\eps}(\D)}^2 .
\end{eqnarray*}
Similarly, we obtain again with Gagliardo--Nirenberg and Theorem \ref{thrm:loc-existence:u_h} that
\begin{eqnarray*}
\tfrac{1}{\eps^2}  \| u-u_h \|_{L^4(\D)}^4 &\lesssim&
\tfrac{1}{\eps^2}  \| u-u_h \|_{L^4(\D)}^2 \| u-u_h \|_{L^4(\D)}^2 
\,\,\,\lesssim\,\,\, \eps^{-d/2} \| u-u_h \|_{L^4(\D)}^2 \| u-u_h \|_{H^1_{\eps}(\D)}^2 \\
&\lesssim& \eta(\eps)^2 M^2 \| u-u_h \|_{H^1_{\eps}(\D)}^2 \,\,\,\lesssim \,\,\,\| u-u_h \|_{H^1_{\eps}(\D)}^2.
\end{eqnarray*}
Consequently, the estimate for $|\lambda_h - \lambda |$ reduces to
\begin{eqnarray*}
 |\lambda_h - \lambda | 
&\lesssim&
 \| u_h-u\|_{H^1_{\eps}(\D)}^2 + \tfrac{1}{\eps^2} \left| ( |u|^2 u , u_h - u )_{L^2(\D)} \right|.
\end{eqnarray*}
\end{proof}
Next, we need to prove that $\tfrac{1}{\eps^2} \left| ( |u|^2 u , u_h - u )_{L^2(\D)} \right|$ can be also bounded by a quadratic contribution in the error. The argument was developed in \cite[Proof of Thm.~3.3.]{HY25MathComp} and we adopt it our setting, which requires a careful use of the Gagliardo-Nirenberg inequality and a tracing of all $\eps$-dependencies as sharp as possible.

\begin{lemma}\label{lem:est-nl-term}
Suppose again that we are in the setting of Theorem \ref{thrm:loc-existence:u_h}. It holds
\begin{eqnarray}
\label{lem:est-nl-term-v1}
 \tfrac{1}{\eps^2} \left| ( |u|^2 u , u_h - u )_{L^2(\D)} \right| 
&\lesssim&   
 \tfrac{1}{\eps \eta(\eps)} \Big( \tfrac{h}{\eps} \| P_h u - u  \|_{H^1_{\eps}(\D)} + |\lambda_h - \lambda| \, \| u_h - u \|_{L^2(\D)} \\
\nonumber &\enspace&\hspace{50pt}  + 
  \tfrac{1}{\eps^2}  \| u-u_h \|_{L^4(\D)}^2  + \tfrac{1}{\eps^{2+d/6}}  \| u-u_h \|_{L^4(\D)}^3 \Big)
\end{eqnarray}
as well as
\begin{eqnarray}
\label{lem:est-nl-term-v2}
\lefteqn{  \tfrac{1}{\eps^2} \left| ( |u|^2 u , u_h - u )_{L^2(\D)} \right| 
\,\,\,\lesssim\,\,\,
 \tfrac{1}{\eps \eta(\eps)} \Big( \tfrac{h}{\eps} \| P_h u - u  \|_{H^1_{\eps}(\D)} + |\lambda_h - \lambda| \, \| u_h - u \|_{L^2(\D)} } \\
\nonumber &\enspace& + 
  \tfrac{h}{\eps^2}  \| u-u_h \|_{L^4(\D)}^2 + \tfrac{h}{\eps^{2+d/6}} \| u-u_h \|_{L^4(\D)}^3 
  + \tfrac{1}{\eps^3}  \| u-u_h \|_{L^2(\D)}^2 + \tfrac{1}{\eps^3}  \| u-u_h \|_{L^3(\D)}^3 \Big)
\end{eqnarray}
\end{lemma}

\begin{proof}
With $T_{\ci u} \S  = H_{u} \S \oplus \mbox{span}_{\R}\{ u \} $, we can decompose the error uniquely as 
\begin{align*}
u_h - u  \,\,=\,\, \alpha_{\perp}  \, u + e_{\perp}
\qquad
\mbox{where}
\quad
 \alpha_{\perp} := (\sqrt{1 - \| e_{\perp} \|_{L^2(\D)}^2} -1  )
\end{align*}
for some $e_{\perp} \in H_{u} \S$ and where it holds $|\alpha_{\perp} | \le \| u-u_h\|_{L^2(\D)}^2$ (see \cite[Lem.~5.10.]{HY25MathComp}).

Next, let $z_u \in H_{u}\S$ denote the unique solution to the auxiliary problem
\begin{eqnarray*}
\langle (E^{\prime\prime}(u)  -\lambda \calI) z_u , v \rangle = \tfrac{\beta}{\eps^2} ( |u|^2 u , v )_{L^2(\D)} \qquad \mbox{for all } v\in H_{u} \S.
\end{eqnarray*}
Since $|u|^2 u \in L^{\infty}(\D)$ with $\| |u|^2 u \|_{L^{\infty}(\D)}$, we have the usual stability bounds $\| z_u \|_{H^1_{\eps}(\D)} \lesssim \tfrac{1}{\eps^2\eta(\eps)}$ and $\| z_u \|_{L^{\infty}(\D)}  \lesssim \| D^2 z_u \|_{L^2(\D)} \lesssim \tfrac{1}{\eps^2\eta(\eps)}$. Using the definition of $z_u$, we have
\begin{eqnarray*}
\tfrac{\beta}{\eps^2} ( |u|^2 u , u_h - u  )_{L^2(\D)} 
&=&  \alpha_{\perp} \, \tfrac{\beta}{\eps^2} ( |u|^2 u ,u  )_{L^2(\D)} + \langle (E^{\prime\prime}(u)  -\lambda \calI) z_u , e_{\perp} \rangle 
\end{eqnarray*}
Plugging $e_{\perp}= (u_h - P_h u) + (P_h u - u) - \alpha_{\perp} u$ into the equation yields
\begin{eqnarray*}
\lefteqn{  \tfrac{\beta}{\eps^2} ( |u|^2 u , u_h - u  )_{L^2(\D)} } \\
&=&  \alpha_{\perp} \, \tfrac{\beta}{\eps^2} ( |u|^2 u ,u  )_{L^2(\D)} + \langle (E^{\prime\prime}(u)  -\lambda \calI) z_u ,  (u_h - P_h u) + (P_h u - u) - \alpha_{\perp} u \rangle \\
&=&  \underbrace{\alpha_{\perp} \, \tfrac{\beta}{\eps^2} ( |u|^2 u ,u  )_{L^2(\D)} }_{:= \mbox{I}}
- \underbrace{ \alpha_{\perp} \, \langle (E^{\prime\prime}(u)  -\lambda \calI) z_u ,  u \rangle }_{:= \mbox{II}} \\
&\enspace&+ \underbrace{ \langle (E^{\prime\prime}(u)  -\lambda \calI) (z_u - P_h z_u) ,  P_h u - u  \rangle }_{:= \mbox{III}}
+ \underbrace{ \langle (E^{\prime\prime}(u)  -\lambda \calI) z_u ,  u_h - P_h u  \rangle }_{:= \mbox{IV}}.
\end{eqnarray*}
For the first term we have directly with $\| u \|_{L^{\infty}(\D)} \lesssim 1$ that
\begin{eqnarray*}
|\mbox{I}| &\lesssim&  \tfrac{1}{\eps^2} |\alpha_{\perp}|  \,\,\, \le \,\,\, \tfrac{1}{\eps^2} \| u-u_h\|_{L^2(\D)}^2.
\end{eqnarray*}
For the second term, we use $(E^{\prime\prime}(u)  -\lambda \calI)u=E^{\prime\prime}(u)\,u  -E^{\prime}(u) =\tfrac{2\beta}{\eps^2} (|u|^2u , \cdot )_{L^2(\D)}$ to obtain
\begin{eqnarray*}
|\mbox{II}| &=&  \tfrac{2\beta}{\eps^2} |\alpha_{\perp}| \,(|u|^2u , z_u )_{L^2(\D)} \,\,\, \lesssim \,\,\, \tfrac{1}{\eps^2} \| u-u_h\|_{L^2(\D)}^2 \| z_u \|_{L^2(\D)} \\
&\lesssim& \tfrac{1}{\eps} \| u-u_h\|_{L^2(\D)}^2 \| z_u \|_{H^1_{\eps}(\D)} 
\,\,\, \lesssim \,\,\, \tfrac{1}{\eps^3\eta(\eps)} \| u-u_h\|_{L^2(\D)}^2.
\end{eqnarray*}
For term three we exploit the properties of $P_h$ (which requires $h\lesssim \eps \,\eta(\eps)$) and we get with the continuity of $E^{\prime\prime}(u)  -\lambda \calI$ that
\begin{eqnarray*}
|\mbox{III}| &=&  |\langle (E^{\prime\prime}(u)  -\lambda \calI) (z_u - P_h z_u) ,  P_h u - u  \rangle| 
\,\,\, \lesssim \,\,\,  \| z_u - P_h z_u \|_{H^1_{\eps}(\D)} \| P_h u - u  \|_{H^1_{\eps}(\D)}  \\
&\lesssim& h \| D^2 z_u \|_{L^2(\D)}   \| P_h u - u  \|_{H^1_{\eps}(\D)}  
\,\,\, \lesssim \,\,\, \tfrac{h}{\eps^2 \eta(\eps)} \| P_h u - u  \|_{H^1_{\eps}(\D)}.
\end{eqnarray*}
To estimate the fourth term, we first note that
\begin{eqnarray*}
 \lefteqn{ \langle (E^{\prime\prime}(u)  -\lambda \calI) z_u ,  u_h - P_h u  \rangle 
 \,\,\,=\,\,\,  \langle (E^{\prime\prime}(u)  -\lambda \calI) P_h z_u ,  u_h - P_h u  \rangle} \\
&=&  \langle (E^{\prime\prime}(u)  -\lambda \calI) P_h z_u ,  u_h - u  \rangle  \\
&\overset{\eqref{eq:lemma-Epp-difference}}{=}&  
(\lambda_h - \lambda)\,(u_h , P_h z_u )_{L^2(\D)} - \langle R(u) (u-u_h) , P_h z_u \rangle \\
&\overset{P_h z_u \in H_{u}\S }{=}& (\lambda_h - \lambda)\,(u_h - u, P_h z_u )_{L^2(\D)} - \langle R(u) (u-u_h) , P_h z_u \rangle.
\end{eqnarray*}
Now we present two different estimates for \,\,$\mbox{IV}= \langle (E^{\prime\prime}(u)  -\lambda \calI) z_u ,  u_h - P_h u \rangle$, depending on if we want to achieve \eqref{lem:est-nl-term-v1} or \eqref{lem:est-nl-term-v2}. For \eqref{lem:est-nl-term-v1} we obtain
\begin{eqnarray*}
 \lefteqn{ |\langle (E^{\prime\prime}(u)  -\lambda \calI) z_u ,  u_h - P_h u  \rangle | }\\
&\lesssim& |\lambda_h - \lambda|\, \eps \, \| u_h - u \|_{L^2(\D)} \| P_h z_u \|_{H^1_{\eps}(\D)} + |\langle R(u) (u-u_h) , P_h z_u \rangle | \\
&\overset{\eqref{remainder-bound}}{\lesssim}& 
 |\lambda_h - \lambda|\, \eps \, \| u_h - u \|_{L^2(\D)} \| P_h z_u \|_{H^1_{\eps}(\D)} +
 \tfrac{1}{\eps^2} \int_{\D} \left( |u-u_h|^2 + |u-u_h|^3 \right) |P_h z_u| \dx \\
 &\overset{\eqref{Ph-Honeeps-stability}}{\lesssim}& |\lambda_h - \lambda|\, \eps \, \| u_h - u \|_{L^2(\D)} \| z_u \|_{H^1_{\eps}(\D)} +
 \tfrac{1}{\eps^2} \int_{\D} \left( |u-u_h|^2 + |u-u_h|^3 \right) |P_h z_u| \dx \\
  &\lesssim& \tfrac{1}{\eps \eta(\eps)} |\lambda_h - \lambda| \, \| u_h - u \|_{L^2(\D)} + 
  \tfrac{1}{\eps^2}  \| u-u_h \|_{L^4(\D)}^2 \| P_h z_u \|_{L^2(\D)} \\
  &\enspace& \quad + \tfrac{1}{\eps^2}  \| u-u_h \|_{L^4(\D)}^3 \| P_h z_u \|_{L^3(\D)} \\
  &\overset{\eqref{gag-ni-L3}}{\lesssim}&  \tfrac{1}{\eps \eta(\eps)} |\lambda_h - \lambda| \, \| u_h - u \|_{L^2(\D)} + 
  \tfrac{1}{\eps}  \| u-u_h \|_{L^4(\D)}^2 \| z_u \|_{H^1_{\eps}(\D)} \\
  &\enspace& \quad + \tfrac{1}{\eps^2} \eps^{1-d/6}  \| u-u_h \|_{L^4(\D)}^3 \| z_u \|_{H^1_{\eps}(\D)} \\
  &\lesssim& \tfrac{1}{\eps \eta(\eps)} \left( |\lambda_h - \lambda| \, \| u_h - u \|_{L^2(\D)} + 
  \tfrac{1}{\eps^2}  \| u-u_h \|_{L^4(\D)}^2  + \tfrac{1}{\eps^{2+d/6}}  \| u-u_h \|_{L^4(\D)}^3 \right).
\end{eqnarray*}
For the alternative estimate \eqref{lem:est-nl-term-v2} we only deviate in the treatment of $\int_{\D} \left( |u-u_h|^2 + |u-u_h|^3 \right) |P_h z_u| \dx$. In this case, we obtain
\begin{eqnarray*}
 \lefteqn{ |\langle (E^{\prime\prime}(u)  -\lambda \calI) z_u ,  u_h - P_h u  \rangle | }\\
&\lesssim& \tfrac{1}{\eps \eta(\eps)} |\lambda_h - \lambda| \, \| u_h - u \|_{L^2(\D)} +
 \tfrac{1}{\eps^2} \int_{\D} \left( |u-u_h|^2 + |u-u_h|^3 \right) |P_h z_u| \dx \\
 &\le& \tfrac{1}{\eps \eta(\eps)} |\lambda_h - \lambda| \, \| u_h - u \|_{L^2(\D)} +
 \tfrac{1}{\eps^2} \int_{\D} \left( |u-u_h|^2 + |u-u_h|^3 \right) (|P_h z_u - z_u| + |z_u|) \dx \\
  &\lesssim& \tfrac{1}{\eps \eta(\eps)} |\lambda_h - \lambda| \, \| u_h - u \|_{L^2(\D)} + 
  \tfrac{1}{\eps^2}  \| u-u_h \|_{L^4(\D)}^2 \| P_h z_u - z_u \|_{L^2(\D)} \\
  &\enspace& \quad + \tfrac{1}{\eps^2}  \| u-u_h \|_{L^4(\D)}^3 \| P_h z_u  - z_u\|_{L^3(\D)} 
  + \tfrac{1}{\eps^2}  \| u-u_h \|_{L^2(\D)}^2 \| z_u \|_{L^{\infty}(\D)} \\
  &\enspace& \quad + \tfrac{1}{\eps^2}  \| u-u_h \|_{L^3(\D)}^3 \| z_u \|_{L^{\infty}(\D)}
  \\
  &\overset{\eqref{gag-ni-L3}}{\lesssim}& \tfrac{1}{\eps \eta(\eps)} |\lambda_h - \lambda| \, \| u_h - u \|_{L^2(\D)} + 
  \tfrac{h}{\eps^3 \eta(\eps)}  \| u-u_h \|_{L^4(\D)}^2 + \tfrac{h}{\eps^3 \eta(\eps)} \eps^{-d/6} \| u-u_h \|_{L^4(\D)}^3  \\
  &\enspace& \quad 
  + \tfrac{1}{\eps^4 \eta(\eps)}  \| u-u_h \|_{L^2(\D)}^2 + \tfrac{1}{\eps^4 \eta(\eps)}  \| u-u_h \|_{L^3(\D)}^3.
  \\
  &\lesssim& \tfrac{1}{\eps \eta(\eps)} \left( |\lambda_h - \lambda| \, \| u_h - u \|_{L^2(\D)} \right. \\ 
  &\enspace&\quad\left. + 
  \tfrac{h}{\eps^2}  \| u-u_h \|_{L^4(\D)}^2 + \tfrac{h}{\eps^{2+d/6}} \| u-u_h \|_{L^4(\D)}^3 
  + \tfrac{1}{\eps^3}  \| u-u_h \|_{L^2(\D)}^2 + \tfrac{1}{\eps^3}  \| u-u_h \|_{L^3(\D)}^3
   \right).
\end{eqnarray*}
Combining the estimates for I, II, III and IV (two versions) proves the lemma.
\end{proof}
Let us now combine the results of Conclusion \ref{conc:ev-estimate} and Lemma \ref{lem:est-nl-term}.

\begin{conclusion}\label{conclusion:ev-estimate-sharper}
Assume the setting of Theorem \ref{thrm:loc-existence:u_h}. Then there exists some $M^{\ast} >0$ (independent of $\eps$ and $h$) such that for all $M\le M^{\ast}$ it holds
\begin{eqnarray}
\label{conclusion:ev-estimate-sharper-est-v1}
\lefteqn{ |\lambda_h - \lambda | } \\
\nonumber &\lesssim & \| u_h-u\|_{H^1_{\eps}(\D)}^2  +
\tfrac{1}{\eps \eta(\eps)} \Big( \tfrac{h}{\eps} \| P_h u - u  \|_{H^1_{\eps}(\D)}  + 
  \tfrac{1}{\eps^2}  \| u-u_h \|_{L^4(\D)}^2  + \tfrac{1}{\eps^{2+d/6}}  \| u-u_h \|_{L^4(\D)}^3 \Big)
\end{eqnarray}
and
\begin{eqnarray}
\label{conclusion:ev-estimate-sharper-est-v2}
\lefteqn{ |\lambda_h - \lambda | \,\,\, \lesssim \,\,\,  \| u_h-u\|_{H^1_{\eps}(\D)}^2  + \tfrac{1}{\eps \eta(\eps)} \Big( \tfrac{h}{\eps} \| P_h u - u  \|_{H^1_{\eps}(\D)} 
+\tfrac{h}{\eps^2}  \| u-u_h \|_{L^4(\D)}^2 }\\
\nonumber &\enspace& \qquad
 + \tfrac{h}{\eps^{2+d/6}} \| u-u_h \|_{L^4(\D)}^3 
  + \tfrac{1}{\eps^3}  \| u-u_h \|_{L^2(\D)}^2 + \tfrac{1}{\eps^3}  \| u-u_h \|_{L^3(\D)}^3 \Big).
\end{eqnarray}
We will use the two estimates to treat different terms in the final estimate.
\end{conclusion}

\begin{proof}
We insert the estimates \eqref{lem:est-nl-term-v1} and \eqref{lem:est-nl-term-v2} of Lemma \ref{lem:est-nl-term} into \eqref{ev-estimate}. For example, in the case of estimate \eqref{lem:est-nl-term-v1} we obtain
\begin{eqnarray*}
|\lambda_h - \lambda | &\lesssim & \| u_h-u\|_{H^1_{\eps}(\D)}^2 + \tfrac{1}{\eps \eta(\eps)}  |\lambda_h - \lambda| \, \| u_h - u \|_{L^2(\D)} \\
&\enspace& +
\tfrac{1}{\eps \eta(\eps)} \Big( \tfrac{h}{\eps} \| P_h u - u  \|_{H^1_{\eps}(\D)}  + 
  \tfrac{1}{\eps^2}  \| u-u_h \|_{L^4(\D)}^2  + \tfrac{1}{\eps^{2+d/6}}  \| u-u_h \|_{L^4(\D)}^3 \Big).
\end{eqnarray*}
We apply the estimate $\| u- u_h \|_{L^2(\D)} \le \eta(\eps) \eps^{d/2} M$ from Theorem \ref{thrm:loc-existence:u_h} to bound the eigenvalue-contribution on the right hand side by
\begin{eqnarray*}
\tfrac{1}{\eps \eta(\eps)}  |\lambda_h - \lambda| \, \| u_h - u \|_{L^2(\D)}
&\lesssim& \eps^{d/2-1} |\lambda_h - \lambda| \, M \,\,\, \overset{d=2,3}{\lesssim} \,\,\,  |\lambda_h - \lambda| \, M.
\end{eqnarray*}
Hence, if $M$ is sufficiently small (independent of $\eps$ and $h$), we can absorb $|\lambda_h - \lambda| \, M$ into the left hand side. The other estimate is analogous. This proves the conclusion.
\end{proof}

We are now prepared to estimate the defect $u_h-P_h u$ against the error $u-u_h$.
\begin{lemma}\label{lem:defect-est}
Assume again that we are in the setting of Theorem \ref{thrm:loc-existence:u_h}, in particular, the mesh size fulfills (at least) $h \lesssim \eps \, \eta(\eps) \, \eps^{d/2} M$. Then, there exists a constant $M^{\ast} >0$ (independent of $\eps$ and $h$) such that for all $M\le M^{\ast}$ it holds
\begin{eqnarray*}
 \| u_h-P_h u \|_{H^1_{\eps}(\D)} 
&\lesssim& M  \| u-u_h \|_{H^1_{\eps}(\D)}.
\end{eqnarray*}
\end{lemma}

\begin{proof}
We start from the G{\aa}rding inequality \eqref{eq:secE-weak-coercive} to estimate
\begin{eqnarray}
\label{garding-once-more}
\tfrac{1}{2}\, \| u_h-P_h u \|_{H^1_{\eps}(\D)}^2 &\le&  \langle (E^{\prime\prime}(u)  -\lambda \mathcal{I}) (u_h-P_h u) , u_h-P_h u \rangle 
\;+\;
\tilde{c}_1 \,\tfrac{1}{\eps^2}\,\| u_h-P_h u \|_{L^2(\D)}^2.
\end{eqnarray}
Using a Schatz argument, we let $\xi \in H_{u}\S$ denote the unique solution to
\begin{align*}
\langle (E^{\prime\prime}(u)  -\lambda \mathcal{I})\xi , v \rangle  \,\,\, = \,\,\, \tilde{c}_1 \,\tfrac{1}{\eps^2} ( u_h-P_h u , v)_{L^2(\D)} \qquad
\mbox{for all } v\in H_u\S.
\end{align*}
Hence, $P_h\xi \in V_h \cap H_{u}\S$, and we have the stability estimate
 \begin{align}
 \label{xi-est}
\nonumber\tfrac{1}{\eps} \| P_h\xi \|_{L^2(\D)} \,\,\,&\le\,\,\, \| P_h\xi \|_{H^1_{\eps}(\D)} 
 \,\,\,\overset{\eqref{Ph-Honeeps-stability}}{\lesssim} \,\,\, \| \xi \|_{H^1_{\eps}(\D)} \\
 \,\,\,&\lesssim \,\,\, \tfrac{\eps}{\eta(\eps) \eps^2} \| u_h - P_h u \|_{L^2(\D)} 
  \,\,\,\lesssim \,\,\, \tfrac{1}{\eta(\eps)} \| u_h - P_h u \|_{H^1_{\eps}(\D)}.
 \end{align} 
Since the defect $u_h - P_h u$ is not an admissible test function for the dual problem, we need to project it into the horizontal space. For that, we recall from the proof of Lemma \ref{lem:est-nl-term} the decomposition 
\begin{align*}
u_h - u  \,\,=\,\, \alpha_{\perp}  \, u + e_{\perp}
\qquad
\mbox{where}
\quad
e_{\perp} \in H_{u} \S \mbox{ and } |\alpha_{\perp} | \le \| u-u_h\|_{L^2(\D)}^2.
\end{align*}
Since $u-P_h u \in H_{u} \S $, we obtain
\begin{align*}
u_h - P_h u   \,\,=\,\, \alpha_{\perp}  \, u + \tilde{e}_{\perp}
\qquad
\mbox{where}
\quad
\tilde{e}_{\perp} := e_{\perp} +  (u - P_h u)  \in H_{u} \S.
\end{align*}
Hence, with $\tilde{e}_{\perp} \in H_{u} \S$ as test function in the dual problem, we have 
\begin{eqnarray*}
\lefteqn{ \tilde{c}_1 \tfrac{1}{\eps^2}\,\| u_h-P_h u \|_{L^2(\D)}^2 \,\,\,=\,\,\, 
\tilde{c}_1 \, \tfrac{ \alpha_{\perp}  }{\eps^2}\, (  u_h-P_h u , u )_{L^2(\D)} 
+ \tilde{c}_1\, \tfrac{1}{\eps^2}\, (  u_h-P_h u , \tilde{e}_{\perp}  )_{L^2(\D)} } \\
&=& 
\tilde{c}_1\, \tfrac{ \alpha_{\perp}  }{\eps^2}\, (  u_h-P_h u , u )_{L^2(\D)}  + 
\, \langle (E^{\prime\prime}(u)  -\lambda \mathcal{I})\xi , \tilde{e}_{\perp} \rangle  \\
&=& \underbrace{ \tilde{c}_1\, \tfrac{ \alpha_{\perp}  }{\eps^2}\, (  u_h-P_h u , u )_{L^2(\D)} 
 -  \alpha_{\perp} \langle (E^{\prime\prime}(u)  -\lambda \mathcal{I})\xi , u  \rangle }_{=: \mbox{III}} + 
 \langle (E^{\prime\prime}(u)  -\lambda \mathcal{I})\xi ,u_h-P_h u  \rangle .
\end{eqnarray*}
From this identity and the G{\aa}rding estimate \eqref{garding-once-more} we can bound $\| u_h-P_h u \|_{H^1_{\eps}(\D)}^2$ as
\begin{eqnarray*}
\lefteqn{ \tfrac{1}{2}\, \| u_h-P_h u \|_{H^1_{\eps}(\D)}^2 \,\,\,\le\,\,\,  \langle (E^{\prime\prime}(u)  -\lambda \mathcal{I}) (u_h-P_h u + \xi ) , u_h-P_h u \rangle \,+\, \mbox{III}} \\
&=& \langle (E^{\prime\prime}(u)  -\lambda \mathcal{I}) ( \xi -P_h\xi ) , u_h-P_h u \rangle +   \langle (E^{\prime\prime}(u)  -\lambda \mathcal{I}) (u_h-P_h u + P_h\xi ) , u_h-P_h u \rangle 
 \,+\, \mbox{III}
\\
&\overset{\eqref{galerkin-orth}}{=}&  \underbrace{\langle (E^{\prime\prime}(u)  -\lambda \mathcal{I}) ( \xi -P_h\xi ) , u_h-P_h u \rangle}_{=:\mbox{I}} +  \underbrace{ \langle (E^{\prime\prime}(u)  -\lambda \mathcal{I}) (u_h-P_h u + P_h\xi ) , u_h - u \rangle }_{=: \mbox{II}}  \,+\, \mbox{III}.
\end{eqnarray*}
We now estimate the terms on the right hand side individually. For the first term we can proceed as in the proof of Lemma \ref{lem:H1-ritz-proj-est}, where the continuity of $E^{\prime\prime}(u)  -\lambda \mathcal{I}$ (cf. \eqref{continuity-constrained-Hessian}), the projection estimates in Lemma \ref{lem:H1-ritz-proj-est} and the regularity estimates in Lemma \ref{H2-lemma} yield
\begin{eqnarray*}
|\mbox{I}| &\lesssim& \| \xi -P_h\xi \|_{H^1_{\eps}(\D)}  \| u_h-P_h u \|_{H^1_{\eps}(\D)} \,\,\, \lesssim \,\,\,  h  \, \| D^2\xi \|_{L^2(\D)}  \| u_h-P_h u \|_{H^1_{\eps}(\D)}\\
&\lesssim& h  \, \eta(\eps)^{-1} \,\tfrac{1}{\eps^2} \| u_h-P_h u  \|_{L^2(\D)}  \| u_h-P_h u \|_{H^1_{\eps}(\D)} \,\,\, \lesssim \,\,\, \tfrac{h}{\eps \eta(\eps)}  \| u_h-P_h u \|_{H^1_{\eps}(\D)}^2 \\
&\overset{h\lesssim \eps \eta(\eps) \eps^{d/2} M }{\lesssim}& \eps^{d/2}  M \| u_h-P_h u \|_{H^1_{\eps}(\D)}^2.
\end{eqnarray*}
To treat the second term, we will use the following identity which shows, by means of Lemma \ref{lem:def-ritz-proj}, that $(u_h , u_h-P_h u + P_h\xi )_{L^2(\D)} $ is a higher order term. We have
\begin{eqnarray}
\label{step-proof-def-aux-1}
\nonumber \lefteqn{(u_h , u_h-P_h u + P_h\xi )_{L^2(\D)} \,\,\,\overset{P_h\xi\in H_u \S}{=}\,\,\,
(u_h - u , u_h-P_h u + P_h\xi )_{L^2(\D)} + ( u , u_h-P_h u  )_{L^2(\D)} }  \\
\nonumber&=& (u_h - u , u_h-P_h u )_{L^2(\D)} + (u_h - u , P_h\xi )_{L^2(\D)}  + ( u , (u_h-u)+ (u - P_h u)  )_{L^2(\D)}  \\
\nonumber&\overset{u - P_h u\in H_u \S}{=}&  (u_h - u , u_h-P_h u )_{L^2(\D)} + (u_h - u , P_h\xi )_{L^2(\D)}  + ( u , u_h-u )_{L^2(\D)} \\
\nonumber&\overset{u,u_h\in\S}{=}& (u_h - u , u_h-P_h u )_{L^2(\D)} + (u_h - u , P_h\xi )_{L^2(\D)} \\
\nonumber&\enspace&\quad - \tfrac{1}{2} \left(\| u\|^2_{L^2(\D)} - 2( u , u_h)_{L^2(\D)} + \| u_h \|^2_{L^2(\D)} \right) \\
&=& (u_h - u , u_h-P_h u )_{L^2(\D)} + (u_h - u , P_h\xi )_{L^2(\D)}   -\tfrac{1}{2} \| u - u_h \|_{L^2(\D)}^2.
\end{eqnarray}
We can now turn to the second term. Using Lemma \ref{lem:identity-lam-lamh} (and the symmetry of $E^{\prime\prime}(u) - \lambda \calI$)  to obtain 
\begin{eqnarray*}
|\mbox{II}|
&\overset{\eqref{eq:lemma-Epp-difference}}{=}& \left| (\lambda_h - \lambda)\,(u_h , u_h-P_h u + P_h\xi )_{L^2(\D)} - \langle R(u) (u-u_h) , u_h-P_h u + P_h\xi \rangle \right| \\
&\overset{\eqref{step-proof-def-aux-1}}{=}& 
\left| (\lambda_h - \lambda)\, \left( (u_h - u , u_h-P_h u )_{L^2(\D)} + (u_h - u , P_h\xi )_{L^2(\D)}   -\tfrac{1}{2} \| u - u_h \|_{L^2(\D)}^2 \right) \right. \\
&\enspace& \quad \left. - \langle R(u) (u-u_h) , u_h-P_h u + P_h\xi \rangle \right| \\
&\overset{\eqref{remainder-bound},\eqref{xi-est}}{\lesssim}&  |\lambda_h - \lambda| \cdot \left( \| u_h- u \|_{L^2(\D)} \left( \| u_h-P_h u \|_{L^2(\D)} +  \| u_h- u \|_{L^2(\D)} + \| P_h\xi  \|_{L^2(\D)} \right) \right)\\
&\enspace&  + \tfrac{1}{\eps^2} \int_{\D} \left( |u-u_h|^2 + |u-u_h|^3 \right) |u_h-P_h u + P_h\xi| \dx.
\end{eqnarray*}
We split the right hand into four contributions such that $|\mbox{II}| \le \mbox{II}_1+\mbox{II}_2+\mbox{II}_3+\mbox{II}_4$, which are given by
\begin{eqnarray*}
\mbox{II}_1 &:=& |\lambda_h - \lambda| \, \| u_h- u \|_{L^2(\D)} \| u_h-P_h u \|_{L^2(\D)}, \qquad
\mbox{II}_2 \,\,\,:=\,\,\, |\lambda_h - \lambda| \, \| u_h- u \|_{L^2(\D)}^2, \\
\mbox{II}_3 &:=& |\lambda_h - \lambda| \, \| u_h- u \|_{L^2(\D)} \| P_h\xi \|_{L^2(\D)}, \\
\mbox{II}_4 &:=& \tfrac{1}{\eps^2} \int_{\D} \left( |u-u_h|^2 + |u-u_h|^3 \right) |u_h-P_h u + P_h\xi| \dx.
\end{eqnarray*}
From now on we will repeatedly exploit the bounds from Theorem \ref{thrm:loc-existence:u_h}, i.e., $\| (\lambda,u) - (\lambda_h,u_h) \|_{\eps} \le \tfrac{\eta(\eps)}{\eps} \eps^{d/4} M$,\,\, $\| u- u_h \|_{L^2(\D)} \le \eta(\eps) \eps^{d/2} M$ and $\| u- u_h \|_{L^4(\D)} \le \eta(\eps) \eps^{d/4} M$ without further mentioning.

We start with $\mbox{II}_1$. Using Conclusion \ref{conclusion:ev-estimate-sharper} we have
\begin{eqnarray*}
\mbox{II}_1 &\overset{\eqref{conclusion:ev-estimate-sharper-est-v1}}{\lesssim}& \| u_h-u\|_{H^1_{\eps}(\D)}^2  \| u_h- u \|_{L^2(\D)} \| u_h-P_h u \|_{L^2(\D)} \\
&\enspace&+
\tfrac{h}{\eps^2 \eta(\eps)}  \| P_h u - u  \|_{H^1_{\eps}(\D)} \| u_h- u \|_{L^2(\D)} \| u_h-P_h u \|_{L^2(\D)} \\
&\enspace&+ \tfrac{1}{\eps^3 \eta(\eps)} \| u-u_h \|_{L^4(\D)}^2 \| u_h- u \|_{L^2(\D)} \| u_h-P_h u \|_{L^2(\D)} \\
&\enspace&+ \tfrac{1}{\eps^3 \eta(\eps)}  \eps^{-d/6}  \| u-u_h \|_{L^4(\D)}^3  \| u_h- u \|_{L^2(\D)} \| u_h-P_h u \|_{L^2(\D)} \\
&\overset{ h \lesssim \eps \eta(\eps) \eps^{d/2} M }{\lesssim}&  \eps^{d/2}  \eta(\eps)^2 M^2 \| u_h-u\|_{H^1_{\eps}(\D)} \| u_h-P_h u \|_{H^1_{\eps}(\D)} \\
&\enspace&+ \,
  \eta(\eps) \, \eps^{d} \,  M^2 \, \| u_h- u \|_{H^1_{\eps}(\D)} \| u_h-P_h u \|_{H^1_{\eps}(\D)} \\
&\enspace&+ \,\eta(\eps)\, \eps^{d/2\,-\,1} M^2  \| u_h- u \|_{H^1_{\eps}(\D)} \| u_h-P_h u \|_{H^1_{\eps}(\D)} \\
&\enspace&+ \eta(\eps)^2 \eps^{7d/12\,-\,1} M^3  \| u_h- u \|_{H^1_{\eps}(\D)} \| u_h-P_h u \|_{H^1_{\eps}(\D)} \\
&\overset{\tfrac{1}{2}<\tfrac{7}{12}}{\lesssim}&\eps^{d/2\,-\,1} M^2\, \| u_h- u \|_{H^1_{\eps}(\D)} \| u_h-P_h u \|_{H^1_{\eps}(\D)}.
\end{eqnarray*}
Next, we estimate $\mbox{II}_2 = |\lambda_h - \lambda| \, \| u_h- u \|_{L^2(\D)}^2$, where use again Conclusion \ref{conclusion:ev-estimate-sharper} to get
\begin{eqnarray*}
\mbox{II}_2 &\overset{\eqref{conclusion:ev-estimate-sharper-est-v1}}{\lesssim}& \| u_h-u\|_{H^1_{\eps}(\D)}^2 \| u_h- u \|_{L^2(\D)}^2
\,\,+\,\,
\tfrac{h}{\eps^2 \eta(\eps)}  \| P_h u - u  \|_{H^1_{\eps}(\D)} \| u_h- u \|_{L^2(\D)}^2 \\
&\enspace&+ \tfrac{1}{\eps^3 \eta(\eps)} \| u-u_h \|_{L^4(\D)}^2 \| u_h- u \|_{L^2(\D)}^2 
\,\,+ \,\, \tfrac{1}{\eps^3 \eta(\eps)}  \eps^{-d/6}  \| u-u_h \|_{L^4(\D)}^3 \| u_h- u \|_{L^2(\D)}^2 \\
&\overset{ h \lesssim \eps \eta(\eps) \eps^{d/2} M}{\lesssim}&  \eta(\eps)^2 \eps^{d} M^2 \| u_h-u\|_{H^1_{\eps}(\D)}^2 
\,\,+\,\,\eta(\eps) \eps^d M^2 \| u_h- u \|_{H^1_{\eps}(\D)}^2 \\
&\enspace& +\eta(\eps) \eps^{d/2\,-\,1} M^2 \| u_h- u \|_{H^1_{\eps}(\D)}^2 
\,\,+ \,\, \eta(\eps)^2 \eps^{7d/12\,-\,1} M^3 \| u_h- u \|_{H^1_{\eps}(\D)}^2   \\
&\overset{M,\eta(\eps)\lesssim 1}{\lesssim}&
\eta(\eps) \, M^2 \, \| u_h- u \|_{H^1_{\eps}(\D)}^2.
\end{eqnarray*}
We turn to $\mbox{II}_3 = |\lambda_h - \lambda| \, \| u_h- u \|_{L^2(\D)} \| P_h\xi \|_{L^2(\D)}$. 
Using Conclusion \ref{conclusion:ev-estimate-sharper} estimate \eqref{conclusion:ev-estimate-sharper-est-v2} together with $\| P_h u - u  \|_{H^1_{\eps}(\D)}\lesssim \tfrac{h}{\eps^2}$, we obtain 
\begin{eqnarray*}
\mbox{II}_3 &\lesssim& \| u_h-u\|_{H^1_{\eps}(\D)}^2 \, \| u_h- u \|_{L^2(\D)} \| P_h\xi \|_{L^2(\D)} 
\,\,+\,\, \tfrac{h^2}{\eps^4 \eta(\eps)}  \, \| u_h- u \|_{L^2(\D)} \| P_h\xi \|_{L^2(\D)} \\
&\enspace& + \tfrac{h}{\eps^3  \eta(\eps)}  \| u-u_h \|_{L^4(\D)}^2  \, \| u_h- u \|_{L^2(\D)} \| P_h\xi \|_{L^2(\D)} \\
&\enspace&+ \tfrac{h}{\eps^{3+d/6} \eta(\eps)} \| u-u_h \|_{L^4(\D)}^3  \, \| u_h- u \|_{L^2(\D)} \| P_h\xi \|_{L^2(\D)} \\
&\enspace&+ \tfrac{1}{\eps^4 \eta(\eps)}  \| u-u_h \|_{L^2(\D)}^3  \, \| P_h\xi \|_{L^2(\D)} 
 \,\,+\,\,  \tfrac{1}{\eps^4 \eta(\eps)}  \| u-u_h \|_{L^3(\D)}^3 \, \| u_h- u \|_{L^2(\D)} \| P_h\xi \|_{L^2(\D)} \\
 &\overset{\eqref{xi-est}}{\lesssim}& 
  \| u_h-u\|_{H^1_{\eps}(\D)}^2 \, \| u_h- u \|_{L^2(\D)}  \tfrac{\eps}{\eta(\eps)} \| u_h - P_h u \|_{H^1_{\eps}(\D)} \\
&\enspace& + \left( \tfrac{h^2\eps}{\eps^4 \eta(\eps)^2}  + \tfrac{h\eps}{\eps^3  \eta(\eps)^2}  \| u-u_h \|_{L^4(\D)}^2  \right) \, \| u_h- u \|_{L^2(\D)} \| u_h - P_h u \|_{H^1_{\eps}(\D)} \\
&\enspace&+ \tfrac{h}{\eps^{2+d/6} \eta(\eps)^2} \| u-u_h \|_{L^4(\D)}^3  \, \| u_h- u \|_{L^2(\D)}  \| u_h - P_h u \|_{H^1_{\eps}(\D)} \\
&\enspace&+ \tfrac{\eps}{\eps^4 \eta(\eps)^2} \left( \| u-u_h \|_{L^2(\D)}^3 
 \,+\,   \| u-u_h \|_{L^3(\D)}^3 \, \| u_h- u \|_{L^2(\D)} \right) \| u_h - P_h u \|_{H^1_{\eps}(\D)} \\
&\overset{ h \lesssim \eps \eta(\eps) \eps^{d/2} M}{\lesssim}& \| u_h-u\|_{H^1_{\eps}(\D)}^2  \eps^{1+d/2} \| u_h - P_h u \|_{H^1_{\eps}(\D)}
\,\,+\,\, \eps^{d-1} M^2  \, \| u_h- u \|_{L^2(\D)} \, \| u_h - P_h u \|_{H^1_{\eps}(\D)}  \\
&\enspace& + \tfrac{1}{\eta(\eps)}  \eps^{d/2-1} M \| u-u_h \|_{L^4(\D)}^2  \, \| u_h- u \|_{L^2(\D)} \,\| u_h - P_h u \|_{H^1_{\eps}(\D)} \\
&\enspace&+  \tfrac{1}{\eta(\eps)}  \eps^{d/3-1} M \| u-u_h \|_{L^4(\D)}^3  \, \| u_h- u \|_{L^2(\D)} \, \| u_h - P_h u \|_{H^1_{\eps}(\D)}  \\
&\enspace&+ \tfrac{1}{\eta(\eps)^2}  \eps^{-3} \left( \| u-u_h \|_{L^2(\D)}^3  + \| u-u_h \|_{L^3(\D)}^3 \, \| u_h- u \|_{L^2(\D)} \right) \, \| u_h - P_h u \|_{H^1_{\eps}(\D)}  \\
&\overset{\eqref{abstract-bounds},\eqref{abstract-bounds-L3}}{\lesssim}& \eps^{d-2} M^2 \,\| u_h- u \|_{H^1_{\eps}(\D)}  \| u_h - P_h u \|_{H^1_{\eps}(\D)}.
\end{eqnarray*}
It remains to bound $\mbox{II}_4$. Here we obtain with the abstract bounds \eqref{abstract-bounds} and \eqref{abstract-bounds-L3} as well as with the Gagliardo-Nirenberg inequalities that
\begin{eqnarray*}
\mbox{II}_4 &=& \tfrac{1}{\eps^2} \int_{\D} \left( |u-u_h|^2 + |u-u_h|^3 \right) |u_h-P_h u + P_h\xi| \dx \\
&\lesssim& 
 \tfrac{1}{\eps^2}  \| u-u_h \|_{L^4(\D)}^2 (  \| u_h-P_h u\|_{L^2(\D)} + \| P_h\xi\|_{L^2(\D)}  ) \\
 &\enspace&\quad +  \tfrac{1}{\eps^2}  \| u-u_h \|_{L^4(\D)}^3  (  \| u_h-P_h u\|_{L^3(\D)} + \| P_h\xi\|_{L^3(\D)} ) \\
 &\overset{\eqref{gag-ni-L4},\eqref{gag-ni-L3}}{\lesssim}& 
\eta(\eps) \eps^{-1-d/4} \eps^{d/4} M  \| u-u_h \|_{H^1_{\eps}(\D)} \left(  \eps \| u_h-P_h u\|_{H^1_{\eps}(\D)} + \tfrac{\eps}{\eta(\eps)} \| u_h-P_h u\|_{H^1_{\eps}(\D)} \right) \\
 &\enspace&\quad +  \eta(\eps)^2 M^2  \eps^{d/4-1} \| u-u_h \|_{H^1_{\eps}(\D)}   \left( \eps^{1-d/6} \| u_h-P_h u\|_{H^1_{\eps}(\D)} + \eps^{1-d/6} \| P_h\xi \|_{H^1_{\eps}(\D)}  \right) \\
 &\overset{\eqref{xi-est}}{\lesssim}& 
 M  \| u-u_h \|_{H^1_{\eps}(\D)}  \| u_h-P_h u\|_{H^1_{\eps}(\D)}  \\& \enspace \quad &  +M^2  (\eta(\eps)^2\eps^{d/12}+\eta(\eps)\eps^{1+d/12}) \| u-u_h \|_{H^1_{\eps}(\D)} \| u_h - P_h u \|_{H^1_{\eps}(\D)}  \\
&\lesssim& M  \| u-u_h \|_{H^1_{\eps}(\D)}  \| u_h-P_h u\|_{H^1_{\eps}(\D)} .
\end{eqnarray*}
It remains to bound $\mbox{III} = \tilde{c}_1\, \tfrac{ \alpha_{\perp}  }{\eps^2}\, (  u_h-P_h u , u )_{L^2(\D)} 
 -  \alpha_{\perp} \langle (E^{\prime\prime}(u)  -\lambda \mathcal{I})\xi , u  \rangle$. Using $ |\alpha_{\perp} | \le \| u-u_h\|_{L^2(\D)}^2$ we have
 \begin{eqnarray*}
 |\mbox{III}| &\lesssim& \tfrac{1}{\eps^2}  \| u_h-P_h u\|_{L^2(\D)} \| u-u_h\|_{L^2(\D)}^2 +\tfrac{1}{\eps}  \| \xi \|_{H^1_{\eps}(\D)} \| u-u_h\|_{L^2(\D)}^2
 \end{eqnarray*}
With the bound $\| u- u_h \|_{L^2(\D)} \le \eta(\eps) \eps^{d/2} M$ from Theorem \ref{thrm:loc-existence:u_h} and the stability bound $ \| \xi \|_{H^1_{\eps}(\D)} \le   \tfrac{1}{\eta(\eps)} \| u_h - P_h u \|_{H^1_{\eps}(\D)}$ we conclude
 \begin{eqnarray*}
 |\mbox{III}| &\lesssim&  \eta(\eps) \eps^{d/2} M \| u_h-P_h u\|_{H^1_{\eps}(\D)} \| u-u_h\|_{H^1_{\eps}(\D)} + \eps^{d/2} M \| u_h - P_h u \|_{H^1_{\eps}(\D)} \| u-u_h\|_{H^1_{\eps}(\D)} \\
 &\lesssim& \eps^{d/2} M \| u_h - P_h u \|_{H^1_{\eps}(\D)} \| u-u_h\|_{H^1_{\eps}(\D)}.
 \end{eqnarray*}
Combining the estimates for $\mbox{I}$, $\mbox{II}_1$, $\mbox{II}_2$, $\mbox{II}_3$, $\mbox{II}_4$ and $\mbox{III}$ and dropping the lower order terms (in $\eps$ or $M$) we obtain 
\begin{eqnarray*}
\lefteqn{  \| u_h-P_h u \|_{H^1_{\eps}(\D)}^2  }\\
&\lesssim& 
\eps^{d/2}\, \| u_h-P_h u\|_{H^1_{\eps}(\D)}^2  
+\,\,
M^2 \, \| u_h- u \|_{H^1_{\eps}(\D)}^2
+
\eps^{d/2} M \| u-u_h \|_{H^1_{\eps}(\D)}  \| u_h-P_h u\|_{H^1_{\eps}(\D)} .
\end{eqnarray*}
Using the Young inequality on the last term to absorb the $\| u_h-P_h u\|_{H^1_{\eps}(\D)} $-contributions into the left hand side, we end up with 
\begin{eqnarray*}
 \| u_h-P_h u \|_{H^1_{\eps}(\D)}^2 
&\lesssim&
M^2 \| u-u_h \|_{H^1_{\eps}(\D)}^2.
\end{eqnarray*}
Taking the square root proves the result.
\end{proof}

\subsection{Error estimates for discrete minimizers}
\label{subsec:combination-estimates}

We are now ready to combine the previous findings.

\begin{theorem}\label{thm:main-H1-est}
Assume \ref{A1}-\ref{A3} and let $u \in \S$ be a quasi-isolated ground state with eigenvalue $\lambda$. Then, there exists a constant $c^{\ast}>0$ (independent of $h$ and $\eps$) such if the mesh size fulfills $h \,\le  c^{\ast}\, \eps \, \eta(\eps)\,\eps^{d/2}$, then there exists unique discrete minimizer $(\lambda_h ,u_h) \in \R \times (V_h \cap T_{\ci u} \S)$ such that 
\begin{align*}
u_h\in \S \qquad \mbox{and } \| (\lambda,u) - (\lambda_h,u_h) \|_{\eps} \le \tfrac{\eta(\eps)}{\eps} \eps^{d/4}
\end{align*}
and satisfying the discrete first- and second-order conditions
\begin{align*}
\langle E^{\prime}(u_h) , v_h \rangle \,\, = \,\, \lambda_h \, (u_h , v_h)_{L^2(\D)} \qquad \mbox{for all } v_h \in V_h
\end{align*}
and
\begin{align*}
\langle (E^{\prime\prime}(u_h) -\lambda_h \mathcal{I}) v_h , v_h \rangle \,\,\, \gtrsim \,\,\, \eta(\eps) \, \| v_h \|_{H^1_{\eps}(\D)}^2 \qquad \mbox{for all } v_h \in V_h \cap H_{u} \S,
\end{align*}
Furthermore, the error is bounded by
\begin{align*}
\| u - u_h \|_{H^1_{\eps}(\D)} 
\,\,\,\lesssim \,\,\,  \inf_{v_h \in V_h} \| u - v_h \|_{H^1_{\eps}(\D)}
\,\,\,\lesssim \,\,\, \tfrac{h}{\eps^2}.
\end{align*}
\end{theorem}

\begin{proof}
We apply Theorem \ref{thrm:loc-existence:u_h} for some sufficiently small $M$ to be specified later and the mesh size condition $h \le c^{\ast} \,\eps \,\eta(\eps) \,\eps^{d/2} M$ to obtain existence of a suitable discrete minimizer $u_h$ which fulfills the first and second-order conditions for (local) minimizers and the abstract bound 
$$
\| (\lambda,u) - (\lambda_h,u_h) \|_{\eps} \,\,\,\le\,\,\, \tfrac{\eta(\eps)}{\eps} \, \eps^{d/4} M. 
$$
We split the error as
\begin{align*}
\| u - u_h \|_{H^1_{\eps}(\D)} \,\,\, \lesssim \,\,\, \| u - P_h u \|_{H^1_{\eps}(\D)} + \| P_h u - u_h \|_{H^1_{\eps}(\D)}.\end{align*}
For the first term we apply Lemma \ref{lem:H1-ritz-proj-est} and for the second term Lemma \ref{lem:defect-est}. This yields, for some constant $C>0$,
\begin{align*}
\| u - u_h \|_{H^1_{\eps}(\D)} \,\,\, &\le \,\,\, C \left( \inf_{v_h \in V_h}  \| u - v_h \|_{H^1_{\eps}(\D)}+ M  \| u-u_h \|_{H^1_{\eps}(\D)} \right).
\end{align*}
Hence, for $M= \tfrac{1}{2C}$, we can absorb the corresponding term into the left hand side. Together with Conclusion \ref{conc:H1-ritz-proj-est-gs}, this proves the error estimate. 
\end{proof}

The next lemma relates the energy error to the error $u-u_h$.
\begin{lemma}\label{lem:energy-diff-quadratic}
Assume \ref{A1}-\ref{A3} and let $(\lambda,u) \in \R \times \S$ be an eigenpair, i.e., $E^{\prime}(u) = \lambda \calI u$. Then, for any $v_h\in \S$, it holds 
\begin{eqnarray*}
 \lefteqn{ E(v_h) - E(u) } \\
  &\le& \tfrac{1}{2}\,a_\eps(v_h-u,v_h-u)
     - \tfrac{\lambda}{2}\,\| v_h-u \|_{L^2(\D)}^2 + \tfrac{3\beta}{\eps^2}\int_0^1 \int_\D (1-t) \,|u+t (v_h-u)|^2\,|v_h-u|^2\dx \dt.
\end{eqnarray*}
If $u$ is a ground state, then $\lambda\ge 0$ and $E(v_h) \ge E(u)$ and the estimate simplifies to 
\begin{eqnarray*}
E(v_h) - E(u) 
  &\le& \tfrac{1}{2}\,a_\eps(v_h-u,v_h-u) + \tfrac{3\beta}{\eps^2}\int_0^1 \int_\D (1-t) \,|u+t (v_h-u)|^2\,|v_h-u|^2\dx \dt.
\end{eqnarray*}
\end{lemma}

\begin{proof}
Define the Lagrangian $\mathcal{L}(v) := E(v) - \tfrac{\lambda}{2}\|v\|_{L^2(\D)}^2$. For $u,v_h\in \S$ we get
\begin{align*}
  E(v_h) - E(u) = \mathcal{L}(v_h) - \mathcal{L}(u).
\end{align*}
Let $e_h := v_h - u$ and consider the segment $u_t := u + t e_h$, $t\in[0,1]$. Since $\mathcal{L}'(u)=0$ by the eigenvalue equation for $u$, Taylor expansion yields
\begin{align*}
  \mathcal{L}(v_h) - \mathcal{L}(u)
  = \int_0^1 (1-t)\,\mathcal{L}''(u_t)[e_h,e_h]\dt,
\end{align*}
where $\mathcal{L}''(w)[e_h,e_h]$ denotes the second Fr\'echet derivative of $\mathcal{L}$ at $w$
applied twice to $e_h$. Next, we compute $\mathcal{L}''(w)[e_h,e_h]$. For that, recall $E(v) = \tfrac{1}{2}\,a_\eps(v,v) + \tfrac{\beta}{4\eps^2}\int_\D |v|^4\dx$, which gives for any $w \in H^1_0(\D)$ 
\begin{align*}
  E''(w)[e_h,e_h]
  &= a_\eps(e_h,e_h) + \tfrac{\beta}{\eps^2}\int_\D |w|^2\,|e_h|^2 + 2 (\Re(w \overline{e_h}))^2 \dx \\
  &\le a_\eps(e_h,e_h) + \tfrac{3\beta}{\eps^2}\int_\D |w|^2\,|e_h|^2 \dx
\end{align*}
On the other hand, the derivative of $\tfrac{1}{2}\|v\|_{L^2(\D)}^2$ is given by $\bigl(\tfrac{1}{2}\|\cdot\|_{L^2}^2\bigr)''(w)[e_h,e_h] = \|e_h\|_{L^2(\D)}^2$. Hence
\begin{align*}
  \mathcal{L}''(w)[e_h,e_h]
  = E''(w)[e_h,e_h] - \lambda \|e_h\|_{L^2(\D)}^2
  \le a_\eps(e_h,e_h)
    + \tfrac{3\beta}{\eps^2}\int_\D |w|^2\,|e_h|^2\dx
    - \lambda \|e_h\|_{L^2(\D)}^2.
\end{align*}
Setting $w = u_t = u + t e_h$ and inserting into the Taylor formula gives
\begin{eqnarray*}
  \mathcal{L}(v_h) - \mathcal{L}(u)
  &\le& \int_0^1 (1-t)\Big(
       a_\eps(e_h,e_h)
       + \tfrac{3\beta}{\eps^2}\int_\D |u_t|^2\,|e_h|^2\dx
       - \lambda \|e_h\|_{L^2(\D)}^2
     \Big)\dt \\
  &=& \tfrac{1}{2}\,a_\eps(e_h,e_h)
     - \tfrac{\lambda}{2}\,\|e_h\|_{L^2(\D)}^2
     + \tfrac{3\beta}{\eps^2}\int_0^1 (1-t)
        \int_\D |u+t e_h|^2\,|e_h|^2\dx \dt.
\end{eqnarray*}
Recalling that $\mathcal{L}(v_h) - \mathcal{L}(u) = E(v_h) - E(u)$ finishes the proof.
\end{proof}
At first glance, the energy error inherits the same restrictions on the mesh size as the error $\| u - u_h \|_{H^1_{\eps}(\D)}$. However, note that Lemma \ref{lem:energy-diff-quadratic} holds for any $u_h \in V_h \cap \S$. In particular, we can directly apply it to $\tfrac{1}{\| I_h u \|_{L^2(\D)}} I_h u$ for some suitable interpolation operator $I_h$. This can be bounded without the local existence result and smallness from Theorem \ref{thrm:loc-existence:u_h}. We obtain the following.
\begin{conclusion}[Energy error]\label{conclusion:energy-error}
Assume \ref{A1}-\ref{A3} and let $u \in \S$ denote a ground state and $u_h \in V_h \cap \S$ an arbitrary discrete ground state. Then there exists a generic constant $C^*>0$ (independent of $h$ and $\eps$) such that if $h \le C^{\ast} \eps$ it holds
\begin{align*}
0 \,\,\,\, \le \,\,\,\, E(u_h) - E(u) \,\,\,\,\lesssim\,\,\,\, 
\inf_{v_h \in V_h} \| u - v_h \|_{H^1_{\eps}(\D)}^2 + \tfrac{1}{\varepsilon^2} \inf_{v_h \in V_h} \| u - v_h \|_{L^4(\D)}^4 
\,\,\,\,\lesssim\,\,\,\, \left( \tfrac{h}{\eps^2} \right)^2.
\end{align*}
\end{conclusion}

\begin{proof}
We apply Lemma \ref{lem:energy-diff-quadratic} to  $v_h = \tfrac{1}{\| I_h u \|_{L^2(\D)}} I_h u$ where $I_h : H^1_0(\D) \rightarrow V_h$ denotes, e.g., the Ern-Guermond-quasi-interpolation operator \cite[Section~5]{ErnGuermond2017}, which admits the classical approximation and stability properties ($d\le 3$)
\begin{align}
\label{stability-interpolation-estimates-Ih}
\| v - I_h v \|_{L^{p}(\D)} \lesssim h \| \nabla v \|_{L^{p}(\D)}
\qquad \mbox{and}
\qquad | v - I_h v |_{H^{1}(\D)} \lesssim h |v |_{H^2(\D)} 
\end{align}
for all $v\in H^2(\D) \cap H^1_0(\D)$ and $1 \le p \le 6$. By these properties and the reverse triangle inequality we have
\begin{align*}
\|I_h u\|_{L^2(D)} &\ge \|u\|_{L^2(\D)} - \|I_h u - u\|_{L^2(\D)} 
\ge 1 - C_{\mathrm{int}}\, h\, \|\nabla u\|_{L^2(D)} \\
&\ge 1 - C_{\mathrm{int}}\, C\, \tfrac{h}{\eps},
\end{align*}
where $C_{\mathrm{int}}>0$ is the constant from the interpolation estimate and $C$ the constant in the stability estimate $\| \nabla u\| \le C \tfrac{1}{\eps}$. Hence, for $\tfrac{h}{\eps} \le \tfrac{1}{2} (C_{\mathrm{int}}\, C)^{-1}$ we have
\begin{align}
\label{lower-bound_Ih}
\|I_h u\|_{L^2(D)} \ge \tfrac{1}{2}.
\end{align}
Now let $e_h := u - \tfrac{1}{\| I_h u \|_{L^2(\D)}} I_h u$, then by $u \in \S$ we have the pointwise bound for the error as
\begin{eqnarray*}
|e_h| &=& \left| \frac{\left( \| I_h u \|_{L^2(\D)} - \| u \|_{L^2(\D)} \right) u  \,+\, (u - I_h u) }{  \| I_h u \|_{L^2(\D)} } \right| 
\,\,\overset{\eqref{lower-bound_Ih}}{\le}\,\,
\tfrac{1}{2} \left( \| I_h u - u \|_{L^2(\D)} \, |u| \,+\, |u - I_h u| \right) \\
&\overset{\mbox{\tiny Lem.~\ref{lem:stability-bounds-groundstates}}}{\le}& C \| I_h u - u \|_{L^2(\D)} \,+\, \tfrac{1}{2} | I_h u - u |.
\end{eqnarray*}
From this inequality and the natural embeddings we conclude that, for $p\ge 2$,
\begin{align}
\label{est-fe-interp}
\| e_h \|_{L^p(\D)} \,\,\lesssim\,\, \|  I_h u - u  \|_{L^p(\D)} 
\qquad
\mbox{and}
\qquad
\| e_h \|_{H^1_{\eps}(\D)} \,\,\lesssim\,\, \|  I_h u - u  \|_{H^1_{\eps}(\D)}. 
\end{align}
With $v_h = \tfrac{1}{\| I_h u \|_{L^2(\D)}} I_h u$, Lemma \ref{lem:energy-diff-quadratic} together with the continuity of \(a_\varepsilon(\cdot,\cdot)\) yields 
\begin{eqnarray*}
E(u_h)-E(u)
&\le&\tfrac{1}{2} a_\eps(e_h,e_h)
+\tfrac{3\beta}{\eps^2}
\int_0^1\!\!\int_D (1-t)\,
|u+t \hspace{1pt} e_h|^2\,|e_h|^2\dx\dt \\
&\lesssim& \| e_h \|_{H_\varepsilon^1(\D)}^2 
+ \tfrac{1}{\varepsilon^2} \int_0^1\!\!\int_{\D} (1-t)\, |u+t \hspace{1pt}e_h|^2\,|e_h|^2\dx\dt \\
&\overset{\eqref{est-fe-interp}}{\lesssim}&  \|  I_h u - u  \|_{H^1_{\eps}(\D)}^2 
+ \tfrac{1}{\varepsilon^2} \int_0^1\!\!\int_{\D} (1-t)\, |u+t \hspace{1pt}e_h|^2\,|e_h|^2\dx\dt \\
&\overset{\mbox{\tiny Lem.~\ref{lem:stability-bounds-groundstates}}}{\lesssim}&  \|  I_h u - u  \|_{H^1_{\eps}(\D)}^2 
+ \tfrac{1}{\varepsilon^2} \| e_h \|_{L^2(\D)}^2 + \tfrac{1}{\varepsilon^2} \| e_h \|_{L^4(\D)}^4 \\
&\overset{\eqref{est-fe-interp}}{\lesssim}& 
\|  I_h u - u  \|_{H^1_{\eps}(\D)}^2 + \tfrac{1}{\varepsilon^2} \|  I_h u - u  \|_{L^4(\D)}^4 \\
&\overset{\eqref{stability-interpolation-estimates-Ih}}{\lesssim}& 
\tfrac{1}{\eps^2} \left( \tfrac{h^2}{\eps^2} + \tfrac{h^4}{\eps^4} \right) \,\,\lesssim \,\,\, \tfrac{1}{\eps^2}  \tfrac{h^2}{\eps^2},
\end{eqnarray*}
where we used the stability estimate $\| \nabla u \|_{L^4(\D)} \lesssim \tfrac{1}{\eps}$ from Lemma~\ref{lem:stability-bounds-groundstates} in the last step. Note that the interpolation errors in the intermeidate step can be replaced by best-approximation errors (up to constants), because $I_h$ is both $L^4$-stable and $H^1_{\eps}$-stable. This finishes the proof.
\end{proof}

\section{Numerical experiments}
\label{sec:num-experiments}

In this section, we present a series of numerical experiments 
supporting the main theoretical findings of this work. Specifically, we illustrate the vortex structure of the ground state as the system approaches the rapid-rotation Thomas--Fermi regime, we verify the stability estimate $\|u\|_{H^1_\varepsilon(\mathcal{D})} \lesssim 
\varepsilon^{-1}$ implied by Lemma~\ref{lem:stability-bounds-groundstates}, and we examine both the resolution condition and the asymptotic error behavior established in Theorem~\ref{theorem-main-result}.

Naturally, the computation of the ground state constitutes the starting point for all subsequent experiments. As already defined in~\eqref{definition-groundstate}, the 
ground state arises as the global minimizer of the Gross--Pitaevskii energy functional over the unit sphere $\mathbb{S}$,
\begin{equation*}
    E(u) = \inf_{v \in \mathbb{S}}\, E(v), \qquad 
    \mathbb{S} := \bigl\{ v \in H^1_0(\mathcal{D}) \;\big|\; 
    \|v\|_{L^2(\mathcal{D})} = 1 \bigr\}.
\end{equation*}
To compute it numerically, the energy functional is 
discretized using conforming $\mathbb{P}^1$ Lagrange finite elements  on a triangular mesh of mesh size $h$, implemented in the 
\textsc{FEniCSx} framework~\cite{baratta_2023_10447666,Scroggs_2022,
Scroggs2022,alnaes2013unifiedformlanguagedomainspecific}. 
The resulting constrained minimization problem is then solved iteratively using 
the Riemannian Sobolev conjugate gradient (RSCG) scheme 
of~\cite{AHYY24}, equipped with an energy-adaptive Sobolev metric 
and the Polak--Rib\`ere momentum parameter. In addition to that, following 
\cite{Bao-et-al-2005}, the iterative solver is initialized with the 
$L^2$-normalized interpolant of
\begin{align*}
    u_0(x_1,x_2)
    \;:=\;(x_1 + \ci x_2)\,
    e^{-(x_1^2+x_2^2)/2},
\end{align*}
and terminated when the energy decrease between consecutive iterates 
satisfies
\[
|E(u_h^{n+1}) - E(u_h^n)| < 10^{-11}.
\]
All experiments are carried out in two spatial dimensions ($d=2$) on the 
rectangular domain $\mathcal{D} = [-0.9,\,0.9]\times[-1.4,\,1.4]$. We choose the physical parameters so as to place the system in a well-controlled yet physically meaningful regime. The angular velocity is 
fixed at $\Omegaeps = 9$, below the critical value 
$\Omega_{\mathrm{crit}}$, so that Assumption \ref{A3} is satisfied and the 
effective trapping potential $V_R$ remains confining throughout $\mathcal{D}$. 
The particle interaction strength is set to $\beta = 10$, of the same order 
as $\Omegaeps$, while the trapping 
potential is taken to be slightly anisotropic and harmonic,
\begin{align*}
\Veps(x) \;=\; 26\cdot \!\left(1.25\,x_1^2 + 0.98\,x_2^2\right).
\end{align*}
The mild anisotropy of $\Veps$ aids in minimizing the rotational symmetry of the trap. This is particularly important for the numerical error estimates, since on a fully symmetric trap the discrete minimizer may settle into a rotated but energetically equivalent vortex configuration. This would introduce a spurious growth in the error due to vortex misalignment rather than a true resolution deficiency. Hence, by applying an anisotropic potential we ensure that the observed pre-asymptotic regime genuinely reflects the resolution condition of Theorem~\ref{theorem-main-result} and is not an artifact of misalignment approximations.

The complete numerical implementation is openly available at \url{https://github.com/chrisplh258/gpe-fem}. All computations were carried out on the Dardel supercomputer at the PDC Center for High Performance Computing, KTH Royal Institute of Technology, using resources provided by the National Academic Infrastructure for Supercomputing in Sweden (NAISS), partially funded by the Swedish Research Council through grant agreement no.\ 2022-06725. Finally, AI based tools were used to assist in optimizing parts of the simulation code, fully reviewed and verified by the authors.

\subsection{Vortex structure and stability of the ground state}
Resolving the ground state becomes increasingly demanding as 
$\varepsilon$ decreases: the vortex cores shrink, their number grows, 
and the energy landscape near the ground state flattens. Each effect 
independently drives the need for finer spatial resolution. 
Figure~\ref{fig:density} illustrates the first two of these effects, 
displaying the density $\rho = |u_{\mathrm{ref}}|^2$ of the reference 
ground state, computed on a fine mesh of size $h_{\mathrm{ref}} = 2^{-10}$, 
for four decreasing values of $\varepsilon$. As $\varepsilon$ decreases 
toward the Thomas--Fermi regime, the number of vortices grows rapidly 
and their cores shrink to size $\mathcal{O}(\varepsilon)$.

\begin{figure}[ht]
\centering
\begin{subfigure}{0.24\textwidth}
\centering
\includegraphics[width=\linewidth]{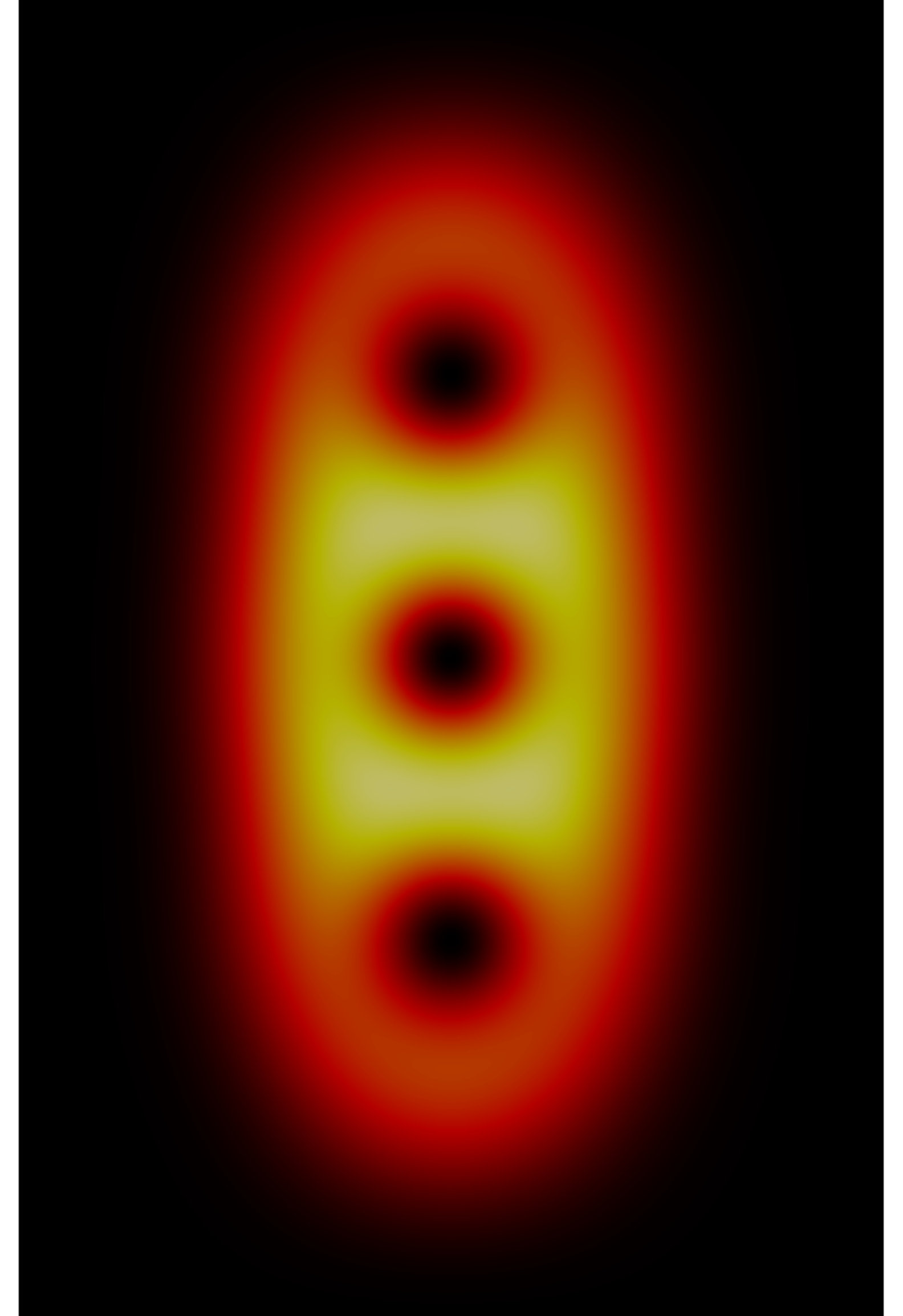}
\caption{$\epsilon=0.40$}
\end{subfigure}
\begin{subfigure}{0.24\textwidth}
\centering
\includegraphics[width=\linewidth]{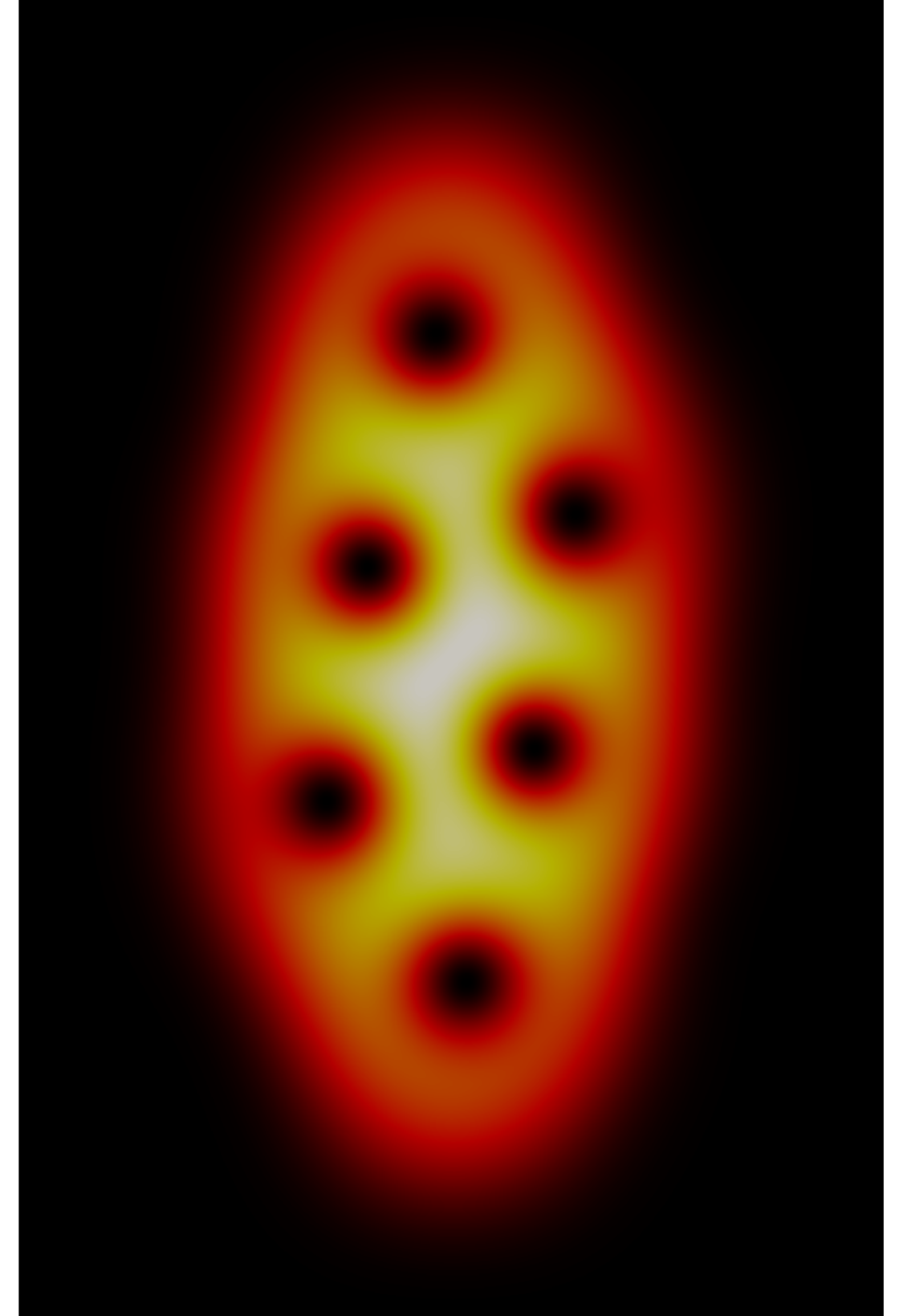}
\caption{$\epsilon=0.25$}
\end{subfigure}
\begin{subfigure}{0.24\textwidth}
\centering
\includegraphics[width=\linewidth]{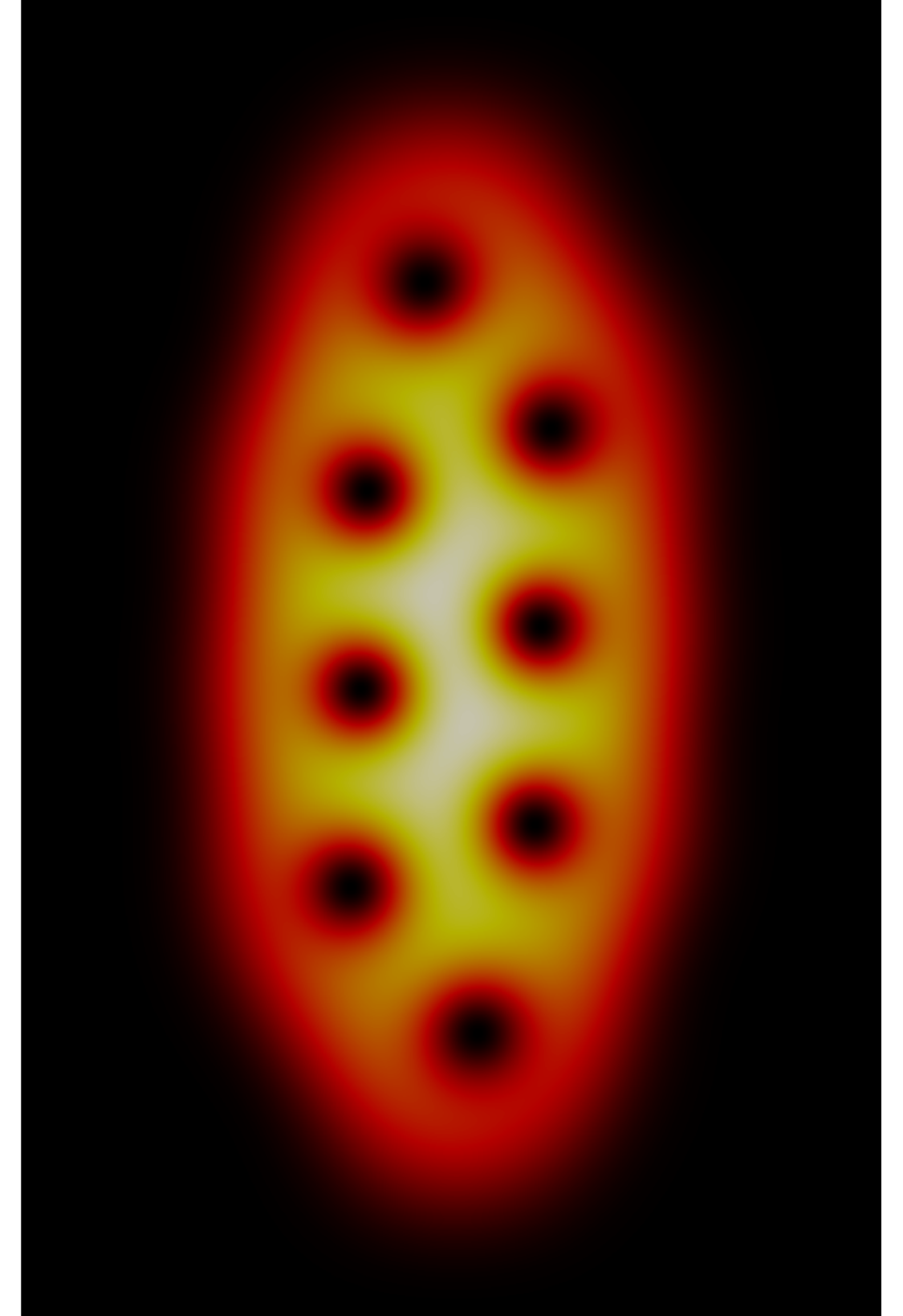}
\caption{$\epsilon=0.20$}
\end{subfigure}
\begin{subfigure}{0.24\textwidth}
\centering
\includegraphics[width=\linewidth]{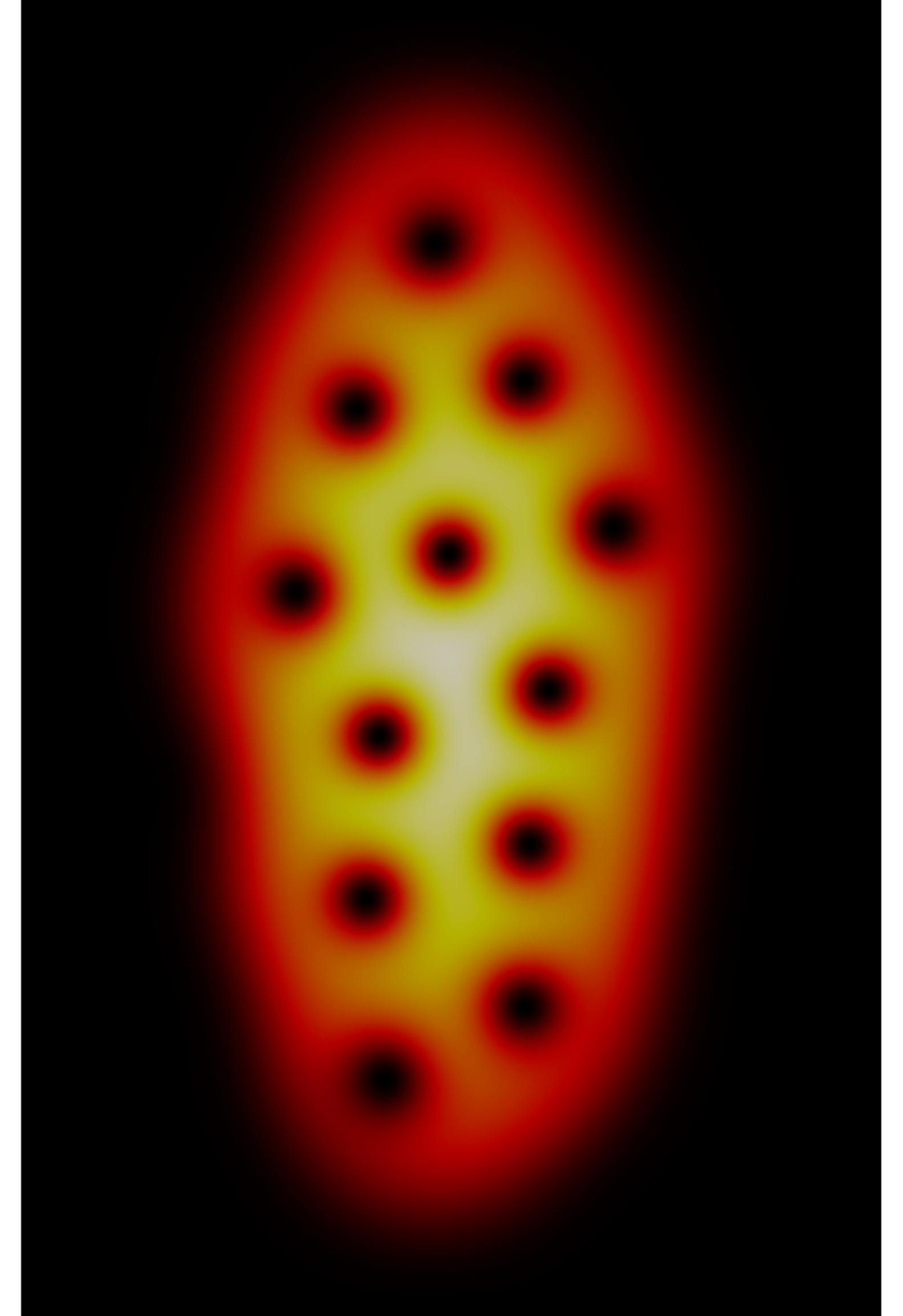}
\caption{$\epsilon=0.15$}
\end{subfigure}
\caption{Density $\rho = |u_{\mathrm{ref}}|^2$ of the reference 
ground state for $\varepsilon \in \{0.15, 0.20, 0.25, 0.40\}$. 
As $\varepsilon$ decreases, the vortex cores shrink and the vortex 
lattice becomes increasingly dense.}
\label{fig:density}
\end{figure}

From there we turn to the stability bound of 
Lemma~\ref{lem:stability-bounds-groundstates}, which asserts that $\|\nabla u\|_{L^2(\mathcal{D})} \lesssim 
\varepsilon^{-1}$, 
or equivalently that $\|u\|_{H^1_\varepsilon(\mathcal{D})} \lesssim \varepsilon^{-1}$, and confirm this bound numerically by tracking 
the $H^1_\varepsilon$ norm of the reference ground state as a 
function of $\varepsilon$. As seen in Figure~\ref{fig:h1-vs-eps}, the norm scales in agreement with the theoretical $\mathcal{O}(\varepsilon^{-1})$ rate.

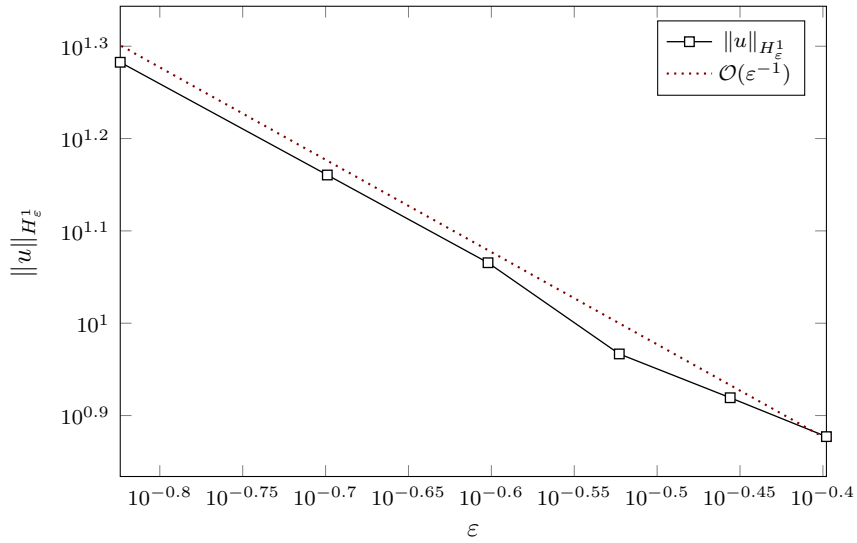
\begin{figure}[ht]
\centering
\begin{tikzpicture}
\begin{axis}[
    width=0.7\textwidth,
    height=0.5\textwidth,
    xmode=log,
    ymode=log,
    xmin=0.15, xmax=0.40,
    xlabel={$\varepsilon$},
    ylabel={$\|u\|_{H^1_\varepsilon}$ },
    legend pos=north east,
    tick label style={font=\scriptsize},
    label style={font=\small},
    legend style={font=\scriptsize},
    minor tick num=0
]

\addplot[black, line width=0.5pt, mark=square*, mark size=1.8pt,
         mark options={fill=white}]
    table[x=Eps, y=H1_eps, col sep=comma]{plots/h1_eps.csv};
\addlegendentry{$\|u\|_{H^1_\varepsilon}$}


\addplot[Maroon, thick, dotted, domain=0.13:0.45, samples=200]
    {3.0*x^(-1)};
\addlegendentry{$\mathcal{O}(\varepsilon^{-1})$}

\end{axis}
\end{tikzpicture}
\caption{Scaled $H^1_\varepsilon$ norm
for a varying $\varepsilon$.}
\label{fig:h1-vs-eps}
\end{figure}

\subsection{Resolution conditions}
In this section we numerically investigate the resolution condition 
of Theorem~\ref{theorem-main-result}. A reference solution 
$u_{\mathrm{ref}}$ is computed on a mesh of size $h = 2^{-10}$, 
and errors are then measured on a sequence of coarser meshes of size 
$h = 2^{-k}$ for $k \in \{2,\,3,\,3.5,\,\ldots,\,8.5\}$. To ensure that the observed errors reflect discretization effects 
alone, and not the choice of initial condition, the 
discrete minimizer on each coarse mesh is computed by initializing 
the Riemannian gradient method with the $L^2$-projection of 
$u_{\mathrm{ref}}$ onto the respective coarse finite element space. 
Three quantities are tracked as functions of $h$: the energy error 
$|E_{\mathrm{ref}} - E_h|$, the scaled FEM error 
$\|u_{\mathrm{ref}} - u_h\|_{H^1_\varepsilon(\mathcal{D})}$, and 
the best-approximation error 
$\|u_{\mathrm{ref}} - P_h u_{\mathrm{ref}}\|_{H^1_\varepsilon(\mathcal{D})}$,
where $P_h$ denotes the $H^1_\varepsilon$-projection onto the finite element space $V_h$, 
giving the closest possible approximation to $u_{\mathrm{ref}}$ in 
that space. 

Figure~\ref{fig:error-comparison} reveals a clear contrast between 
the two error quantities. Since the energy is a continuous functional, no pre-asymptotic 
regime is expected for the energy error, a fact confirmed theoretically 
by Conclusion \ref{conclusion:energy-error}. The energy error thus, enters its asymptotic regime $\mathcal{O}(h^2)$ already on coarse 
meshes, with no visible pre-asymptotic threshold. The scaled $H^1_\varepsilon$ error, by contrast, exhibits a 
pronounced pre-asymptotic plateau, shifting to increasingly finer 
scales as $\varepsilon$ decreases. This is a clear indicator of a 
resolution condition that must be satisfied before the asymptotic 
regime is entered, and which becomes increasingly stricter as 
$\varepsilon$ decreases.


\begin{figure}[H]
\centering
\makebox[\textwidth][c]{%


\begin{subfigure}{0.49\textwidth}
\centering
\begin{tikzpicture}
\begin{axis}[
    width=\linewidth,
    height=0.9\linewidth,
    xmode=log,
    ymode=log,
    xmin=2e-3, xmax=1,
    ymin=1e-5, ymax=1e2,
    xlabel={$h$},
    ylabel={$|E-E_h|$},
    legend pos=south east,
    minor tick num=0,
    tick label style={font=\scriptsize},
    label style={font=\small},
    legend style={font=\scriptsize}
]

\addplot[Maroon, thick, dotted, domain=1e-3:5e-1, samples=200]
{1e1*x^2};
\addlegendentry{$\mathcal{O}(h^2)$}

\addplot[
    line width=0.5pt,
    color=black!35,
    mark=*,
    mark size=1.8pt,,
]
table[
    x=h,
    y=energy_error,
    col sep=comma
]
{plots/errors_eps0.15_20260502_170737.csv};
\addlegendentry{$\varepsilon = 0.15$}

\addplot[
    line width=0.5pt,
    color=black!55,
    mark=square*,
    mark size=1.8pt,
]
table[
    x=h,
    y=energy_error,
    col sep=comma
]
{plots/errors_eps0.2_20260430_135308.csv};
\addlegendentry{$\varepsilon = 0.20$}

\addplot[
    line width=0.5pt,
    color=black!75,
    mark=triangle*,
    mark size=2pt,
]
table[
    x=h,
    y=energy_error,
    col sep=comma
]
{plots/errors_eps0.25_20260430_135819.csv};
\addlegendentry{$\varepsilon = 0.25$}

\addplot[
    line width=0.5pt,
    black,
    mark=diamond*,
    mark size=2pt,
]
table[
    x=h,
    y=energy_error,
    col sep=comma
]
{plots/errors_eps0.4_20260503_214818.csv};
\addlegendentry{$\varepsilon = 0.40$}
\end{axis}
\end{tikzpicture}
\caption{Energy error $|E-E_h|$.}
\label{fig:energy-error}
\end{subfigure}%
\hspace{0.01\textwidth}%

\begin{subfigure}{0.49\textwidth}
\centering
\begin{tikzpicture}
\begin{axis}[
    width=\linewidth,
    height=0.9\linewidth,
    xmode=log,
    ymode=log,
    xmin=2e-3, xmax=1,
    ymin=1e-2, ymax=1e2,
    xlabel={$h$},
    ylabel={$\|u-u_h\|_{H^1_\epsilon}$},
    legend pos=south east,
    minor tick num=0,
    legend style={font=\footnotesize},
    tick label style={font=\footnotesize},
    label style={font=\footnotesize},
]

\addplot[Maroon, thick, dotted, domain=1e-3:5e-1, samples=200]
{1e1*x};
\addlegendentry{$\mathcal{O}(h)$}

\addplot[
    line width=0.5pt,
    color=black!35,
    mark=*,
    mark size=1.8pt,
]
table[
    x=h,
    y=H1_fem_error,
    col sep=comma
]
{plots/errors_eps0.15_20260502_170737.csv};
\addlegendentry{$\varepsilon = 0.15$}

\addplot[
    line width=0.5pt,
    color=black!55,
    mark=square*,
    mark size=1.8pt,
]
table[
    x=h,
    y=H1_fem_error,
    col sep=comma
]
{plots/errors_eps0.2_20260430_135308.csv};
\addlegendentry{$\varepsilon = 0.20$}

\addplot[
    line width=0.5pt,
    color=black!75,
    mark=triangle*,
    mark size=2pt,
]
table[
    x=h,
    y=H1_fem_error,
    col sep=comma
]
{plots/errors_eps0.25_20260430_135819.csv};
\addlegendentry{$\varepsilon = 0.25$}

\addplot[
    line width=0.5pt,
    black,
    mark=diamond*,
    mark size=2pt,
]
table[
    x=h,
    y=H1_fem_error,
    col sep=comma
]
{plots/errors_eps0.4_20260503_214818.csv};
\addlegendentry{$\varepsilon = 0.4$}

\end{axis}
\end{tikzpicture}
\caption{Scaled FEM error $\|u-u_h\|_{H^1_\epsilon}$.}
\label{fig:fem-error}
\end{subfigure}%
}

\caption{Energy error (left) and scaled $H^1_\varepsilon$ finite element error
(right) as functions of mesh size $h$, for
$\varepsilon \in \{0.15, 0.20, 0.25, 0.40\}$.}
\label{fig:error-comparison}
\end{figure}
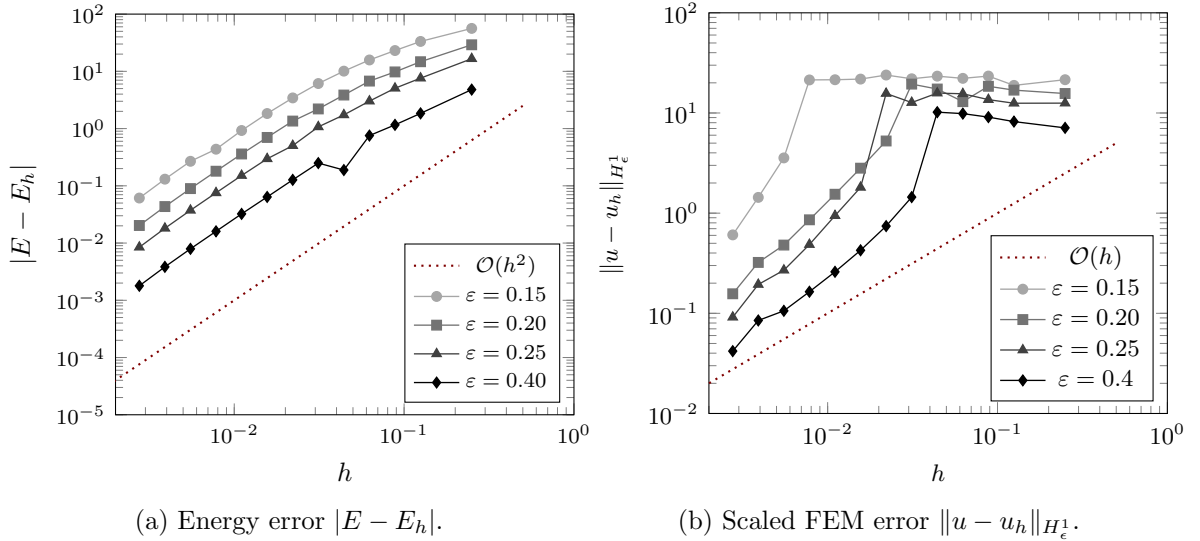

Figure~\ref{fig:all-errors} makes this resolution condition more 
precise. For each value of $\varepsilon$, the FEM error is plotted 
alongside the best-approximation error 
$\|u_{\mathrm{ref}} - P_h u_{\mathrm{ref}}\|_{H^1_\varepsilon}$, 
with vertical lines marking $h = \varepsilon$, $h = \varepsilon^2$, 
and $h = \varepsilon^3$. As expected, the best-approximation error achieves the 
optimal rate $\mathcal{O}(h)$ at all mesh sizes, confirming that $V_h$ possesses sufficient approximation properties. The finite element error, however, does not enter its asymptotic 
regime until a resolution condition strictly stronger than 
$h \lesssim \varepsilon^2$ is met, with the required threshold indicating higher powers of $\varepsilon$. This is 
consistent with Theorem~\ref{theorem-main-result}, which asserts such a condition of the form 
$h \,\,\le\,\,  c^{\ast}\, \mucrit \, (1- \tfrac{\lambda_{1}}{\lambda_{2}})\,\eps^{(d+2)/2} \,$, 
whose exact $\varepsilon$-dependence the experiments confirm 
qualitatively but do not attempt to quantify.


\newcommand{\errorplot}[1]{%
\pgfplotstableread[col sep=comma]{#1}\datatable
\pgfplotstablegetelem{0}{epsilon}\of\datatable
\pgfmathsetmacro{\epsval}{\pgfplotsretval}
\pgfmathsetmacro{\epsone}{\epsval}
\pgfmathsetmacro{\epstwo}{\epsval * \epsval}
\pgfmathsetmacro{\epsthree}{\epsval * \epsval * \epsval}
\begin{tikzpicture}
\begin{axis}[
    width=\linewidth,
    height=\linewidth,
    xmode=log,
    ymode=log,
    xmin=2e-3, xmax=1,
    ymin=1e-2, ymax=1e3,
    xlabel={$h$},
    ylabel={$\|u_{\mathrm{ref}}-u_h\|_{H^1_\varepsilon}$,\\
        $\|u_{\mathrm{ref}}-P_h u_{\mathrm{ref}}\|_{H^1_\varepsilon}$},
    minor tick num=0,
    legend pos=north west,
    legend style={font=\footnotesize},
    tick label style={font=\footnotesize},
    legend style={
        font=\footnotesize,
        legend columns=2,
        column sep=0.8em,
    },    
    ylabel style={align=center, font=\footnotesize},
]
\addplot[black, line width=0.8pt, mark=diamond*, mark size=1.5pt]
    table[x=h, y=H1_fem_error, col sep=comma]{#1};
\addlegendentry{$\|u_{\mathrm{ref}}-u_h\|_{H^1_\varepsilon}$}

\addplot[red!80!black, line width=0.7pt, dash pattern=on 5pt off 2pt]
    coordinates {(\epsone,1e-2) (\epsone,1e2)};
\addlegendentry{$\varepsilon$}

\addplot[gray, line width=0.8pt, mark=square*, mark size=1.5pt]
    table[x=h, y=H1_interpolation_error, col sep=comma]{#1};
\addlegendentry{$\|u_{\mathrm{ref}}-P_h u_{\mathrm{ref}}\|_{H^1_\varepsilon}$}

\addplot[orange!90!black, line width=0.7pt, dash pattern=on 4pt off 2pt on 1pt off 2pt]
    coordinates {(\epstwo,1e-2) (\epstwo,1e2)};
\addlegendentry{$\varepsilon^2$}

\addplot[black, thick, dotted, domain=1e-3:5e-1]{1e1*x};
\addlegendentry{$\mathcal{O}(h)$}

\addplot[brown!80!black, line width=0.7pt, dash pattern=on 3pt off 2pt]
    coordinates {(\epsthree,1e-2) (\epsthree,1e2)};
\addlegendentry{$\varepsilon^3$}
\end{axis}
\end{tikzpicture}
}

\begin{figure}[H]
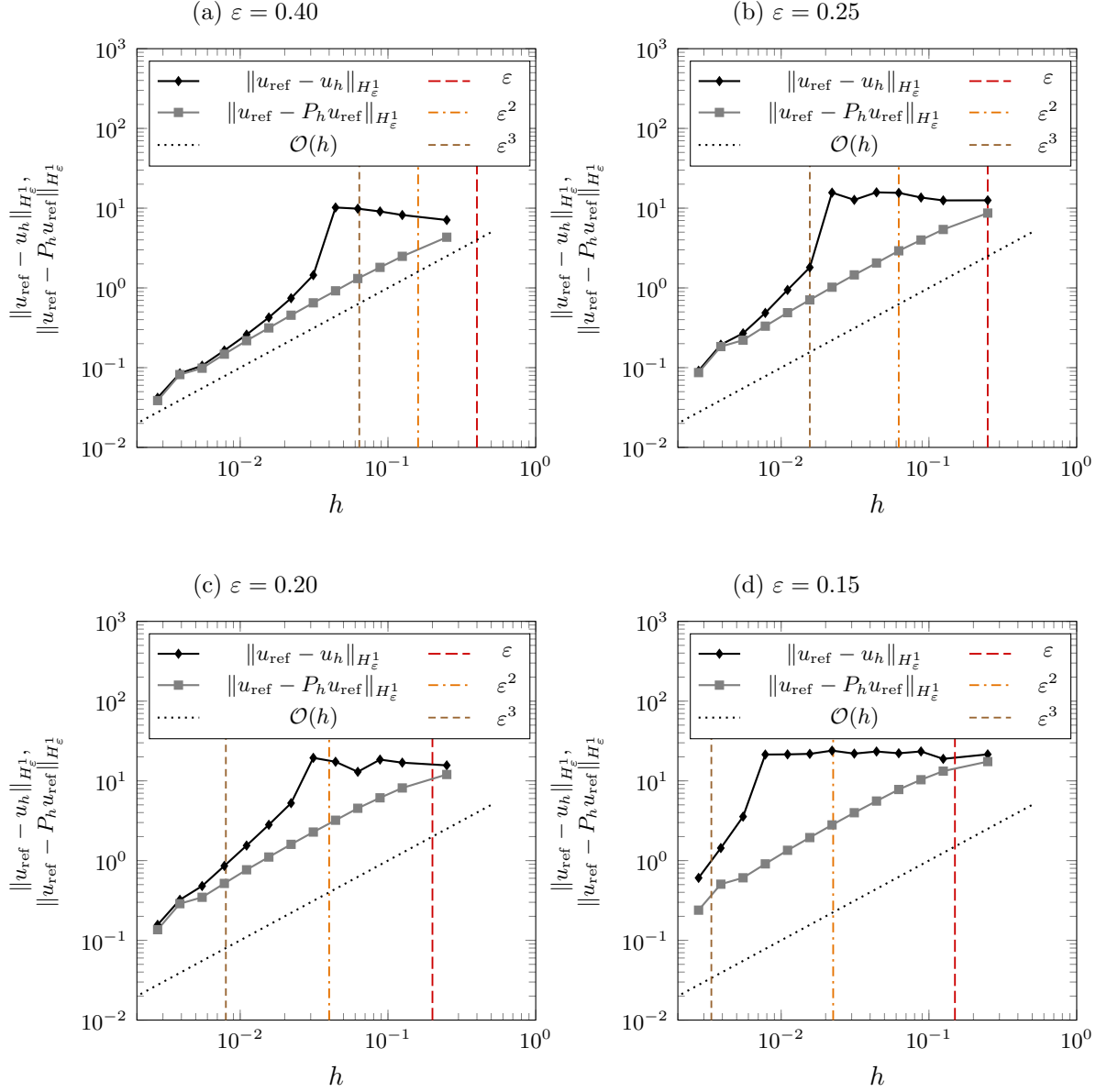

\centering
\makebox[\textwidth][c]{%
\begin{subfigure}{0.47\textwidth}
\centering
\caption{$\varepsilon=0.40$}
\errorplot{plots/errors_eps0.4_20260503_214818.csv}
\end{subfigure}%
\hspace{0.03\textwidth}%
\begin{subfigure}{0.47\textwidth}
\caption{$\varepsilon=0.25$}
\centering
\errorplot{plots/errors_eps0.25_20260430_135819.csv}
\end{subfigure}
}
\vspace{0.2cm}
\makebox[\textwidth][c]{%
\begin{subfigure}{0.47\textwidth}
\centering
\caption{$\varepsilon=0.20$}
\errorplot{plots/errors_eps0.2_20260430_135308.csv}
\end{subfigure}%
\hspace{0.03\textwidth}%
\begin{subfigure}{0.47\textwidth}
\centering
\caption{$\varepsilon=0.15$}
\errorplot{plots/errors_eps0.15_20260502_170737.csv}
\end{subfigure}
}
\caption{Scaled $H^1_\varepsilon$ FEM error
$\|u_{\mathrm{ref}}-u_h\|_{H^1_\varepsilon}$ (black) and
best-approximation error
$\|u_{\mathrm{ref}}-P_h u_{\mathrm{ref}}\|_{H^1_\varepsilon}$
(gray) as functions of $h$, for
$\varepsilon\in\{0.40,\,0.25,\,0.20,\,0.15\}$.
Vertical lines mark $h=\varepsilon$, $h=\varepsilon^2$, and
$h=\varepsilon^3$.}
\label{fig:all-errors}
\end{figure}

The section closes with Figure~\ref{fig:eps-dependence} displaying the scaled $H^1_\varepsilon$ 
finite element error and best-approximation error against $\varepsilon$, 
computed on a fixed, sufficiently fine mesh $h \ll \varepsilon$ so that both quantities lie 
firmly within the asymptotic regime. Both grow at rate 
$\mathcal{O}(\varepsilon^{-2})$ as $\varepsilon \to 0$, confirming 
the $\varepsilon$-explicit asymptotic estimate of 
Theorem~\ref{theorem-main-result}. The close agreement between the 
two curves across the full range of $\varepsilon$ values indicates 
that the discrete minimizer tracks the best-approximation error, with 
no visible influence of the coercivity constant $\eta(\varepsilon)$. 
This is precisely what the theory predicts: once the resolution 
condition $h \,\,\le\,\,  c^{\ast}\, \mucrit \, (1- \tfrac{\lambda_{1}}{\lambda_{2}})\,\eps^{(d+2)/2} \,$ is 
satisfied and the discrete space is fine enough to detect the narrow 
valley of local convexity in the energy landscape, the FEM error 
enters its asymptotic regime and is governed entirely by the 
approximation power of $V_h$, independently of $\eta(\varepsilon)$.


\begin{figure}[H]
\centering
\begin{tikzpicture}
\begin{axis}[
    width=0.7\linewidth,
    height=0.5\linewidth,
    xmode=log,
    ymode=log,
    xmin=0.15, xmax=0.4,
    ymin=1e-2, ymax=0.7,
    xlabel={$\varepsilon$},
    ylabel={$\|u_{\mathrm{ref}}-u_h\|_{H^1_\varepsilon}$,\\
        $\|u_{\mathrm{ref}}-P_h u_{\mathrm{ref}}\|_{H^1_\varepsilon}$},
    legend pos=north east,
    tick label style={font=\scriptsize},
    label style={font=\small},
    ylabel style={align=center},
    legend style={font=\scriptsize},
    minor tick num=0
]

\addplot[Maroon, thick, dotted, domain=0.1:0.5, samples=200]
{0.002*x^(-2)};
\addlegendentry{$\mathcal{O}(\varepsilon^{-2})$}

\addplot[
    black,
    line width=0.5pt,
    mark=square*,
    mark size=1.8pt,
    mark options={fill=white}
]
table [x=epsilon, y=H1_fem_error, col sep=comma]
{plots/eps_vs_error.csv};
\addlegendentry{$\|u-u_h\|_{H^1_\epsilon}$}

\addplot[
    color=black!75,
    line width=0.5pt,
    mark=diamond*,
    mark size=1.8pt
]
table [x=epsilon, y=H1_interpolation_error, col sep=comma]
{plots/eps_vs_error.csv};
\addlegendentry{$\|u_{\mathrm{ref}}-P_h u_{\mathrm{ref}}\|_{H^1_\varepsilon}$}

\end{axis}
\end{tikzpicture}

\caption{
Finite element and interpolation errors in the scaled $H^1_\epsilon$ norm as functions of $\epsilon$.
}
\label{fig:eps-dependence}

\end{figure}
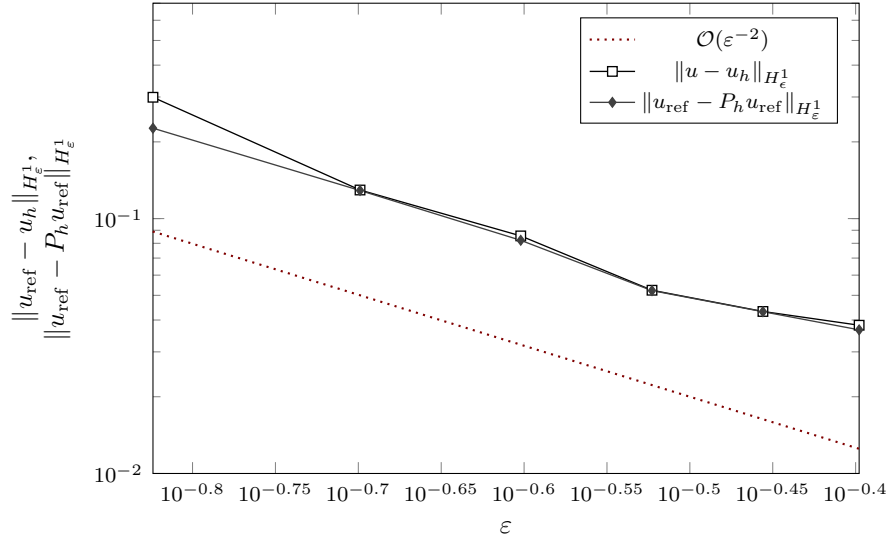

\end{document}